%% file: main.tex
\documentclass{article}
\usepackage{graphicx} % Required for inserting images
\usepackage{hyperref}
 \usepackage{bm}
 \usepackage{tikz}
 \usepackage{graphicx}
 \usepackage{float}
\usepackage{amsmath,amssymb,amsthm,siunitx}
\usetikzlibrary{calc,backgrounds}
\usepackage{subcaption}
\usepackage{pgfplots}

\newcommand{\dpartial}[2]{\frac{\partial{#1}}{\partial{#2}}}
\newcommand{\J}{J}
\newcommand{\Jmin}{J_{\text{min}}}

\def\WIi{W^{1,\infty}}
\def\WIIi{W^{2,\infty}}
\def\B{\mathcal{B}}

\newtheorem{theorem}{Theorem}
\newtheorem{lemma}{Lemma}
\newtheorem{remark}{Remark}

\title{Taming Slivers: A Robust TFEM Framework for Reliable Computations on Degenerate Tetrahedral Meshes}

\author{
Antoine Quiriny\thanks{UCLouvain, Institute of Mechanics, Materials and Civil Engineering (iMMC), Louvain-la-Neuve, Belgium}
\and
Jonathan Lambrechts\footnotemark[1]
\and
Nicolas Mo\"es\footnotemark[1]
\and
V\'aclav Ku\v{c}era\thanks{Charles University, Department of Mathematics and Physics, Prague, Czech Republic}
\and
Jean-Fran\c{c}ois Remacle\footnotemark[1]
}

\begin{document}

\maketitle

\input{intro}
\input{slivers}

\input{degenerated}

\input{Isolated_slivers}
\input{tfem}
\input{stresstest}
\input{physics}
\input{conclusion}

\appendix
\renewcommand{\thesection}{A\arabic{section}}
\section*{Appendices}
\addcontentsline{toc}{section}{Appendices}
\input{minmax}

\bibliographystyle{plain}
\bibliography{main}

\end{document}

%% file: intro.tex
\section{Introduction and Motivation}
\label{sec:intro}

Three–dimensional Delaunay triangulation is a fundamental tool for automatic tetrahedral mesh generation.
It guarantees a number of desirable combinatorial and geometric properties, including uniqueness (under general position assumptions), local conformity to the empty-sphere condition, and compatibility with incremental refinement schemes \cite{Shewchuk1998_TetrahedralMeshGeneration}.
However, unlike its two-dimensional counterpart, the 3D Delaunay triangulation \emph{does not} maximize the minimum dihedral angle or provide any general bound on element shape.
As a consequence, even a perfectly regular distribution of points can yield extremely distorted tetrahedra, known as \emph{slivers}.

Slivers are traditionally regarded as the sworn enemies of numerical simulation.
They are believed to corrupt interpolation accuracy, to deteriorate the conditioning of the discrete system, and to threaten the stability of explicit time‐integration schemes.

Ultimately, despite the wide range of available techniques (see \S \ref{sec:slivers} for a concise review), complete elimination of slivers in general geometries remains impossible.
In practice, more than 90\% of slivers can be removed very efficiently through local topological modifications such as those described in \cite{FreitagOllivierGooch1997_MeshImprovement}.
However, attempting to eliminate every single sliver typically requires runtimes far exceeding those needed to generate the initial mesh itself, with no guarantee of success or even of preserving mesh validity.

In our view, it is neither realistic nor desirable to place the full responsibility for numerical robustness on the mesh generator alone.
A finite element mesh containing several million tetrahedra should not be rejected by a solver merely because it includes a few dozen slivers.
We need more tolerant numerical simulation technologies capable of handling nearly‐degenerate or zero‐volume elements.

Recent research efforts have focused on increasing the resilience of the finite element method to severely distorted and even degenerate tetrahedral elements.  Among these approaches, the Tempered Finite Element Method (TFEM) has shown that accurate and stable solutions can still be obtained in the presence of highly ill-shaped elements.
In the seminal paper on the Tempered Finite Element Method~\cite{quiriny2024tempered}, the authors developed theoretical foundations \emph{showing that finite element discretizations can remain valid on a much broader range of meshes than traditionally assumed.}
In particular, isolated zero- or near-zero-size elements do not cause any specific numerical difficulty, and even in the presence of continuous bands of degenerate elements, the finite element solution can still achieve optimal convergence rates.

In their first paper on TFEM, the authors focused primarily on the theoretical aspects of the approach.  The numerical demonstrations were essentially two-dimensional and involved relatively simple problems.  In the present work, we extend TFEM to far more demanding physical settings: incompressible fluid flows, Cahn-Hilliard phase-field problems, wave propagation, and vibro-acoustic fluid-structure interactions.

The paper is organized as follows.
In \S\ref{sec:slivers}, we provide a concise overview of the state of the art in sliver elimination and present examples of realistic geometries, including the typical number of slivers observed and the efficiency of the algorithms implemented in \texttt{gmsh} \cite{geuzaine2009gmsh} for handling them — together with the amount of dedicated code required, which is by no means negligible.

% Subsequently, we demonstrate in \S\ref{sec:tfem} how TFEM naturally handles slivers.
% We show (i) that TFEM removes all potential issues associated with degenerate elements, such as locking or divisions by zero, and (ii) that it does not suffer from ill-conditioning of the resulting linear systems.
% Concerning time integration, we employ an implicit scheme in this work.
% While explicit formulations could also be stabilized through IMEX-type schemes, such developments lie beyond the scope of the present paper.
Subsequently, we demonstrate in \S\ref{sec:tfem} how TFEM naturally handles slivers.
We show that TFEM eliminates the numerical difficulties traditionally associated
with degenerate elements, including locking phenomena and divisions by zero.
The effect of slivers on the conditioning of the resulting linear systems is not
addressed here, as this aspect has already been analyzed in the seminal TFEM
work, where appropriate preconditioning strategies were shown to provide robust
solver performance despite the presence of highly degenerate elements.
For time integration, we employ an implicit scheme in this work. While
explicit formulations could also be stabilized through IMEX-type schemes \cite{ASCHER1997151}, such
developments lie beyond the scope of the present paper.

The remainder of the paper is devoted to a series of application-driven numerical experiments designed to assess the performance of TFEM across a broad spectrum of physical regimes. In \S\ref{sec:physics}, we consider progressively more complex problems, including incompressible flows, phase-field dynamics governed by the Cahn–Hilliard equations, transient wave propagation, and coupled vibro-acoustic fluid–structure interactions. These examples have been deliberately selected to span diffusive, advective, and hyperbolic behaviors, as well as multiphysics coupling across interfaces. For each case, we emphasize configurations in which sliver elements are known to severely impair standard finite element discretizations. The objective is not only to demonstrate the robustness of TFEM in the presence of highly distorted meshes, but also to show that it consistently delivers accurate and physically meaningful solutions without requiring mesh preprocessing or ad hoc stabilization.

%nico add ref for IMEX
%nico l'intro se lit super bien. 

%% file: slivers.tex
\section{Slivers}
\label{sec:slivers}
Sliver elements are a well-known geometric artifact of three-dimensional Delaunay tetrahedrizations and have been the subject of extensive research in computational geometry and mesh generation. Understanding their origin requires revisiting some fundamental properties of Delaunay tessellations and the differences between their behavior in two and three dimensions. This section therefore recalls the basic definitions of Delaunay tessellations, explains why slivers cannot arise in two-dimensional Delaunay triangulations but may naturally appear in three-dimensional tetrahedrizations, and provides a concise review of the principal techniques that have been developed to suppress or eliminate them.
\subsection{Delaunay tessellation in dimension $d$}
A \emph{Delaunay tessellation} of a finite point set 
$P \subset \mathbb{R}^d$ is a simplicial complex $\mathcal{T}$ whose 
vertices are the points of $P$ and that satisfies the \emph{empty circumsphere
condition}: for every $d$-simplex $\sigma \in \mathcal{T}$, there exists a 
unique 
%nico es-tu du "unique" ? Tu donnais des cas vendredi de sphere circonscrite non unique.
$d$-dimensional sphere $S(\sigma)$ passing through all vertices of 
$\sigma$, and the interior of $S(\sigma)$ contains no point of $P$. 
Among all possible triangulations of $P$, the Delaunay tessellation is the one 
in which every simplex satisfies this empty circumsphere property.

\subsection{Delaunay triangulations in dimension two}

In two dimensions, Delaunay triangulations are endowed with a 
fundamental property called the \emph{min--max angle property} 
(a concise proof is given in Appendix \ref{annex:1} for completeness). 
%nico-done explique que tu mets la concise proof car elle est est difficile à trouver dans le littérature.
This property prevents Delaunay triangulations from containing triangles with extremely small or extremely
large angles, since any locally non-optimal configuration would violate the 
empty--circle condition. As a consequence, degenerate elements -- 
triangles with vanishing angles -- cannot appear in a proper two-dimensional 
Delaunay triangulation.

\subsection{Slivers in Delaunay tetrahedrizations}

A tetrahedron $K \subset \mathbb{R}^3$ with vertices 
$\{\bm{x}_1,\bm{x}_2,\bm{x}_3,\bm{x}_4\}$ is called a \emph{sliver} if its four vertices lie almost on a
%nico-Done lie almost on a common circle(or sphere) is  %not clear
%nico on pourrait définir le critère comme le volume %de l'élément divisé par la racine carrée du produit %des 6 longueurs d'arêtes. Cela donne un critère %sans unité. On peut aussi considéré le volume %divisé 
%par la racine 3/8 du produit des aires de faces
% le volume étant le produit base hauteur, on peut aussi faire intervenir les hauteur. Genre critère min hauteur sur max edge.
%nico un vignoble près de Namur s'appele le dièdre noir en référence à un mont local.
common circle so that $K$ has:
\[
\operatorname{vol}(K) = \frac{1}{6} |\det(\bm{x}_2-\bm{x}_1,\, \bm{x}_3-\bm{x}_1,\, \bm{x}_4-\bm{x}_1)| 
\approx 0,
\]
whereas all edge lengths satisfy
\[
\|\bm{x}_i - \bm{x}_j\| \approx O(1) \quad \text{for all } i \neq j.
\]
In this situation, all dihedral angles of $K$ are close to either
$0$ or $\pi$, and the tetrahedron is nearly flat despite having no small edges/faces.

Slivers can appear in Delaunay tetrahedrizations because the empty--circumsphere condition does not prevent four points from lying almost on the same circle, thus forming a sliver. Even well-spaced point sets may produce tetrahedra with arbitrarily poor dihedral angles and volumes, as shown in Figure \ref{fig:many_slivers} for a regular cube. Unlike elongated "needles", a sliver does not have a large circumradius: its circumscribed sphere can be quite small, comparable to that of well-shaped tetrahedra. The degeneracy arises because the sphere's center happens to lie very close to the plane of the four vertices, so the tetrahedron has almost zero height even though all edges are short and nearly equal. As a consequence, geometric quality measures based solely on edge lengths may fail to identify such elements. A more discriminating indicator is provided by the ratio between the inscribed and circumscribed sphere radii. Indeed, the \emph{aspect ratio} $\gamma$ of a tetrahedron is defined as
\[
\gamma = {3} \frac{r_{\mathrm{in}}}{r_{\mathrm{circ}}}.
\]
where $r_in$ and $r_circ$ denote the radii of the inscribed and circumscribed spheres, respectively. The normalization factor is chosen so that $\gamma \in [0,1]$, with $\gamma=1$ for a regular tetrahedron. For a sliver, the near-coplanarity of the vertices causes the tetrahedral volume, and hence $r_in$, to become very small, while $r_circ$ remains of the same order as for well-shaped tetrahedra. Consequently, $\lambda$ approaches zero, correctly revealing the poor element quality despite the absence of excessively small edges.
Such elements are perfectly valid Delaunay tetrahedra—since no vertex lies inside their circumscribed sphere—but they can have dihedral angles arbitrarily close to $0$ or $\pi$.  
In summary, slivers are an intrinsic geometric artifact of three-dimensional Delaunay tetrahedrizations: the Delaunay property alone does not enforce good element shape in dimension three. An example of a sliver in a Delaunay tetrahedrization is also shown in Figure \ref{fig:0}.
\input{many}
\begin{figure}
    \centering
    \includegraphics[width=0.85\linewidth]{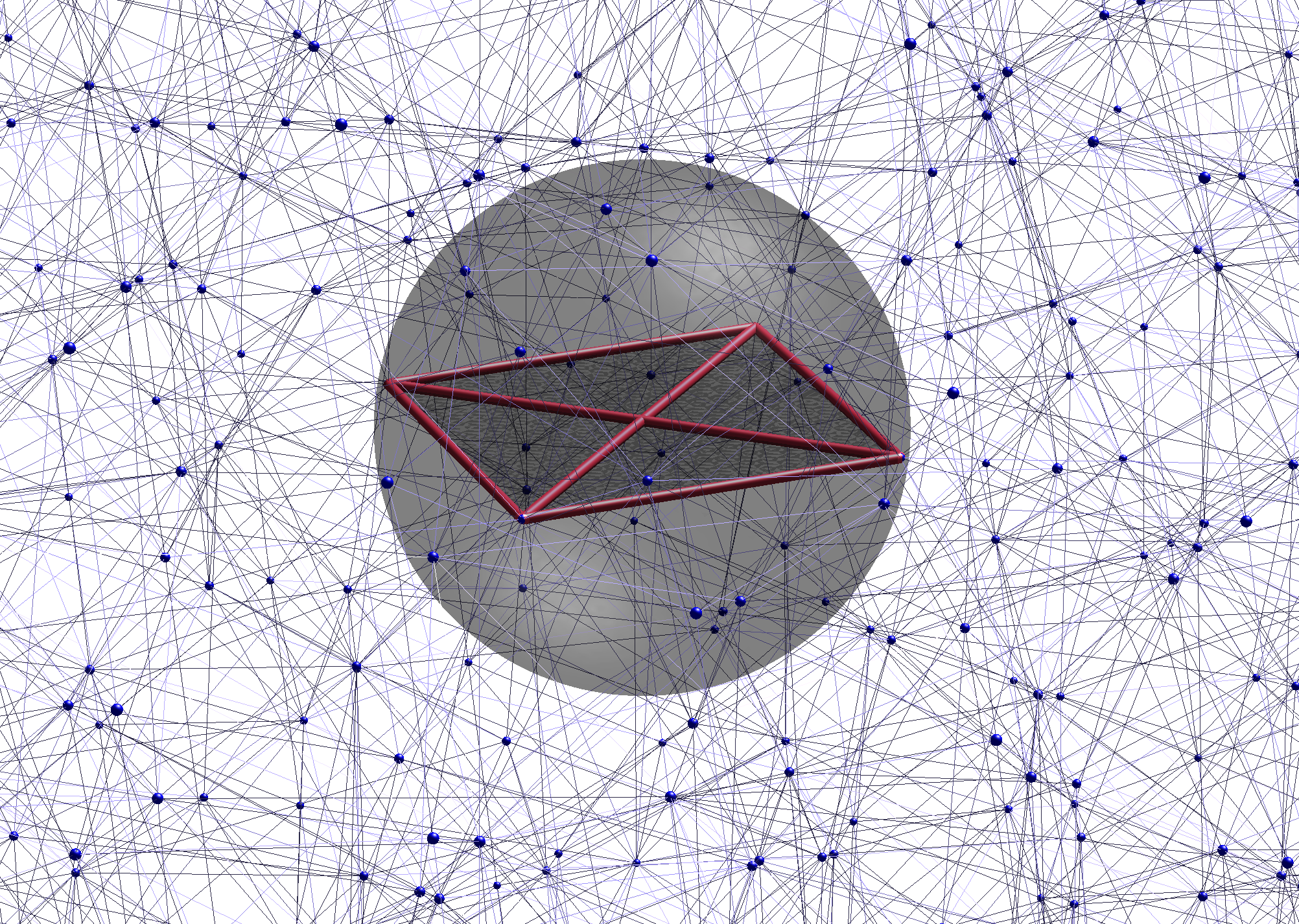}
    \caption{A sliver in a Delaunay tetrahedrization and its circumsphere. Despite the poor quality of the tetrahedron, no mesh vertex lies inside the sphere.}
    %nico - Done: add "The circumsphere does not contain any node in its interior." 
    \label{fig:0}
\end{figure}

\subsection{Sliver elimination -- A concise review}
Over the years, numerous strategies have been developed to mitigate or eliminate slivers.  
They can be broadly classified into three complementary families:  
(i)~methods that modify the geometry of the triangulation itself through weighting or refinement,  
(ii)~methods that alter the local connectivity or vertex placement through topological and optimization-based operations,  
and (iii)~variational or energy-minimization approaches that directly penalize poor element shapes.

%\subsection{Weighted Delaunay and Refinement-Based Methods}

The first family of approaches modifies the construction of the triangulation so that slivers are no longer Delaunay with respect to a suitably chosen metric.  
The most influential among them is the \emph{sliver exudation} technique introduced by Cheng, Dey, Edelsbrunner, Facello, and Teng~\cite{ChengDeyEdelsbrunnerFacelloTeng2000}.  
This method associates a positive weight \(w_i\) with each vertex and constructs a weighted (regular) triangulation instead of the ordinary Delaunay triangulation.  
By iteratively adjusting these weights -- a process known as "pumping" -- one can remove tetrahedra that correspond to poor-quality configurations.
The approach provides theoretical guarantees that all remaining elements will have bounded aspect ratios under certain sampling conditions. 
%nico: sampling ? AQ: sampling des noeuds du maillage comme j'ai compris
Refinement-based algorithms such as those of Shewchuk~\cite{Shewchuk1998_TetrahedralMeshGeneration} or later weighted-Delaunay refinement schemes~\cite{ChengDeyRamosRay2007} extend this concept to boundary-constrained and adaptive contexts.  
The main limitations of these methods are their complexity and the potential proliferation of additional vertices, which may increase mesh size and reduce geometric fidelity on curved boundaries.

%\subsection{Local Topological and Optimization-Based Methods}

A second, very practical class of algorithms improves mesh quality after the initial triangulation by locally modifying connectivity and/or vertex positions.  
Classical operations include face and edge flips, vertex smoothing, and optimization-based relocation.  
Freitag and Ollivier-Gooch~\cite{FreitagOllivierGooch1997_MeshImprovement} and later Freitag and Knupp~\cite{FreitagKnupp2002} formulated mesh improvement as a local optimization of quality metrics such as the element condition number or the mean ratio.  
Klingner~\cite{Klingner2008_AggressiveTetrahedralMeshImprovement} proposed an aggressive sequence of flips and optimization steps that significantly reduces the number of poor elements in practical meshes.
While such local improvements can remove many suboptimal tetrahedra, they often fail to eliminate the worst slivers because the configuration space of flips is limited: only a few connectivity changes are allowed at a time, and the energy landscape is highly non-convex.  
To overcome this limitation, more general reconnection schemes have been proposed.  
Among them, the \emph{Small Polyhedron Reconnection} (SPR) operation introduced by Liu and Sun~\cite{LiuSun2006_SPR} is particularly notable.  
SPR extracts a small polyhedral cavity surrounding low-quality elements and exhaustively enumerates all possible tetrahedralizations of that cavity, selecting the one that maximizes a chosen quality metric.  
Later implementations by Marot, Verhetsel, and Remacle~\cite{MarotVerhetselRemacle2019_SPR} drastically accelerated this enumeration process, making SPR practical for large-scale meshes.  
Related techniques such as the Multi-Face Reconstruction (MFRC) operator~\cite{MaWang2021_MFRC} extend this idea by reconstructing larger cavities in a single step, improving the worst elements without inserting additional vertices. In
the same spirit, frameworks such as HXT~\cite{MarotPelerinRemacle2019} integrate these reconnection operations into high-performance meshing pipelines, allowing hybrid workflows that combine Delaunay refinement, flipping, and cavity-based optimization.

\input{aspect_indu_mesh}
%\subsection{Energy-Based and Variational Optimization Approaches}

A third family formulates sliver suppression as a global or semi-global optimization problem.  
Here, the mesh is treated as a set of vertex positions \(\{x_i\}\) minimizing an energy functional that measures element distortion or shape irregularity.  
Typical choices of energy include the condition number, the mean-ratio metric, or physically inspired shape-matching potentials.  
Ni~\cite{Ni2017_SliverSuppressing} introduced a gradient-based shape-matching energy that specifically penalizes sliver-like configurations while maintaining overall mesh regularity.  
More recent works minimize radius-ratio or squared-volume energies to eliminate degenerate tetrahedra~\cite{WangChenWei2025_RadiusRatioSliverElimination,Yu2024_WSVM}.  
These methods have the advantage of producing smooth, high-quality meshes without introducing new vertices or modifying topology, but they require non-convex optimization and can be both slow and sensitive to initialization.  
Combining them with topological operations (e.g. SPR or flipping) often yields the best results.

\subsection{Sliver elimination in {\texttt{gmsh}}}
Since its early versions, \texttt{gmsh} has implemented the approach proposed in \cite{FreitagOllivierGooch1997_MeshImprovement}, one of the first systematic frameworks for tetrahedral mesh improvement combining local topological flips and optimization-based smoothing. Since 2019, a more aggressive mesh improvement technique has been added that uses Small Polyhedron Reconnection (SPR, \cite{MarotVerhetselRemacle2019_SPR}). This latter method requires computing, on the fly, all possible tetrahedralizations of a local cavity.
The SPR algorithm starts from small cavities and progressively enlarges them if no configuration leads to an improvement in mesh quality.
Given the highly combinatorial nature of the problem, it is clear that arbitrarily increasing the cavity size rapidly results in long — or even prohibitively long — computation times.

In smooth geometries, the only tetrahedra that have a very poor aspect ratio are slivers. Figure \ref{fig:aspect_ratio_indu} shows examples of aspect-ratio statistics for three simple geometries that contain about two million tetrahedra each. By applying Gmsh's full suite of mesh-improvement tools with a threshold \(\gamma = 0.4\), we are able to eliminate all slivers.
Nevertheless, the resulting meshes do not meet the \(\gamma = 0.4\) target for all tetrahedra, which is entirely expected: to the best of our knowledge, no algorithm guarantees such a quality bound even for very simple geometries. More precisely, among all existing approaches, the \emph{sliver exudation} algorithm of Cheng \emph{et al.}~\cite{ChengDeyEdelsbrunnerFacelloTeng2000}
remains the only method that provides a formal guarantee of sliver-free tetrahedral meshes.  
It ensures the existence of a positive lower bound on the aspect ratio
$\gamma_{\min} > 0$
independent of the geometry or point distribution, thereby excluding tetrahedra of vanishing volume and arbitrarily small dihedral angles.

%% file: many.tex
% Preamble:
% \usepackage{tikz}
% \usepackage{graphicx}

\begin{figure}[h!]
\centering
\begin{tikzpicture}

    \begin{axis}[
            xlabel={Aspect-ratio $\gamma$},
            ylabel={Density},
            legend style={at={(0.65,0.4)},anchor=west},
            xmin=0,
            xmax=1,
            grid=none,
            width=\textwidth,
            height=0.8\textwidth,
            y label style={at={(axis description cs:0.22,0.85)},rotate=-90,anchor=south},
            legend style={draw=none}
        ]
        %regular
        \addplot[
            color=blue,
            line width = 1.5pt
            ]
            coordinates {
            (8.333333e-03, 8.939499e+00) (2.500000e-02, 4.015994e+00) (4.166667e-02, 1.846314e+00) (5.833333e-02, 4.589708e-01) (7.500000e-02, 1.043115e-02) (9.166667e-02, 0.000000e+00) (1.083333e-01, 0.000000e+00) (1.250000e-01, 0.000000e+00) (1.416667e-01, 0.000000e+00) (1.583333e-01, 0.000000e+00) (1.750000e-01, 0.000000e+00) (1.916667e-01, 0.000000e+00) (2.083333e-01, 0.000000e+00) (2.250000e-01, 0.000000e+00) (2.416667e-01, 0.000000e+00) (2.583333e-01, 0.000000e+00) (2.750000e-01, 0.000000e+00) (2.916667e-01, 0.000000e+00) (3.083333e-01, 0.000000e+00) (3.250000e-01, 0.000000e+00) (3.416667e-01, 0.000000e+00) (3.583333e-01, 0.000000e+00) (3.750000e-01, 0.000000e+00) (3.916667e-01, 0.000000e+00) (4.083333e-01, 0.000000e+00) (4.250000e-01, 0.000000e+00) (4.416667e-01, 0.000000e+00) (4.583333e-01, 0.000000e+00) (4.750000e-01, 0.000000e+00) (4.916667e-01, 0.000000e+00) (5.083333e-01, 0.000000e+00) (5.250000e-01, 0.000000e+00) (5.416667e-01, 0.000000e+00) (5.583333e-01, 0.000000e+00) (5.750000e-01, 2.086231e-02) (5.916667e-01, 3.442281e-01) (6.083333e-01, 2.096662e+00) (6.250000e-01, 4.725313e+00) (6.416667e-01, 2.805981e+00) (6.583333e-01, 3.233658e-01) (6.750000e-01, 0.000000e+00) (6.916667e-01, 5.215577e-02) (7.083333e-01, 3.588317e+00) (7.250000e-01, 1.375869e+01) (7.416667e-01, 1.022253e+01) (7.583333e-01, 5.319889e+00) (7.750000e-01, 5.632823e-01) (7.916667e-01, 1.043115e-02) (8.083333e-01, 0.000000e+00) (8.250000e-01, 0.000000e+00) (8.416667e-01, 0.000000e+00) (8.583333e-01, 0.000000e+00) (8.750000e-01, 0.000000e+00) (8.916667e-01, 0.000000e+00) (9.083333e-01, 0.000000e+00) (9.250000e-01, 0.000000e+00) (9.416667e-01, 0.000000e+00) (9.583333e-01, 0.000000e+00) (9.750000e-01, 0.000000e+00) (9.916667e-01, 8.970793e-01) 
            };

        \node [align=center] at (rel axis cs: 
            0.35, 0.6) {\includegraphics[width=5cm, trim={1000 200 400 1200},clip]{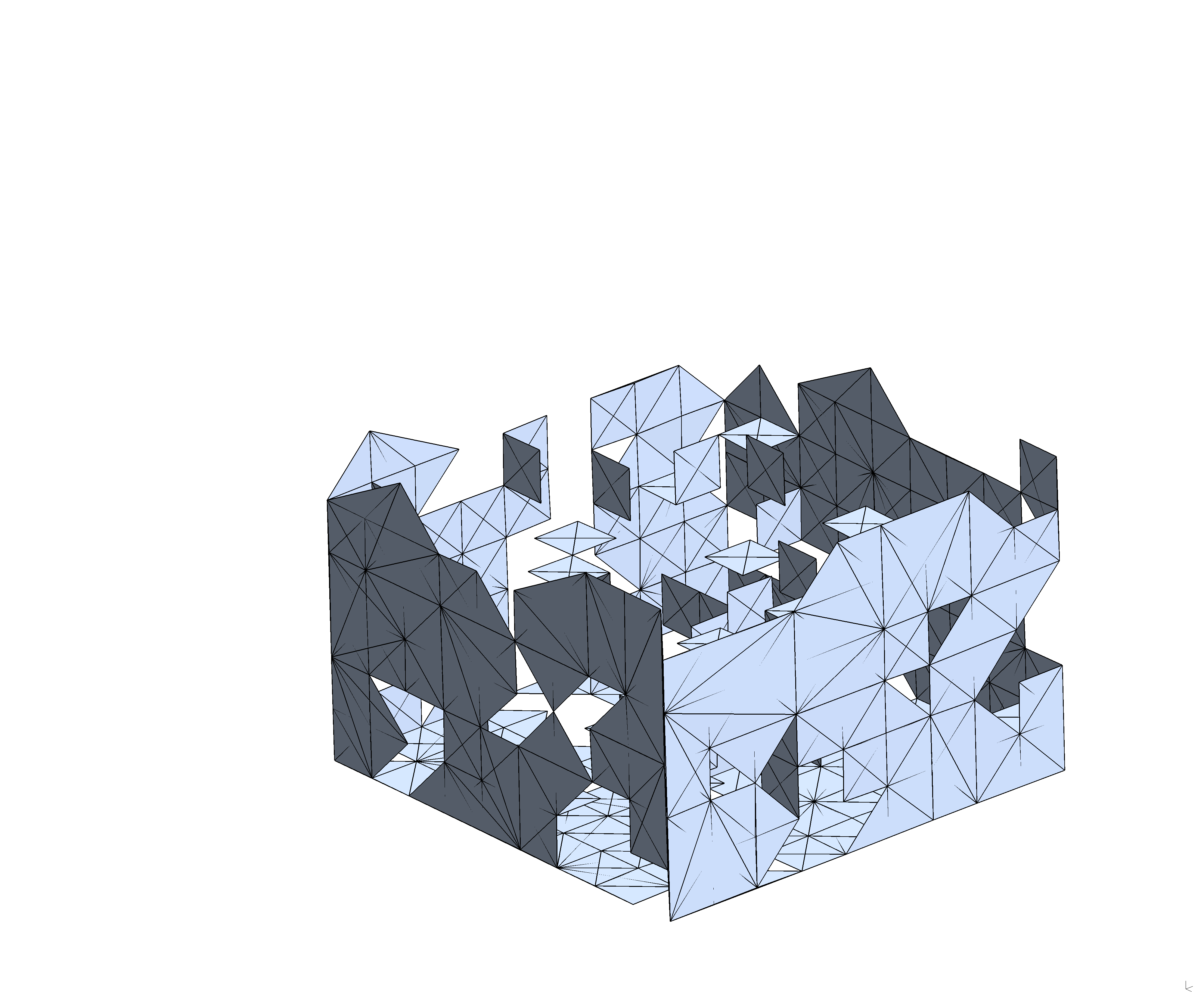}};

    \end{axis}

\end{tikzpicture}
\caption{Aspect-ratio distribution for a Delaunay tetrahedrization of a uniform grid. Only slivers are shown.}
\label{fig:many_slivers}
\end{figure}

%% file: aspect_indu_mesh.tex
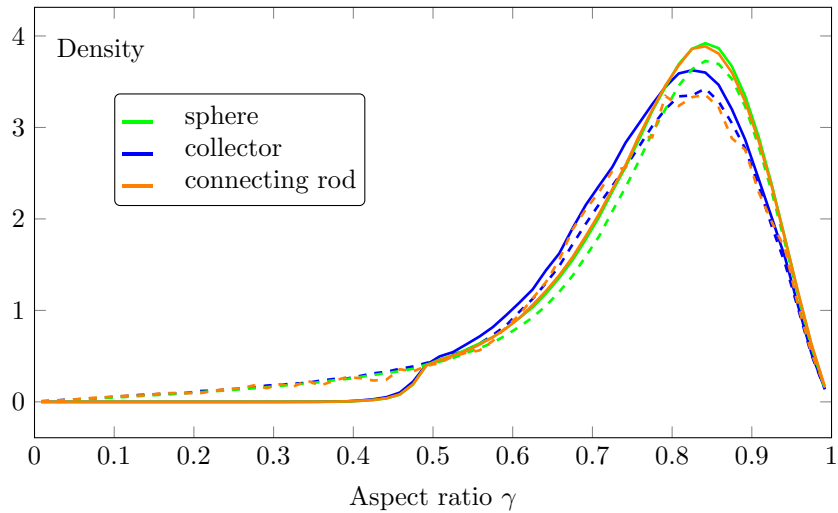
\begin{figure}
    \centering
\begin{tikzpicture}
        \begin{axis}[
            xlabel={Aspect ratio $\gamma$},
            ylabel={Density},
            legend style={at={(0.65,0.4)},anchor=west},
            xmin=0, xmax=1,
            width=\textwidth,
            height=0.6\textwidth,
            x label style={at={(axis description cs:0.5,0.0)},anchor=north},
            y label style={at={(axis description cs:0.2,0.85)},rotate=-90,anchor=south},
        ]
        % collector no opti
        \addplot[
            color=blue,
            dashed, 
            line width=1pt
            ]
            coordinates {
                (8.333333e-03, 5.039424e-03) (2.500000e-02, 1.441098e-02) (4.166667e-02, 2.448983e-02) (5.833333e-02, 2.935243e-02) (7.500000e-02, 3.819353e-02) (9.166667e-02, 5.030583e-02) (1.083333e-01, 5.879328e-02) (1.250000e-01, 7.214333e-02) (1.416667e-01, 7.992350e-02) (1.583333e-01, 8.602385e-02) (1.750000e-01, 9.203579e-02) (1.916667e-01, 9.866662e-02) (2.083333e-01, 1.188243e-01) (2.250000e-01, 1.232449e-01) (2.416667e-01, 1.350919e-01) (2.583333e-01, 1.449055e-01) (2.750000e-01, 1.579904e-01) (2.916667e-01, 1.785901e-01) (3.083333e-01, 1.877848e-01) (3.250000e-01, 1.987478e-01) (3.416667e-01, 2.083846e-01) (3.583333e-01, 2.310178e-01) (3.750000e-01, 2.497609e-01) (3.916667e-01, 2.548003e-01) (4.083333e-01, 2.826498e-01) (4.250000e-01, 3.120022e-01) (4.416667e-01, 3.284467e-01) (4.583333e-01, 3.623965e-01) (4.750000e-01, 3.866211e-01) (4.916667e-01, 4.256987e-01) (5.083333e-01, 4.567309e-01) (5.250000e-01, 5.166736e-01) (5.416667e-01, 5.665373e-01) (5.583333e-01, 6.351442e-01) (5.750000e-01, 7.300976e-01) (5.916667e-01, 8.357487e-01) (6.083333e-01, 9.849863e-01) (6.250000e-01, 1.125383e+00) (6.416667e-01, 1.298580e+00) (6.583333e-01, 1.484066e+00) (6.750000e-01, 1.720477e+00) (6.916667e-01, 1.946986e+00) (7.083333e-01, 2.159879e+00) (7.250000e-01, 2.361898e+00) (7.416667e-01, 2.567100e+00) (7.583333e-01, 2.788746e+00) (7.750000e-01, 3.012957e+00) (7.916667e-01, 3.201184e+00) (8.083333e-01, 3.339900e+00) (8.250000e-01, 3.349714e+00) (8.416667e-01, 3.425394e+00) (8.583333e-01, 3.283759e+00) (8.750000e-01, 3.046288e+00) (8.916667e-01, 2.769384e+00) (9.083333e-01, 2.368176e+00) (9.250000e-01, 1.918960e+00) (9.416667e-01, 1.511827e+00) (9.583333e-01, 1.010360e+00) (9.750000e-01, 5.096007e-01) (9.916667e-01, 1.371254e-01)
            };

        % collector opti
        \addplot[
            color=blue,
            line width=1pt
            ]
            coordinates {
                (8.333333e-03, 0.000000e+00) (2.500000e-02, 0.000000e+00) (4.166667e-02, 0.000000e+00) (5.833333e-02, 0.000000e+00) (7.500000e-02, 0.000000e+00) (9.166667e-02, 0.000000e+00) (1.083333e-01, 0.000000e+00) (1.250000e-01, 0.000000e+00) (1.416667e-01, 0.000000e+00) (1.583333e-01, 0.000000e+00) (1.750000e-01, 0.000000e+00) (1.916667e-01, 0.000000e+00) (2.083333e-01, 0.000000e+00) (2.250000e-01, 0.000000e+00) (2.416667e-01, 0.000000e+00) (2.583333e-01, 0.000000e+00) (2.750000e-01, 9.321214e-05) (2.916667e-01, 0.000000e+00) (3.083333e-01, 0.000000e+00) (3.250000e-01, 6.524850e-04) (3.416667e-01, 9.321214e-04) (3.583333e-01, 1.398182e-03) (3.750000e-01, 3.914910e-03) (3.916667e-01, 6.058789e-03) (4.083333e-01, 1.510037e-02) (4.250000e-01, 2.861613e-02) (4.416667e-01, 5.173274e-02) (4.583333e-01, 1.014148e-01) (4.750000e-01, 2.217517e-01) (4.916667e-01, 4.069642e-01) (5.083333e-01, 4.969139e-01) (5.250000e-01, 5.438928e-01) (5.416667e-01, 6.281566e-01) (5.583333e-01, 7.146575e-01) (5.750000e-01, 8.170044e-01) (5.916667e-01, 9.468489e-01) (6.083333e-01, 1.084244e+00) (6.250000e-01, 1.229188e+00) (6.416667e-01, 1.437984e+00) (6.583333e-01, 1.621891e+00) (6.750000e-01, 1.896028e+00) (6.916667e-01, 2.152641e+00) (7.083333e-01, 2.362555e+00) (7.250000e-01, 2.567435e+00) (7.416667e-01, 2.834488e+00) (7.583333e-01, 3.042724e+00) (7.750000e-01, 3.252265e+00) (7.916667e-01, 3.443909e+00) (8.083333e-01, 3.590159e+00) (8.250000e-01, 3.625206e+00) (8.416667e-01, 3.600039e+00) (8.583333e-01, 3.466000e+00) (8.750000e-01, 3.198388e+00) (8.916667e-01, 2.866926e+00) (9.083333e-01, 2.442251e+00) (9.250000e-01, 2.014594e+00) (9.416667e-01, 1.577988e+00) (9.583333e-01, 1.041086e+00) (9.750000e-01, 5.244115e-01) (9.916667e-01, 1.414960e-01)
            };
            
        %sph no opti
        \addplot[
            color=green,
            dashed,
            line width=1pt
            ]
            coordinates {
                (8.333333e-03, 3.733649e-03) (2.500000e-02, 1.247444e-02) (4.166667e-02, 2.083897e-02) (5.833333e-02, 2.749586e-02) (7.500000e-02, 3.603405e-02) (9.166667e-02, 4.460119e-02) (1.083333e-01, 5.070816e-02) (1.250000e-01, 5.968050e-02) (1.416667e-01, 6.905804e-02) (1.583333e-01, 7.684371e-02) (1.750000e-01, 8.471621e-02) (1.916667e-01, 9.704593e-02) (2.083333e-01, 1.061630e-01) (2.250000e-01, 1.166404e-01) (2.416667e-01, 1.239919e-01) (2.583333e-01, 1.373925e-01) (2.750000e-01, 1.476962e-01) (2.916667e-01, 1.611258e-01) (3.083333e-01, 1.751052e-01) (3.250000e-01, 1.885637e-01) (3.416667e-01, 2.006619e-01) (3.583333e-01, 2.162043e-01) (3.750000e-01, 2.352488e-01) (3.916667e-01, 2.502703e-01) (4.083333e-01, 2.719775e-01) (4.250000e-01, 2.936558e-01) (4.416667e-01, 3.153052e-01) (4.583333e-01, 3.357969e-01) (4.750000e-01, 3.619903e-01) (4.916667e-01, 4.078650e-01) (5.083333e-01, 4.390945e-01) (5.250000e-01, 4.789780e-01) (5.416667e-01, 5.336224e-01) (5.583333e-01, 5.875722e-01) (5.750000e-01, 6.459502e-01) (5.916667e-01, 7.296824e-01) (6.083333e-01, 8.213160e-01) (6.250000e-01, 9.227033e-01) (6.416667e-01, 1.050487e+00) (6.583333e-01, 1.201482e+00) (6.750000e-01, 1.377311e+00) (6.916667e-01, 1.590795e+00) (7.083333e-01, 1.808967e+00) (7.250000e-01, 2.083926e+00) (7.416667e-01, 2.348176e+00) (7.583333e-01, 2.632107e+00) (7.750000e-01, 2.939395e+00) (7.916667e-01, 3.202024e+00) (8.083333e-01, 3.456578e+00) (8.250000e-01, 3.637008e+00) (8.416667e-01, 3.726876e+00) (8.583333e-01, 3.693563e+00) (8.750000e-01, 3.529803e+00) (8.916667e-01, 3.229056e+00) (9.083333e-01, 2.800034e+00) (9.250000e-01, 2.293821e+00) (9.416667e-01, 1.731284e+00) (9.583333e-01, 1.140007e+00) (9.750000e-01, 5.920294e-01) (9.916667e-01, 1.474647e-01)
            };

        %sph opti
        \addplot[
            color=green,
            line width=1pt
            ]
            coordinates {
                 (8.333333e-03, 0.000000e+00) (2.500000e-02, 0.000000e+00) (4.166667e-02, 0.000000e+00) (5.833333e-02, 0.000000e+00) (7.500000e-02, 0.000000e+00) (9.166667e-02, 0.000000e+00) (1.083333e-01, 0.000000e+00) (1.250000e-01, 0.000000e+00) (1.416667e-01, 0.000000e+00) (1.583333e-01, 0.000000e+00) (1.750000e-01, 0.000000e+00) (1.916667e-01, 0.000000e+00) (2.083333e-01, 0.000000e+00) (2.250000e-01, 0.000000e+00) (2.416667e-01, 0.000000e+00) (2.583333e-01, 0.000000e+00) (2.750000e-01, 6.085943e-05) (2.916667e-01, 1.217189e-04) (3.083333e-01, 3.651566e-04) (3.250000e-01, 5.173051e-04) (3.416667e-01, 7.911725e-04) (3.583333e-01, 8.824617e-04) (3.750000e-01, 1.551915e-03) (3.916667e-01, 4.868754e-03) (4.083333e-01, 1.055911e-02) (4.250000e-01, 2.154424e-02) (4.416667e-01, 3.739812e-02) (4.583333e-01, 7.869124e-02) (4.750000e-01, 1.906422e-01) (4.916667e-01, 4.018548e-01) (5.083333e-01, 4.627142e-01) (5.250000e-01, 5.085718e-01) (5.416667e-01, 5.708614e-01) (5.583333e-01, 6.389023e-01) (5.750000e-01, 7.068822e-01) (5.916667e-01, 8.021272e-01) (6.083333e-01, 9.159648e-01) (6.250000e-01, 1.033819e+00) (6.416667e-01, 1.183564e+00) (6.583333e-01, 1.355552e+00) (6.750000e-01, 1.548325e+00) (6.916667e-01, 1.776182e+00) (7.083333e-01, 2.019498e+00) (7.250000e-01, 2.299360e+00) (7.416667e-01, 2.570063e+00) (7.583333e-01, 2.864988e+00) (7.750000e-01, 3.176527e+00) (7.916667e-01, 3.442939e+00) (8.083333e-01, 3.692494e+00) (8.250000e-01, 3.859279e+00) (8.416667e-01, 3.920929e+00) (8.583333e-01, 3.866186e+00) (8.750000e-01, 3.669093e+00) (8.916667e-01, 3.343678e+00) (9.083333e-01, 2.905246e+00) (9.250000e-01, 2.376469e+00) (9.416667e-01, 1.789997e+00) (9.583333e-01, 1.179882e+00) (9.750000e-01, 6.164147e-01) (9.916667e-01, 1.536396e-01)
            };

        % cr no opti
        \addplot[
            color=orange,
            dashed,
            line width=1pt
            ]
            coordinates {
                (8.333333e-03, 8.806909e-03) (2.500000e-02, 1.321036e-02) (4.166667e-02, 2.201727e-02) (5.833333e-02, 3.375982e-02) (7.500000e-02, 4.550236e-02) (9.166667e-02, 5.284145e-02) (1.083333e-01, 6.458400e-02) (1.250000e-01, 6.898745e-02) (1.416667e-01, 7.045527e-02) (1.583333e-01, 9.394036e-02) (1.750000e-01, 9.687599e-02) (1.916667e-01, 9.540818e-02) (2.083333e-01, 9.540818e-02) (2.250000e-01, 1.218289e-01) (2.416667e-01, 1.409105e-01) (2.583333e-01, 1.453140e-01) (2.750000e-01, 1.805416e-01) (2.916667e-01, 1.497174e-01) (3.083333e-01, 1.937520e-01) (3.250000e-01, 1.673313e-01) (3.416667e-01, 2.333831e-01) (3.583333e-01, 1.776060e-01) (3.750000e-01, 2.289796e-01) (3.916667e-01, 2.700785e-01) (4.083333e-01, 2.598038e-01) (4.250000e-01, 2.333831e-01) (4.416667e-01, 2.436578e-01) (4.583333e-01, 3.566798e-01) (4.750000e-01, 3.331947e-01) (4.916667e-01, 3.919074e-01) (5.083333e-01, 4.271351e-01) (5.250000e-01, 4.961225e-01) (5.416667e-01, 5.357536e-01) (5.583333e-01, 5.621743e-01) (5.750000e-01, 6.663894e-01) (5.916667e-01, 8.161069e-01) (6.083333e-01, 9.526139e-01) (6.250000e-01, 1.130220e+00) (6.416667e-01, 1.309294e+00) (6.583333e-01, 1.552952e+00) (6.750000e-01, 1.883211e+00) (6.916667e-01, 2.103383e+00) (7.083333e-01, 2.294200e+00) (7.250000e-01, 2.523179e+00) (7.416667e-01, 2.564278e+00) (7.583333e-01, 2.900409e+00) (7.750000e-01, 2.891602e+00) (7.916667e-01, 3.343690e+00) (8.083333e-01, 3.235071e+00) (8.250000e-01, 3.330479e+00) (8.416667e-01, 3.356900e+00) (8.583333e-01, 3.220393e+00) (8.750000e-01, 2.881327e+00) (8.916667e-01, 2.756562e+00) (9.083333e-01, 2.298603e+00) (9.250000e-01, 1.966876e+00) (9.416667e-01, 1.670377e+00) (9.583333e-01, 1.045086e+00) (9.750000e-01, 5.284145e-01) (9.916667e-01, 1.673313e-01)
            };
            
        % cr opti
        \addplot[
            color=orange,
            line width=1pt
            ]
            coordinates {
                (8.333333e-03, 0.000000e+00) (2.500000e-02, 0.000000e+00) (4.166667e-02, 0.000000e+00) (5.833333e-02, 0.000000e+00) (7.500000e-02, 0.000000e+00) (9.166667e-02, 0.000000e+00) (1.083333e-01, 0.000000e+00) (1.250000e-01, 0.000000e+00) (1.416667e-01, 0.000000e+00) (1.583333e-01, 0.000000e+00) (1.750000e-01, 0.000000e+00) (1.916667e-01, 0.000000e+00) (2.083333e-01, 0.000000e+00) (2.250000e-01, 2.717060e-05) (2.416667e-01, 2.717060e-05) (2.583333e-01, 5.434121e-05) (2.750000e-01, 1.358530e-04) (2.916667e-01, 8.151181e-05) (3.083333e-01, 2.445354e-04) (3.250000e-01, 3.803884e-04) (3.416667e-01, 6.520945e-04) (3.583333e-01, 8.694593e-04) (3.750000e-01, 2.200819e-03) (3.916667e-01, 5.597144e-03) (4.083333e-01, 1.152034e-02) (4.250000e-01, 2.032361e-02) (4.416667e-01, 3.994079e-02) (4.583333e-01, 8.172917e-02) (4.750000e-01, 1.935362e-01) (4.916667e-01, 3.962561e-01) (5.083333e-01, 4.566563e-01) (5.250000e-01, 5.056993e-01) (5.416667e-01, 5.568072e-01) (5.583333e-01, 6.251141e-01) (5.750000e-01, 7.079029e-01) (5.916667e-01, 8.042770e-01) (6.083333e-01, 9.176871e-01) (6.250000e-01, 1.060007e+00) (6.416667e-01, 1.212135e+00) (6.583333e-01, 1.383581e+00) (6.750000e-01, 1.582334e+00) (6.916667e-01, 1.812823e+00) (7.083333e-01, 2.055293e+00) (7.250000e-01, 2.325831e+00) (7.416667e-01, 2.590282e+00) (7.583333e-01, 2.904266e+00) (7.750000e-01, 3.192111e+00) (7.916667e-01, 3.461317e+00) (8.083333e-01, 3.678682e+00) (8.250000e-01, 3.860046e+00) (8.416667e-01, 3.886429e+00) (8.583333e-01, 3.807770e+00) (8.750000e-01, 3.600703e+00) (8.916667e-01, 3.281964e+00) (9.083333e-01, 2.859869e+00) (9.250000e-01, 2.355338e+00) (9.416667e-01, 1.778560e+00) (9.583333e-01, 1.190888e+00) (9.750000e-01, 6.277224e-01) (9.916667e-01, 1.643278e-01)
            };

        % --- Position de la légende (facile à modifier) ---
        \def\legendx{0.1}
        \def\legendy{0.8}
        
        \node[
            anchor=north west,
            fill=white,
            draw=black,
            rounded corners=2pt,
            inner sep=3pt
        ] at (rel axis cs:\legendx,\legendy) {
        \begin{tabular}{@{}ll@{}}
        \textcolor{green}{\rule{0.4cm}{1.5pt}} & sphere \\
        \textcolor{blue}{\rule{0.4cm}{1.5pt}} & collector \\
        \textcolor{orange}{\rule{0.4cm}{1.5pt}} & connecting rod
        \end{tabular}
        };

        \end{axis}
    \end{tikzpicture}
    
    \caption{Distribution of aspect ratio in three tetrahedral meshes of smooth/simple geometries. Solid lines are for optimized meshes. A threshold aspect ratio is fixed at $\gamma=0.4$. Yet, gmsh's optimizer is only able to produce meshes with $\gamma > 0.2$.}
    \label{fig:aspect_ratio_indu}
    %nico: why 0.4 then 0.2 ?
\end{figure}

%% file: degenerated.tex
\section{Stiffness matrices of degenerate tetrahedra}
\label{sec:deg}
Edelsbrunner \cite{Edelsbrunner2001} proposed a geometric classification of degenerate tetrahedra based on how they collapse within the six-dimensional space of edge lengths. Degeneracies occur when the volume of the tetrahedron approaches zero, producing two fundamental families: flat and skinny tetrahedra. Flat tetrahedra are nearly coplanar, while skinny tetrahedra are close to a line. This classification provides a conceptual framework for analyzing and avoiding pathological elements in mesh generation and finite-element computations.
\subsection{Flat tetrahedra}
\begin{table}[]
    \centering
    \begin{tabular}{|c|c|c|c|}
    \hline
      \input{sliver}    &  \input{cap} &  \input{spade}&  \input{wedge}\\
    \hline    
$\bm{x}_1=(-h,-h,0)$&
$\bm{x}_1=(-h,0,0)$&
$\bm{x}_1=(0,0,0)$&
$\bm{x}_1=(0,0,0)$\\
$\bm{x}_2=(h,-h,\varepsilon)$&
$\bm{x}_2=(h,0,0)$&
$\bm{x}_2=(0,h,\varepsilon)$&
$\bm{x}_2=(0,h,0)$\\
$\bm{x}_3=(h,h,0)$&
$\bm{x}_3=(0,\sqrt{3}h,0)$&
$\bm{x}_3=(0,2h,0)$&
$\bm{x}_3=(h,0,0)$\\
$\bm{x}_4=(-h,h,\varepsilon)$&
$\bm{x}_4=(0,h \sqrt{3}/3, \varepsilon)$&
$\bm{x}_4=(h,h,0)$&
$\bm{x}_4=(h,0,\varepsilon)$\\
    \hline    
   $\bm{v}_1 = \frac{1}{2} \begin{pmatrix} 1 \\ -1 \\ 1 \\ -1 \end{pmatrix}$
 & 
    $\bm{v}_1 = \frac{1}{2\sqrt{3}} \begin{pmatrix} -1 \\ -1 \\ -1 \\ 3 \end{pmatrix}$
& 
   $\bm{v}_1 = \frac{1}{\sqrt{6}} \begin{pmatrix} 1 \\ -2 \\ 1 \\ 0 \end{pmatrix}$
& 
   $\bm{v}_1 = \frac{1}{\sqrt{2}} \begin{pmatrix} 0 \\ 0 \\ 1 \\ -1 \end{pmatrix}$
\\  
    \hline
    \end{tabular}
    \caption{Flat tetrahedra.}
    \label{tab:flat}
\end{table}
%nico - done: utilisation de la notation v pour deux choses différentes. 
%nico: mettre le numéro des noeuds sur les tets de la figure aiderait.
%nico: pour le sliver peux-tu mettre les 3 premiers noeuds dans le plan x,y ? 
%nico - done: même remarque pour le spade - AQ: c'est déjà le cas 
First, consider a sliver tetrahedron with vertices
\[
\bm{x}_1=(-h,-h,0), \quad
\bm{x}_2=(-h,h,\varepsilon), \quad
\bm{x}_3=(h,h,0), \quad
\bm{x}_4=(h,-h,\varepsilon).
\]
We aim to compute analytically the finite element stiffness matrix
associated with the Laplace operator, using linear (P1) shape functions. Define
\[
\bm{e}_1=\bm{x}_2-\bm{x}_1,\qquad
\bm{e}_2=\bm{x}_3-\bm{x}_1,\qquad
\bm{e}_3=\bm{x}_4-\bm{x}_1,
\]
we can compute the element volume as
\[
6V = \det[\bm{e}_1\;\bm{e}_2\;\bm{e}_3]
     = 8\,\varepsilon\,h^2,
\qquad\Rightarrow\qquad
V = \frac{4}{3}\,\varepsilon\,h^2.
\]
Consider the four linear
shape functions $N_j$, $j=1,\dots,4$, of the tetrahedron, whose gradients are constant
within the element:
\[
\begin{aligned}
\nabla N_1 &= \left(-\tfrac{1}{4h},\;-\tfrac{1}{4h},\;-\tfrac{1}{2\varepsilon}\right)~~~,~~~
\nabla N_2 &= \left(-\tfrac{1}{4h},\;\;\tfrac{1}{4h},\;\;\tfrac{1}{2\varepsilon}\right),\\
\nabla N_3 &= \left(\;\tfrac{1}{4h},\;\;\tfrac{1}{4h},\;-\tfrac{1}{2\varepsilon}\right)~~~,~~~
\nabla N_4 &= \left(\;\tfrac{1}{4h},\;-\tfrac{1}{4h},\;\;\tfrac{1}{2\varepsilon}\right).
\end{aligned}
\]
The stiffness matrix of the Laplacian can be written as
\[
K_{ij} = \int_T \nabla N_i \cdot \nabla N_j \,\mathrm{d}V
        = V\,(\nabla N_i \cdot \nabla N_j).
\]
Using $V=\tfrac{4}{3}\varepsilon h^2$, we obtain
\[
K=
\begin{pmatrix}
\frac{\varepsilon}{6}+\frac{h^2}{3\varepsilon} & -\frac{h^2}{3\varepsilon} & -\frac{\varepsilon}{6}+\frac{h^2}{3\varepsilon} & -\frac{h^2}{3\varepsilon} \\[0.3em]
-\frac{h^2}{3\varepsilon} & \frac{\varepsilon}{6}+\frac{h^2}{3\varepsilon} & -\frac{h^2}{3\varepsilon} & -\frac{\varepsilon}{6}+\frac{h^2}{3\varepsilon} \\[0.3em]
-\frac{\varepsilon}{6}+\frac{h^2}{3\varepsilon} & -\frac{h^2}{3\varepsilon} & \frac{\varepsilon}{6}+\frac{h^2}{3\varepsilon} & -\frac{h^2}{3\varepsilon} \\[0.3em]
-\frac{h^2}{3\varepsilon} & -\frac{\varepsilon}{6}+\frac{h^2}{3\varepsilon} & -\frac{h^2}{3\varepsilon} & \frac{\varepsilon}{6}+\frac{h^2}{3\varepsilon}
\end{pmatrix}.
\]
It is convenient to separate the contributions associated with
the out-of-plane ($z$) direction and the in-plane $(x,y)$ derivatives:
%nico -done: est-ce que ce n'est pas d'abord out-plane et puis in-plane ? - AQ: oui je corrige
\[
K = 
\frac{h^2}{3\varepsilon}
\underbrace{\begin{pmatrix}
1 & -1 & 1 & -1\\
-1& 1 & -1& 1 \\
1 & -1& 1 & -1\\
-1& 1 & -1& 1
\end{pmatrix}}_{K^{(1/\varepsilon)}}
+\frac{\varepsilon}{6}
\underbrace{
\begin{pmatrix}
1 & 0 & -1 & 0\\
0 & 1 & 0 & -1\\
-1& 0 & 1 & 0\\
0 & -1& 0 & 1
\end{pmatrix}}_{K^{(\varepsilon)}}.
\]
In the limit $\varepsilon \rightarrow 0$, the volume of the element tends to $V \rightarrow 0$ and the stiffness matrix is not computable.
Matrix \({K^{(1/\varepsilon)}}\) is a rank-1 matrix with a single non-zero eigenvalue.

\[
\lambda_1 = 4, \qquad
\bm{v}_1 = \frac{1}{2}
\begin{pmatrix}
1 \\ -1 \\ 1 \\ -1
\end{pmatrix}.
\]

All other eigenvalues are zero.  
An orthonormal basis of the nullspace (orthogonal to \(\bm{v}_1\)) is
\[
\bm{v}_2 = \tfrac{1}{2}
\begin{pmatrix}
1 \\ 1 \\ -1 \\ -1
\end{pmatrix},
\qquad
\bm{v}_3 = \tfrac{1}{2}
\begin{pmatrix}
1 \\ -1 \\ -1 \\ 1
\end{pmatrix},
\qquad
\bm{v}_4 = \tfrac{1}{2}
\begin{pmatrix}
1 \\ 1 \\ 1 \\ 1
\end{pmatrix}.
\]

With the orthogonal matrix \(Q = [\bm{v}_1\;\bm{v}_2\;\bm{v}_3\;\bm{v}_4]\),
the spectral decomposition is therefore
\[
{K^{(1/\varepsilon)}} = Q\,
\operatorname{diag}(4,0,0,0)\,
Q^\top
  = 4\,\bm{v}_1\bm{v}_1^\top.
\]

As with slivers, all flat tetrahedra have a stiffness matrix that can be expressed as
\[
K_e \;=\; 
C_1\,\frac{h^2}{\varepsilon}\,K^{(1/\varepsilon)} 
\;+\;
C_2\,\varepsilon\,K^{(\varepsilon)},
\]
where the first term dominates as the element becomes thinner. 
The \( \tfrac{h^2}{\varepsilon} \) contribution reflects the artificially large stiffness 
associated with gradients across the small thickness~\(\varepsilon\), 
while the \( \varepsilon\,K^{(\varepsilon)} \) part corresponds to 
the in-plane modes that vanish as the tetrahedron flattens. In all cases, \(K^{(1/\varepsilon)}\) is of rank one and can be written as
\[
K^{(1/\varepsilon)} = \boldsymbol{v}_1 \boldsymbol{v}_1^{\!\top},
\]
see Table~\ref{tab:flat}.

We see here that flat tetrahedra introduce a single mode with infinite energy,
which cannot appear in the finite element solution. We have shown in \cite{quiriny2024tempered} that for isolated degenerate elements, it is sufficient to choose $0 < \varepsilon \ll h$ (e.g. $\varepsilon = 10^{-12} h$) to recover a convergent finite element solution.
By doing so, the flat element acts as a local constraint on the solution, preventing the activation of the infinite-energy mode.

As an example, consider the sliver. If \(u_j,\; j = 1, \dots, 4\) are the four values of a finite element field
at its vertices, then the finite element solution will be such that 
\begin{equation}
\label{eq:locksliver}
u_1 + u_3 = u_2 + u_4.    
\end{equation}
If the sliver is not perfectly regular, a similar conclusion holds:
the solution must take the same value where the two edges intersect within the sliver.

\input{lock}

Now, let us consider the band of slivers in Figure \ref{fig:sliverband}. Assume that $u_i$ is the value of a finite element solution at point $i$. Fix the values $u_1$, $u_3$, $u_5$, $u_7$, and $u_9$ at the lower nodes. We still have one degree of freedom; thus, let us fix $u_2$ as well. Then all other nodes are \emph{locked}:
$$u_4 = u_3+(u_2-u_1)~~,~~u_6 = (u_4+u_5)-u_3=u_5+(u_2-u_1)~~,$$
$$u_8 = u_7+(u_2-u_1)~~,~~u_{10} = u_9+(u_2-u_1).$$
The solution at the upper nodes is the same as at the lower nodes, shifted by a constant equal to $u_2-u_1$. This is exactly what is referred to as a \emph{locking phenomenon}. In the case of a sheet of slivers (a stack of band of slivers), the finite element solution cannot converge.
%nico: peut-on dire que le seul champ que l'on représenté est linéaire ? comme en 2D? - AQ: pas exactement, dans le cas présenté ici on pourrait très bien avoir une courbure de la solution le long de la bande mais elle doit être la même sur les noeuds d'en haut et d'en bas. Si on "stack" ces bandes de slivers pour former un plan on voit qu'on oblige d'avoir la même courbure dans une direction sur tout le plan. 

In general, the larger the number of nonzero entries in the vector \(\boldsymbol{v}_1\), 
the more the high-energy mode couples the nodes of the tetrahedron, 
and the more likely this degenerate element is to cause problems when forming bands. 
In this respect, the wedge is perfectly harmless, 
and choosing \(0 < \varepsilon \ll h\) will reliably recover finite element convergence. 

In contrast, the three other flat tetrahedra may lead to numerical issues 
and require a more subtle treatment: we have seen that slivers can form surface-like bands in which the solution locks over the entire surface. 
A similar behavior can be expected from combinations of caps and slivers. 
Spades, on the other hand, may form one-dimensional bands, 
analogous to the cap-type locking bands that occur in two-dimensional triangular meshes. In any case, it is dangerous in three dimensions to choose an arbitrarily small \(\varepsilon\),
as the number of configurations that may lead to locking is much larger than in two dimensions. Fortunately, we have recently introduced a robust mathematical framework, TFEM \cite{quiriny2024tempered},
for selecting \(\varepsilon\) as a function of \(h\), 
which guarantees the recovery of the optimal convergence rate of finite elements.

\subsection{Skinny tetrahedra}
\begin{table}[]
    \centering
    \begin{tabular}{|c|c|c|c|c|}
    \hline
      \input{spire}    &  \input{spear} &  \input{spindle}&  \input{spike} &  \input{splinter}\\
    \hline
    \end{tabular}
    \caption{Skinny tetrahedra.}
    \label{tab:skinny}
\end{table}

Unlike flat tetrahedra, \emph{skinny} tetrahedra possess \textbf{two}
independent directions that tend to zero. Consider a spear tetrahedron with vertices
\[
\bm{x}_1=(0,0,0), \quad
\bm{x}_2=(0,2h,0), \quad
\bm{x}_3=(-\varepsilon_1,h,\varepsilon_2), \quad
\bm{x}_4=(\varepsilon_1,h,\varepsilon_2).
\]
For the spear tetrahedron, the stiffness matrix can be decomposed as
\[
K
= K^{(1/\varepsilon_1)} + K^{(1/\varepsilon_2)} + K^{(\mathrm{\varepsilon_1\varepsilon_2})},
\]
%nico: j'écrirais plutot 
%K(eps1/eps2) que K(1/eps2) - AQ: Je crois que le but ici est d'isoler contribution singulière (1/eps_1 et 1/eps_2) de la partie régulière eps_1*eps_2.
where each singular contribution is of rank one and can be written as
\[
K^{(1/\varepsilon_1)}
= \frac{h\,\varepsilon_2}{6\,\varepsilon_1}\,
\boldsymbol{v}_1\,\boldsymbol{v}_1^{\!\top},
\qquad
K^{(1/\varepsilon_2)}
= \frac{h\,\varepsilon_1}{6\,\varepsilon_2}\,
\boldsymbol{v}_2\,\boldsymbol{v}_2^{\!\top},
\]
with
\[
\boldsymbol{v}_1 =
\begin{bmatrix} 0 \\[2pt] 0 \\[2pt] -1 \\[2pt] 1 \end{bmatrix},
\qquad
\boldsymbol{v}_2 =
\begin{bmatrix} 1 \\[2pt] 1 \\[2pt] -1 \\[2pt] -1 \end{bmatrix},
\]
and a bounded remainder
\[
K^{(\mathrm{\varepsilon_1\varepsilon_2})}
= \frac{\varepsilon_1\,\varepsilon_2}{6\,h}\,
\begin{bmatrix}
1 & -1 & 0 & 0\\[3pt]
-1 & 1 & 0 & 0\\[3pt]
0 & 0 & 0 & 0\\[3pt]
0 & 0 & 0 & 0
\end{bmatrix}.
\]
In fact, if only \(\varepsilon_1 \to 0\), the configuration reduces to the case of a \emph{wedge};
conversely, if only \(\varepsilon_2 \to 0\), it corresponds to a \emph{sliver}.
When \(\varepsilon_2 / \varepsilon_1 \approx 1\),
that is, when the element truly tends to a \emph{skinny} configuration,
no infinite-energy mode exists.
This is immediately clear if one recalls that in the skinny case,
the element volume scales as \(V \sim \varepsilon_1 \varepsilon_2\)
while the gradient products scale as \(1/(\varepsilon_1 \varepsilon_2)\),
thus making the stiffness matrix bounded. The bounded character of the stiffness matrices holds for all
\emph{skinny} tetrahedra — these elements do not require any special
treatment for convergence.

%% file: sliver.tex
\begin{tikzpicture}[scale=.6]
  \tikzstyle{v}=[circle,fill=black,inner sep=1.9pt]
  % Sommets (rectangle)
  \coordinate (BL) at (0,0);      % bottom-left
  \coordinate (TL) at (0,2.6);    % top-left
  \coordinate (BR) at (3.6,0);    % bottom-right
  \coordinate (TR) at (3.6,2.6);  % top-right
  % Bord du rectangle
  \draw[thick] (BL)--(TL)--(TR)--(BR)--cycle;
  % Diagonales (une visible, une “cachée”)
  \draw[thick]       (TL)--(BR);       % visible
  \draw[thick,dashed](BL)--(TR);       % cachée
  % Points
  \node[v] at (BL) {}; \node[v] at (TL) {}; \node[v] at (BR) {}; \node[v] at (TR) {};
  % Légende
  \node at (1.8,-0.6) {Sliver};
\end{tikzpicture}

%% file: cap.tex
\begin{tikzpicture}[scale=.6]
  \tikzstyle{v}=[circle,fill=black,inner sep=1.9pt]
  % Triangle équilatéral de côté 3
  \coordinate (A) at (0,0);
  \coordinate (B) at (3,0);
  \coordinate (C) at (1.5,{3*sqrt(3)/2}); % ~2.598
  % Centre (centroïde)
  \coordinate (O) at (1.5,{sqrt(3)/2});   % ~0.866

  % Arêtes du triangle
  \draw[thick] (A)--(B)--(C)--cycle;
  % Connexions du point central
  \draw[thick] (O)--(A);
  \draw[thick] (O)--(B);
  \draw[thick] (O)--(C);

  % Sommets
  \node[v] at (A) {}; \node[v] at (B) {}; \node[v] at (C) {};
  \node[v] at (O) {};
  % Légende (optionnelle)
   \node at (1.5,-0.6) {Cap};
\end{tikzpicture}

%% file: spade.tex
\begin{tikzpicture}[scale=.6]
  \tikzstyle{v}=[circle,fill=black,inner sep=1.9pt]

  % Trois sommets alignés verticalement à gauche
  \coordinate (T) at (0,2.4);   % top-left
  \coordinate (M) at (-.4,1.2);   % mid-left
  \coordinate (B) at (0,0.0);   % bottom-left
  % Sommet de droite, aligné en y avec M
  \coordinate (R) at (4.0,1.2);

  % --- Arêtes (ordre pour un rendu fidèle) ---
  % Grande arête gauche en tirets (T--B), visible entre les segments solides
  \draw[thick,dashed] (T)--(B);

  % Petits segments solides sur la gauche (recouvrent partiellement les tirets)
  \draw[thick] (T)--(M);
  \draw[thick] (M)--(B);

  % Connexions au sommet de droite
  \draw[thick] (T)--(R);
  \draw[thick] (B)--(R);
  \draw[thick] (M)--(R); % barre horizontale

  % Sommets
  \node[v] at (T) {}; \node[v] at (M) {}; \node[v] at (B) {}; \node[v] at (R) {};

  % Légende (optionnelle)
   \node at (2,-0.6) {Spade};
\end{tikzpicture}

%% file: wedge.tex
\begin{tikzpicture}[scale=.6]
  \tikzstyle{v}=[circle,fill=black,inner sep=1.9pt]
  % Sommets
  \coordinate (A) at (0,0);
  \coordinate (B) at (0,2.8);
  \coordinate (C) at (4.0,2.0);
  \coordinate (D) at (4.0,0.9);
  % Arêtes visibles
  \draw[thick] (A)--(B);
  \draw[thick] (C)--(D);
  \draw[thick] (A)--(C);
  \draw[thick] (A)--(D);
  \draw[thick] (B)--(C);
  % Arête “cachée”
  \draw[thick,dashed] (B)--(D);
  % Points
  \node[v] at (A) {}; \node[v] at (B) {}; \node[v] at (C) {}; \node[v] at (D) {};
  % Légende
  \node at (2,-0.6) {Wedge};
\end{tikzpicture}

%% file: lock.tex
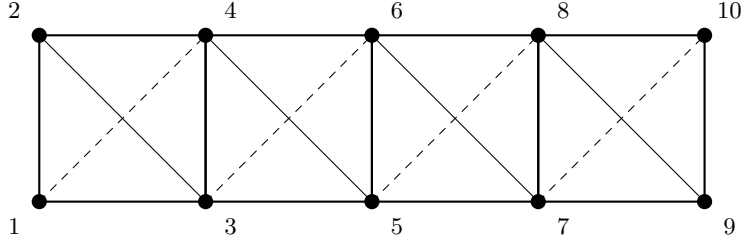
\begin{figure}
    \centering
\begin{tikzpicture}[scale=2.2, line join=round, line cap=round]

% --- Points (10 uniques) ---
\coordinate (P1)  at (0,0);
\coordinate (P2)  at (0,1);
\coordinate (P3)  at (1,0);
\coordinate (P4)  at (1,1);
\coordinate (P5)  at (2,0);
\coordinate (P6)  at (2,1);
\coordinate (P7)  at (3,0);
\coordinate (P8)  at (3,1);
\coordinate (P9)  at (4,0);
\coordinate (P10) at (4,1);

% --- Carrés + diagonales ---
\foreach \a/\b/\c/\d in {P1/P2/P4/P3, P3/P4/P6/P5, P5/P6/P8/P7, P7/P8/P10/P9}{
  \draw[thick] (\a)--(\b)--(\c)--(\d)--cycle;
  \draw[thin]  (\b)--(\d);
  \draw[dashed]  (\a)--(\c) ;
}

% --- Points + numéros ---
\foreach \i/\dx/\dy in {
  1/-0.15/-0.15,
  2/-0.15/ 0.15,
  3/ 0.15/-0.15,
  4/ 0.15/ 0.15,
  5/ 0.15/-0.15,
  6/ 0.15/ 0.15,
  7/ 0.15/-0.15,
  8/ 0.15/ 0.15,
  9/ 0.15/-0.15,
  10/0.15/ 0.15}{
  \fill (P\i) circle (1.3pt);
  \node[font=\small] at ($(P\i)+(\dx,\dy)$) {\i};
}

\end{tikzpicture}
    \caption{A band of slivers.}
    \label{fig:sliverband}
\end{figure}

%% file: spire.tex
\begin{tikzpicture}[scale=0.6]
  \tikzstyle{v}=[circle,fill=black,inner sep=1.9pt]
  % Sommets (petite base triangulaire + sommet très haut)
  \coordinate (BL) at (0,0);
  \coordinate (BM) at (1.2,-0.35);
  \coordinate (BR) at (2.4,0.0);
  \coordinate (T)  at (1.2,5.0);

  % Arêtes
  \draw[thick] (T)--(BL);
  \draw[thick] (T)--(BM);
  \draw[thick] (T)--(BR);
  \draw[thick] (BL)--(BM);
  \draw[thick,dashed] (BL)--(BR);
  \draw[thick] (BM)--(BR);

  % Sommets
  \node[v] at (BL) {}; \node[v] at (BM) {}; \node[v] at (BR) {}; \node[v] at (T) {};
  \node at (1.2,-1.2) {Spire};
\end{tikzpicture}

%% file: spear.tex
\begin{tikzpicture}[scale=0.6]
  \tikzstyle{v}=[circle,fill=black,inner sep=1.9pt]
  % Deux pointes (haut/bas) + deux points latéraux alignés horizontalement
  \coordinate (B)  at (1.6,0.0);
  \coordinate (T)  at (1.6,5.0);
  \coordinate (L)  at (0.6,2.35);
  \coordinate (R)  at (2.6,2.35);

  % Arêtes principales
  \draw[thick] (T)--(L);
  \draw[thick] (T)--(R);
  \draw[thick] (B)--(L);
  \draw[thick] (B)--(R);
  \draw[thick] (T)--(B);      % mât central
  % Arête “cachée” (liaison horizontale)
  \draw[thick,dashed] (L)--(R);

  % Sommets
  \node[v] at (B) {}; \node[v] at (T) {}; \node[v] at (L) {}; \node[v] at (R) {};
  \node at (1.6,-1.) {Spear};
\end{tikzpicture}

%% file: spindle.tex
\begin{tikzpicture}[scale=0.6]
  \tikzstyle{v}=[circle,fill=black,inner sep=1.9pt]
  % Pointes haut/bas + deux points latéraux (légèrement décalés)
  \coordinate (B)  at (1.4,0.0);
  \coordinate (T)  at (1.4,5.0);
  \coordinate (L)  at (0.80,3.333);
  \coordinate (R)  at (2.,1.6666);

  % Contour
  \draw[thick] (T)--(L)--(B);
  \draw[thick] (T)--(R)--(B);
  % Arête cachée (mât)
  \draw[thick,dashed] (T)--(B);
  % Arêtes internes visibles
  \draw[thick] (L)--(R);

  % Sommets
  \node[v] at (B) {}; \node[v] at (T) {}; \node[v] at (L) {}; \node[v] at (R) {};
   \node at (1.6,-1.) {Spindle};
\end{tikzpicture}

%% file: spike.tex
\begin{tikzpicture}[scale=0.5]
  \tikzstyle{v}=[circle,fill=black,inner sep=1.9pt]

  % Géométrie : deux points en bas (base légèrement oblique),
  % un point au-dessus d’eux (verticalement allongé),
  % et un quatrième décalé latéralement.
  \coordinate (B1) at (0,0);
  \coordinate (B2) at (1.2,0);
  \coordinate (T)  at (0.6,5.0);
  \coordinate (S)  at (1.4,2.4); % sommet latéral décalé (le “spike”)

  % Arêtes principales
  \draw[thick] (B1)--(B2);
  \draw[thick] (B1)--(T);
  \draw[thick] (B2)--(T);
  \draw[dashed] (B1)--(S);
  \draw[thick] (B2)--(S);
  \draw[thick] (T)--(S);

  % Arêtes cachées (une diagonale dans la base latérale)
  \draw[thick,dashed] (B1)--(T);

  % Sommets
  \node[v] at (B1) {};
  \node[v] at (B2) {};
  \node[v] at (T) {};
  \node[v] at (S) {};

   \node at (1.0,-1) {Spike};
\end{tikzpicture}

%% file: splinter.tex
\begin{tikzpicture}[scale=0.5]
  \tikzstyle{v}=[circle,fill=black,inner sep=1.9pt]
  % Très fin et quasi-rectiligne
  \coordinate (BL) at (0.20,0.3);
  \coordinate (BR) at (2.0,0.0);
  \coordinate (TL) at (0.35,5.3);
  \coordinate (TR) at (2.15,5.0);

  % Bord long et mince
  \draw[thick] (TL)--(BL);
  \draw[thick] (TR)--(BR);
  \draw[thick] (TL)--(TR);
  \draw[thick] (BL)--(BR);

  % Diagonale “cachée” longue
  \draw[thick,dashed] (TL)--(BR);
  \draw[thick] (BL)--(TR);

  % Petite barre supérieure légèrement oblique
  \draw[thick] (TL)--($(TL)!0.25!(TR)$);

  % Sommets
  \node[v] at (TL) {}; \node[v] at (TR) {}; \node[v] at (BL) {}; \node[v] at (BR) {};
  \node at (1.3,-1) {Splinter};
\end{tikzpicture}

%% file: Isolated_slivers.tex
\section{Isolated slivers do not matter}
In this section, we show that slivers and other `bad' elements are not actually problematic, as long as they are sufficiently isolated. We present a simplified version of the construction from \cite{kuvcera2016necessary}, which makes it possible to greatly weaken the `isolatedness' assumption.

We consider Poisson's problem on a polyhedral domain $\Omega \subset \mathbb{R}^3$, with homogeneous Dirichlet boundary conditions for simplicity:
\begin{equation}
- \Delta u = f \quad \text{in } \Omega,
\qquad u = 0 \quad \text{on } \partial \Omega,
\end{equation}
where $f : \Omega \to \mathbb{R}$ is given. The weak form is defined on the Sobolev space $H^1_0(\Omega)$ equipped with the standard seminorm. Let $V_h$ be the space of piecewise linear continuous functions on a face-to-face partition $\mathcal{T}_h$ of $\Omega$ into tetrahedra. The finite element solution then reads: find $u_h\in V_h$ such that
\begin{equation}
\int_\Omega \nabla u_h\cdot\nabla v_h\,\mathrm{d}x= \int_\Omega f v_h\,\mathrm{d}x,\quad\forall v_h\in V_h.
\end{equation}
We denote $h_K = \mathrm{diam}\,K$ for $K\in\mathcal{T}_h$, and the mesh parameter $h=\max_K h_K$. In the following, $C$ will denote a generic constant independent of $u$ and $h$.

The basic tool to estimate the error in the $H^1$-seminorm is Cea's lemma:
\begin{equation}
|u-u_h|_{H^1(\Omega)}\leq C\inf_{v_h\in V_h}|u-v_h|_{H^1(\Omega)}.
\end{equation}

The basic approach is now to estimate the right-hand side by taking a specific $v_h$ for which $u-v_h$ is easy to estimate. The natural choice is to take the piecewise linear \emph{Lagrange interpolant} $v_h:=\Pi_h u$, defined using the local Lagrange interpolation $\Pi_K u\in P^1(K)$ on element $K\in\mathcal{T}_h$:
\begin{equation}
\Pi_K u(z)=u(z),\quad\forall z\in\{A_K,B_K,C_K,D_K\},
\end{equation}
where $A_K,B_K,C_K,D_K$ are the four vertices of $K$ taken in any order. The problem with Lagrange interpolation is that the approximation of gradients deteriorates with the poor geometry of $K$. To this end, we construct a modified interpolation $\widetilde{\Pi}_K u\in P^1(K)$:
\begin{equation}
\widetilde{\Pi}_K u(A_K)=u(A_K),\quad \nabla\widetilde{\Pi}_K u=\frac{1}{|K|}\int_K\nabla u\,\mathrm{d}x.
\end{equation}
The point of the construction is to interpolate at one arbitrarily chosen vertex $A_K$ and prescribe the gradient of $\widetilde{\Pi}_K u$ to be the averaged gradient of $u$. Thus, the gradient approximation is $\mathcal{O}(h)$ independently of the geometry of $K$. The trade-off in this construction is that we no longer interpolate at all vertices, but only at the chosen vertex $A_K$. We have thus introduced a \emph{vertex defect} at the \emph{defect vertices} $B_K,C_K,D_K$, which is nevertheless small, namely $\mathcal{O}(h^2)$:

\begin{lemma}[Averaged-gradient projection on tetrahedra]
Let $K\in\mathcal{T}_h$ and $h_K:=\operatorname{diam}(K)$.
Then there exists a constant $C$ independent of $K,u$ such that:
\begin{enumerate}
\item If \(u\in H^2(K)\), then
\[
|u-\widetilde{\Pi}_K u|_{H^1(K)}
\le C h_K |u|_{H^2(K)} .
\]

\item If \(u\in W^{2,\infty}(\Omega)\), denote $M := |u|_{W^{2,\infty}(\Omega)}$. Then for every vertex
\(z\in\{A_K,B_K,C_K,D_K\}\),
\begin{equation}
\label{eq:defect_est}
|u(z)-\widetilde{\Pi}_K u(z)|
\le C h_K^2 M.
\end{equation}
\end{enumerate}
\end{lemma}

\begin{proof}
Set
\[
\overline g_K:=\frac{1}{|K|}\int_K \nabla u\,dx .
\]
Then
\[
\nabla \widetilde{\Pi}_K u=\overline g_K.
\]

First assume \(u\in H^2(K)\). Since \(K\) is convex, the Poincaré--Wirtinger
inequality on convex sets gives, componentwise applied to \(\nabla u\),
\[
\|\nabla u-\overline g_K\|_{L^2(K)}
\le C h_K \|D^2u\|_{L^2(K)} .
\]
Therefore
\[
|u-\widetilde{\Pi}_K u|_{H^1(K)}
= \|\nabla u-\overline g_K\|_{L^2(K)}
\le C h_K |u|_{H^2(K)} .
\]

Now assume \(u\in W^{2,\infty}(K)\). The estimate is trivial for
\(z=A_K\), since \(\widetilde{\Pi}_K u(A_K)=u(A_K)\). Let
\(z\in\{B_K,C_K,D_K\}\). Then
\[
u(z)-\widetilde{\Pi}_K u(z) = u(z)-u(A_K)-\overline g_K\cdot (z-A_K).
\]
Add and subtract \(\nabla u(A_K)\cdot (z-A_K)\):
\[
u(z)-\widetilde{\Pi}_K u(z)
=\underbrace{u(z)-u(A_K)-\nabla u(A_K)\cdot (z-A_K)}_{T_1} +\underbrace{\bigl(\nabla u(A_K)-\overline g_K\bigr)\cdot (z-A_K)}_{T_2}.
\]

\noindent Taylor's formula yields
\[
|T_1| \le C |z-A_K|^2 M \le C h_K^2 M.
\]
Furthermore,
\[
|\nabla u(A_K)-\overline g_K|
= \left|\frac{1}{|K|}\int_K \bigl(\nabla u(A_K)-\nabla u(x)\bigr)\,dx
\right| \le C h_K M,
\]
since \(|x-A_K|\le h_K\) for all \(x\in K\). Hence
\[
|T_2| \le |\nabla u(A_K)-\overline g_K|\,|z-A_K| \le C h_K^2 M.
\]
Combining the estimates for \(T_1\) and \(T_2\) yields
\[
|u(z)-\widetilde{\Pi}_K u(z)|
\le C h_K^2 M.
\]
\end{proof}

\subsection{Defect correction on meshes with `bad' elements}
We split the mesh into
\[
\mathcal{T}_h = \mathcal{T}_h^1 \cup \mathcal{T}_h^2,
\]
where $\mathcal{T}_h^1$ consists of shape-regular (`good') elements and $\mathcal{T}_h^2$ contains possibly degenerate (`bad') tetrahedra, e.g. slivers. 

On $\mathcal{T}_h^2$ we will use the averaged-gradient projection introduced above, which is oblivious to the bad geometry of these elements. However, due to the vertex defect, one cannot straightforwardly `glue' these linear functions $\widetilde{\Pi}_Ku$ together continuously, since we no longer nodally interpolate $u$. The trick is to construct a modified function $\tilde{u}$ that $\widetilde{\Pi}_Ku$ nodally interpolates instead of $u$.

Let $\mathcal Z_h$ be the set of all defect vertices of all $K\in\mathcal T_h^2$. For each $z\in\mathcal Z_h$ let
$K_z\in\mathcal T_h^2$ denote the unique bad element associated with $z$, and set $h_z:=h_{K_z}$. For each $z \in \mathcal{Z}_h$ define its defect
\[
\delta_z := \widetilde{\Pi}_{K_z} u(z) - u(z).
\]
For each vertex $z\in\mathcal{Z}_h$ define a \emph{defect correction function} $\varphi_z\in C^\infty(\Omega)$ such that:
\begin{enumerate}
\item  $\varphi_z$ is a scaling and translation of a single radially symmetric smooth `bump' function $\varphi$,
\item $\mathrm{supp}\,\varphi_z(z) = \overline{B(z,ch_z)}$ for some $c$ independent of $z,h$ (i.e. a ball centered at $z$ with radius $ch_z$),
\item $\varphi_z(z)=\delta_z$ and $|\varphi_z(x)|\le |\delta_z|$ for all $x\in\Omega$,
\item the supports of all $\varphi_z$ are disjoint,
\item $\varphi_z(y) = 0$ for all other vertices $y \neq z$ of elements in $\mathcal{T}_h^2$.
\end{enumerate}

\begin{lemma}
\label{lem:defect_correction}
Let $z\in\mathcal{Z}_h$ and $u\in W^{2,\infty}(\Omega)$. Then
\[
\|\nabla \varphi_z\|_{L^\infty} \le C h_K M,
\qquad
\|D^2 \varphi_z\|_{L^\infty} \le C M.
\]
Let
\[
\eta := \sum_{z \in \mathcal{Z}_h} \varphi_z,
\qquad
\widetilde{u} := u + \eta.
\]
Then
\[
|\eta|_{H^1(\Omega)} \le C h M,
\qquad
|\widetilde{u}|_{H^2(\Omega)} \le C M.
\]
\end{lemma}
\begin{proof}
Since the `height' of $\varphi_z$ is $|\delta_z|$ and its support has radius $h_z$, we have
\[
\|\nabla \varphi_z\|_{L^\infty(\Omega)} \le C \frac{|\delta_z|}{h_z}\leq Ch_z M,
\qquad
\|D^2 \varphi_z\|_{L^\infty(\Omega)} \leq \frac{|\delta_z|}{h_z^2} \leq C M,
\]
due to (\ref{eq:defect_est}).
Furthermore,
\[
|\eta|_{H^1(\Omega)}^2 =\sum_{z\in\mathcal{Z}_h} \|\nabla\varphi_z\|_{L^2(\Omega)}^2 \leq \max_{z\in\mathcal{Z}_h}\|\nabla\varphi_z\|_{L^\infty(\Omega)}^2 \sum_{z\in\mathcal{Z}_h} |\mathrm{supp}\,\varphi_z| \leq C h^2 M^2|\Omega|,
\]
since the supports of $\varphi_z$ are disjoint. Similarly, since $\|D^2 \varphi_z\|_{L^\infty(\Omega)} \leq C M$,
\[
|\eta|_{H^2(\Omega)} \le C M,\qquad |\widetilde{u}|_{H^2(\Omega)} \le |u|_{H^2(\Omega)} +|\eta|_{H^2(\Omega)}\leq |\Omega|^{1/2}M + C M,
\]
since $|u|_{H^2(\Omega)}^2\leq M^2|\Omega|$.
\end{proof}

\begin{theorem}
Let $\mathcal{T}_h^2$ have well-separated vertices, i.e. let there exist $c>0$ such that, for all $y,z\in \mathcal{Z}_h$, one has $|y-z|\ge c(h_y+h_z)$. Let $w_h:=\Pi_h\widetilde{u}$ be the Lagrange interpolation of $\widetilde{u}$. Then
\[
|u - w_h|_{H^1(\Omega)} \le C h M,
\]
hence, by Cea's lemma,
\[
|u - u_h|_{H^1(\Omega)} \le C h M.
\]
\end{theorem}
\begin{proof}
Split the domain into $\Omega_1 := \bigcup_{K\in \mathcal{T}_h^1} K$ and $\Omega_2 := \bigcup_{K\in \mathcal{T}_h^2} K$. We note that the well-separatedness assumption on $\mathcal{T}_h^2$ allows the supports of $\varphi_z$ to be disjoint and ensures that $\varphi_z(y) = 0$ for all other vertices $y \neq z$ of elements in $\mathcal{T}_h^2$, which were essential properties in the construction above.

\smallskip

\textbf{Bad elements.} On $K \in \mathcal{T}_h^2$ we have $w_h|_K := {\Pi}_K \widetilde{u} = \widetilde{\Pi}_K u$, since $\widetilde{\Pi}_K u$ attains the values of $\widetilde{u}$ at the vertices of $K$, hence $\widetilde{\Pi}_K u$ coincides with the Lagrange interpolation of $K$. Lemma \ref{lem:defect_correction} therefore gives us
\[
|u - w_h|_{H^1(K)} \le C h_K |u|_{H^2(K)}.
\]
Summing and using $|u|_{H^2(K)}^2\leq M^2|K|$ gives
\[
|u - w_h|_{H^1(\Omega_2)}^2=\sum_{K \in \mathcal{T}_h^2} |u - w_h|_{H^1(K)}^2\leq Ch^2M^2\sum_{K \in \mathcal{T}_h^2} |K| 
\le C h^2 M^2|\Omega|,
\]
hence $|u - w_h|_{H^1(\Omega_2)} \le C hM.$

\smallskip

\textbf{Good elements.}
On $\mathcal{T}_h^1$,
\[
u - w_h = (u - \widetilde{u}) + (\widetilde{u} - \Pi_h \widetilde{u}).
\]
Thus
\[
|u - w_h|_{H^1(\Omega_1)}
\le |\eta|_{H^1(\Omega)} + |\widetilde{u} - \Pi_h \widetilde{u}|_{H^1(\Omega_1)}.
\]
The first term is bounded by $C h M$ by Lemma \ref{lem:defect_correction}. For the second term, standard interpolation theory on shape-regular elements gives
\[
|\widetilde{u} - \Pi_h \widetilde{u}|_{H^1(\Omega_1)}
\le C h |\widetilde{u}|_{H^2(\Omega)}
\le C h M.
\]
Combining the estimates completes the proof.
\end{proof}

\begin{remark}
Once we construct the gradient-averaged interpolation $\widetilde{\Pi}_K u$ on a `bad' element $K$, which no longer interpolates $u$ at the defect vertices (e.g. $z\in\mathcal{Z}_h$), it is tempting to simply modify the Lagrange interpolation on the neighboring `nice' element $\widetilde{K}$ so that it attains the value $\widetilde{\Pi}_K u(z)$ at $z$. This is certainly possible; however, one must then impose additional conditions on the nice element $\widetilde{K}$ for this to work. Consider, for example, the case where $\widetilde{K}$ is very small, e.g. $h_{\widetilde{K}}=h^2$. Now we make an $\mathcal{O}(h^2)$ perturbation to one of its vertex values (this is the vertex defect on $K$). Then the gradient of the interpolant on $\widetilde{K}$ changes like $\mathcal{O}(h^2)/h^2=\mathcal{O}(1)$, which destroys the desired $\mathcal{O}(h)$ approximation properties of the gradient. In our approach, we circumvent this by letting the function $u$ carry the burden of the vertex defects, modifying it to $\widetilde{u}$, which is then interpolated using only the assumption of shape regularity of the `nice' elements. Effectively, using $\varphi_z$, the influence of the defect can be smeared over a neighborhood rather than a single neighboring element.
\end{remark}

\begin{remark}
For clarity, the gradient-averaged interpolation construction presented here is a heavily simplified version of the more general construction from \cite{kuvcera2016necessary}. Here, we interpolate at one vertex $A_K$ of $K$, introducing $\mathcal{O}(h^2)$ defects at the other three vertices of $K$. In the 3D version of \cite{kuvcera2016necessary}, one interpolates $u$ at three vertices $A_K,B_K,C_K$ and introduces a defect only at the one remaining vertex. This is more general, since only the fourth vertex $D_K$ must then be separated from the others (to ensure non-overlapping supports of $\varphi_z$), while $A_K,B_K,C_K$ can be trivially connected to other `nice' elements or other non-defect vertices of other `bad' elements. Thus, one can connect the `bad' elements into large structures spanning the whole domain $\Omega$ while still obtaining $\mathcal{O}(h)$ convergence. Moreover, the vertex defect is estimated more finely in \cite{kuvcera2016necessary}, allowing for less strict separation of the defect vertices.
\end{remark}

%% file: tfem.tex
\section{The Tempered Finite Element Method for Slivers}
\label{sec:tfem}

To evaluate whether a finite element is suitable for computation, one often
examines the interpolation error. This viewpoint has led 
to classical geometric conditions for 2D triangular meshes, such as the
maximum-angle condition introduced by Babuška and Aziz \cite{babuvska1976angle}. 
The convergence of the interpolation error implies the convergence of the finite element method. However, it would be incorrect to conclude that the method cannot converge when the interpolation error does not, since the converse implication does not hold.

A particularly illustrative example is given in Fig.~\ref{fig:hannukainen} from 
Hannukainen et al.~\cite{hannukainen2012maximum}. Two mesh sequences are compared: the first exhibits optimal interpolation and FEM convergence rates, whereas the second—constructed by splitting each triangle into three, with one of them becoming extremely flattened—violates the maximum-angle condition and shows no interpolation convergence.
Despite this, its FEM error still converges optimally. The key observation is 
that the discrete space $V_{h2}$ associated with the refined meshes is a superset of
the space $V_{h1}$ of the original ones. Céa's lemma shows that the FEM solution 
is the best solution in the natural norm; hence, the solution on the \emph{ugly mesh}
must be at least as good as the one in the \emph{beautiful mesh}. 
This example shows that {\bf the maximum-angle condition is not a 
necessary condition for finite element convergence}.
\begin{figure}[h!]
    \centering
    \includegraphics[width=0.9\textwidth]{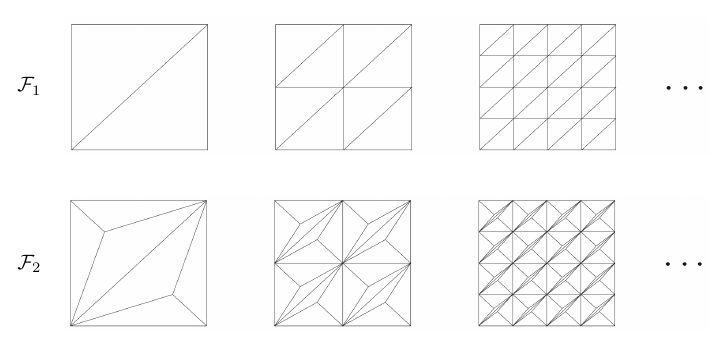}
    \caption{Two families of meshes from \cite{hannukainen2012maximum}: $\mathcal{F}_1$ satisfies the maximum angle condition, while $\mathcal{F}_2$ violates it.}
    \label{fig:hannukainen}
\end{figure}

\subsection{Tempering the equations}
In \S\ref{sec:deg}, we showed that the locking conditions induced by degenerate elements arise from the dominant contribution of their stiffness matrices, which tend to infinity as the element collapses. This behaviour results from the fact that the gradients of the basis functions grow unbounded while the element volume vanishes, producing extremely large entries in the local stiffness matrix. To evaluate both the integrals and the gradients of the basis functions, the standard finite element formulation relies on a reference element $\hat{E}$ with coordinates $\bm{\xi}$ and a mapping between the reference and physical elements, as illustrated in Figure \ref{fig:mapping}.

Let
$$
J =
\begin{pmatrix}
x_\xi & x_\eta & x_\zeta \\
y_\xi & y_\eta & y_\zeta \\
z_\xi & z_\eta & z_\zeta
\end{pmatrix}$$ and 
$\det J = |J| =
x_\xi (y_\eta z_\zeta - y_\zeta z_\eta)
- x_\eta (y_\xi z_\zeta - y_\zeta z_\xi)
+ x_\zeta (y_\xi z_\eta - y_\eta z_\xi).
$
%Then
%\[
%\frac{\partial(\xi,\eta,\zeta)}{\partial(x,y,z)}
%= J^{-1}
%= \frac{1}{|J|}
%\begin{pmatrix}
% y_\eta z_\zeta - y_\zeta z_\eta
% & -x_\eta z_\zeta + x_\zeta z_\eta
% & x_\eta y_\zeta - x_\zeta y_\eta
%\\[0.4em]
% -y_\xi z_\zeta + y_\zeta z_\xi
% & x_\xi z_\zeta - x_\zeta z_\xi
% & -x_\xi y_\zeta + x_\zeta y_\xi
%\\[0.4em]
% y_\xi z_\eta - y_\eta z_\xi
% & -x_\xi z_\eta + x_\eta z_\xi
% & x_\xi y_\eta - x_\eta y_\xi
%\end{pmatrix}.
%\]

Let us go back to the stiffness matrix of a sliver computed in \S\ref{sec:deg}. Assuming 
$|J| = 8 h^2 \varepsilon$, we use TFEM here to stabilize 
the sliver
\[
K = D 
\frac{ 8 h^4}{3 |J|} 
\underbrace{\begin{pmatrix}
1 & -1 & 1 & -1\\
-1& 1 & -1& 1 \\
1 & -1& 1 & -1\\
-1& 1 & -1& 1
\end{pmatrix}}_{K^{(1/\varepsilon)}}
+{\varepsilon \over 6}
\underbrace{
\begin{pmatrix}
1 & 0 & -1 & 0\\
0 & 1 & 0 & -1\\
-1& 0 & 1 & 0\\
0 & -1& 0 & 1
\end{pmatrix}}_{K^{(\varepsilon)}}
\]
with $D = {|J| \over \max(|J|,|J|_{\min})}$ and
\begin{align} \nonumber
        |J|_{\min} = h^4 \frac{|u|_{\WIIi(\B)}}{|u|_{\WIi(\B)}} = {h^4 \over H}
\end{align}
where
$H$ is a length and $\B$ is the band of slivers.
In our context, if $\mathcal B$ is a patch of diameter $D$,
then $H \sim D$. For \emph{bad} functions at the mesh scale -- take $\sin(2\pi x/L)$ as an example -- one even has explicitly
\[
\frac{|u|_{\WIIi(\B)}}{|u|_{\WIi(\B)}} 
\,\sim\, \frac{1}{L}.
\]
%\footnote{On an interval 
%nico: valable aussi en 2D ?- AQ: oui 
%$[a,b]$, one has the classical inequality
%\[
%\|f'\|_{L^\infty(a,b)}
%    \;\le\; (b-a)\,\|f''\|_{L^\infty(a,b)}.
%\]
%It follows that
%\[
%\frac{\|f''\|_{L^\infty(a,b)}}{\|f'\|_{L^\infty(a,b)}}
%    \;\ge\; \frac{1}{b-a}.
%\]
%This provides an a priori bound on the ratio 
%$\max |f''| / \max |f'|$ without requiring explicit knowledge of the function~$f$.
%}. 
Thus, the stabilization is
triggered when $8h^2 \varepsilon < h^4/H$ and thus when $\varepsilon  < {h^2 \over 8 H}$. At that point, the
matrix $K^{(1/\varepsilon)}$ remains constant and $K$ is computed as
\[
K = 
\frac{8H}{3} 
\begin{pmatrix}
1 & -1 & 1 & -1\\
-1& 1 & -1& 1 \\
1 & -1& 1 & -1\\
-1& 1 & -1& 1
\end{pmatrix}
+{\varepsilon \over 6}
\begin{pmatrix}
1 & 0 & -1 & 0\\
0 & 1 & 0 & -1\\
-1& 0 & 1 & 0\\
0 & -1& 0 & 1
\end{pmatrix}.
\]

\begin{figure}[!ht]
\input{tetref}
\caption{Mapping between the physical and reference domains.} \label{fig:mapping}
\end{figure}
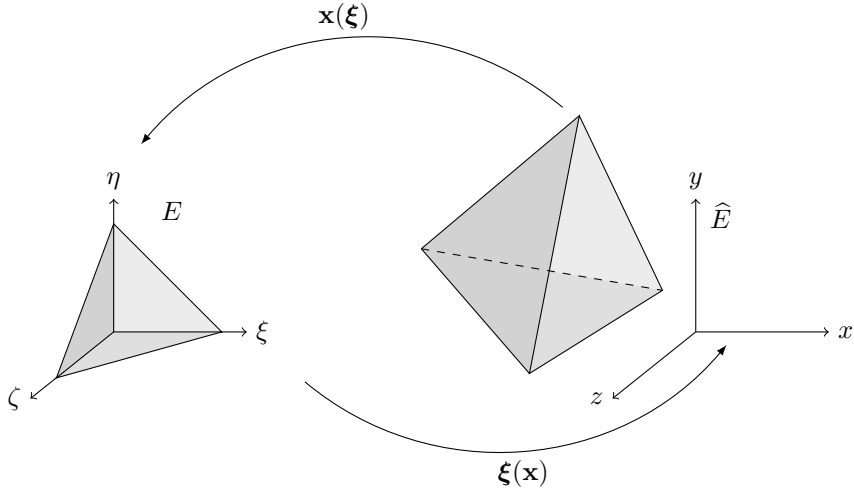

%% file: tetref.tex
\begin{tikzpicture}[scale=1.1]

% ===== Axes de gauche (x,y,z) =====
\draw[->] (0,0,0) -- (1.6,0,0) node[right] {$\xi$};
\draw[->] (0,0,0) -- (0,1.6,0) node[above] {$\eta$};
\draw[->] (0,0,0) -- (-1.0,-0.8,0) node[left] {$\zeta$};

% Un tétraèdre quelconque E
\coordinate (A) at (0,0,0);
\coordinate (B) at (1.3,0.,0);
\coordinate (C) at (0.0,1.3,0);
\coordinate (D) at (-.81*.85,-0.65*.85,0);

\fill[gray!15] (A) -- (B) -- (C) -- cycle;
\fill[gray!25] (A) -- (B) -- (D) -- cycle;
\fill[gray!35] (A) -- (C) -- (D) -- cycle;

\draw (A)--(B)--(C)--cycle;
\draw (A)--(D)--(B);
\draw (C)--(D);

\node at (0.7,1.45) {$E$};

% ===== Axes de droite (xi,eta,zeta), plus grands =====
\begin{scope}[shift={(6,0)}]

  % ---- Tétraèdre de référence dans ce repère ----
  \coordinate (O) at (-1,-.5,0);         % (0,0,0)
  \coordinate (X) at (0.6,.5,0);       % (1,0,0) direction xi
  \coordinate (Y) at (-.4,2.6,0);       % (0,1,0) direction eta
  \coordinate (Z) at (-2.3,1.0,0);    % (0,0,1) direction zeta

  % Faces
  \fill[gray!15] (O)--(X)--(Y)--cycle;
  \fill[gray!25] (O)--(X)--(Z)--cycle;
  \fill[gray!35] (O)--(Y)--(Z)--cycle;

  % Arêtes du tétraèdre
  \draw (O)--(X)--(Y)--cycle;
  \draw (O)--(Z);
  \draw[dashed] (Z)--(X);
  \draw (Y)--(Z);

\draw[->] (0+1,0,0) -- (1.6+1,0,0) node[right] {$x$};
\draw[->] (0+1,0,0) -- (0+1,1.6,0) node[above] {$y$};
\draw[->] (0+1,0,0) -- (-1.0+1,-0.8,0) node[left] {$z$};

  % Labels sommets
%  \node[below] at (O) {$(0,0,0)$};
%  \node[below] at (X) {$(1,0,0)$};
%  \node[right] at (Y) {$(0,1,0)$};
%  \node[left]  at (Z) {$(0,0,1)$};

  % Label \hat{E}
  \node at (1.3,1.4) {$\widehat{E}$};

\end{scope}

% ===== Flèches de transformation entre les deux =====

\draw[-latex] (5.4,2.7) arc (50:140:3.6)
  node[midway, above] {$\mathbf{x}(\boldsymbol{\xi})$};

\draw[-latex] (2.3,-0.6) arc (-130:-40:3.6)
  node[midway, below] {$\boldsymbol{\xi}(\mathbf{x})$};

\end{tikzpicture}

%% file: stresstest.tex
\subsection{Stress test}

We now consider a single, deliberately extreme stress test. The idea is to solve a problem
with a manufactured solution, namely a Poisson equation of the form
$-\nabla^2 u = -\nabla^2 f$, where $f(x,y,z)$ is prescribed.
We assess the behaviour of a direct application of TFEM in this challenging setting.

We solve the problem in a cube of side length~$2$. Inside this cube, we insert two spheres,
both centred at the centre of the cube and with very close radii, $R_1 = 1$ and
$R_2 = 1.01$. We then generate a mesh of the three volumetric domains and move all the points
lying on the sphere of radius $R = 1.0$ onto the sphere of radius $R = 1.01$.
All tetrahedra in the volume between these two spheres are therefore degenerate.

As a first step, we use Gmsh’s transfinite-type meshing functions to mesh the
surfaces of these two spheres, which ultimately produces exclusively sliver
elements (see Figure~\ref{fig:2spheres}).

\begin{figure}[!ht]
    \centering
    \includegraphics[width=0.85\linewidth]{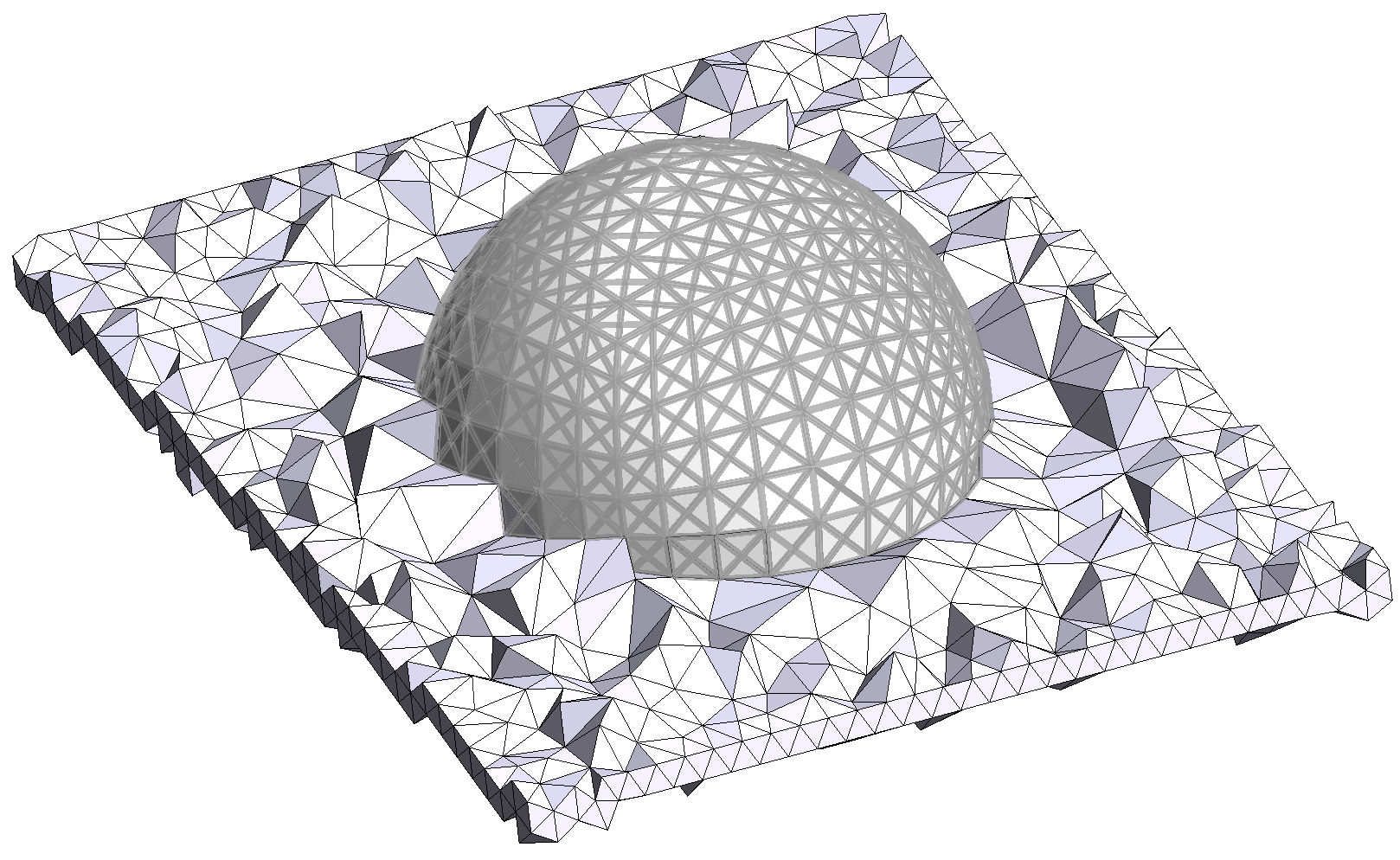}
    \caption{Extreme mesh where a complete region of the mesh is made
    of slivers.}
    \label{fig:2spheres}
\end{figure}

In this mesh, the element size varies. In the context of TFEM stabilization,
we use the local mesh size to determine whether the limit
of $|J|_{\min}$ has been reached.

\medskip
\noindent {\bf Manufactured solution:} We consider the sinusoidal function
$u_{ex}(x,y,z) = \sin(\alpha x)\sin(\alpha y)\sin(\alpha z)$ with $\alpha = 2\pi/8$, and solve
\[
-\nabla^2 u = 3 \alpha^2 u_{ex}(x,y,z).
\]
We anticipate that $H=8$ is a suitable choice for TFEM in this case.

\input{optimal_choice}

The influence of the parameter $H$ is assessed by measuring the $L^2$ error
for a range of mesh sizes $h$ and different values of $H$. The results show
that the optimal choice of $H$ is essentially independent of the mesh size,
with all curves exhibiting a clear minimum around $H = 8$. This confirms that
the parameter can be selected based on the analytical structure of the problem
rather than on discretization considerations.

When $1/H$ is chosen too small, the stabilization becomes overly stiff and
leads to a locking behaviour, resulting in a significant degradation of the
solution. Conversely, when $1/H$ is too large, the stabilization is excessively
weak and alters the solution itself, again increasing the error. The value
$H = 8$ therefore provides a balance between these two effects, yielding a
minimum error that is of the same order, or even smaller, than that obtained on regular meshes.

%% file: optimal_choice.tex
\begin{figure}[!ht]
    \centering
    \begin{tikzpicture}
        \begin{axis}[
            xlabel={$\frac{1}{H}$},
            ylabel={$L^2$ error -- $\sqrt{\frac{\int \|u-u_h\|^2 dv}{\int \|u\|^2 dv}}$ }, %multiple version, a clarifier
            legend style={at={(0.65,0.4)},anchor=west},
            ymajorgrids=true,
            grid=none,
            ymode=log,
            xmode=log,
            ymin=0.001,
            xmax=35,
            width=\textwidth,
            height=0.7\textwidth,
            y label style={at={(axis description cs:0.35,0.8)},rotate=-90,anchor=south},
            legend style={draw=none}
        ]
        %h=0.2
        \addplot[
            color=blue,
            mark = *,
            mark size=1.5pt,
            line width=1pt
            ]
            coordinates {
                (1.000000e-08, 2.172513e-01) (1.603719e-08, 2.171699e-01) (2.571914e-08, 2.170395e-01) (4.124626e-08, 2.168309e-01) (6.614741e-08, 2.164975e-01) (1.060818e-07, 2.159659e-01) (1.701254e-07, 2.151208e-01) (2.728333e-07, 2.137845e-01) (4.375479e-07, 2.116884e-01) (7.017038e-07, 2.084415e-01) (1.125336e-06, 2.035079e-01) (1.804722e-06, 1.962219e-01) (2.894266e-06, 1.858907e-01) (4.641589e-06, 1.720241e-01) (7.443803e-06, 1.546557e-01) (1.193777e-05, 1.345615e-01) (1.914482e-05, 1.131593e-01) (3.070291e-05, 9.209718e-02) (4.923883e-05, 7.280550e-02) (7.896523e-05, 5.621502e-02) (1.266380e-04, 4.270705e-02) (2.030918e-04, 3.219314e-02) (3.257021e-04, 2.427489e-02) (5.223345e-04, 1.844757e-02) (8.376776e-04, 1.421812e-02) (1.343399e-03, 1.118209e-02) (2.154435e-03, 9.033611e-03) (3.455107e-03, 7.520991e-03) (5.541020e-03, 6.456889e-03) (8.886238e-03, 5.713059e-03) (1.425103e-02, 5.189946e-03) (2.285464e-02, 4.799365e-03) (3.665241e-02, 4.479006e-03) (5.878016e-02, 4.176666e-03) (9.426685e-02, 3.847077e-03) (1.511775e-01, 3.594640e-03) (2.424462e-01, 3.729335e-03) (3.888155e-01, 5.576770e-03) (6.235507e-01, 2.559933e-02) (1.000000e+00, 1.597447e-01)
            };
        
        % h=0.15
        \addplot[
            color=orange,
            mark = triangle*,
            mark size=1.5pt,
            line width=1pt
            ]
            coordinates {
               (1.000000e-08, 2.222947e-01) (1.603719e-08, 2.222530e-01) (2.571914e-08, 2.221863e-01) (4.124626e-08, 2.220793e-01) (6.614741e-08, 2.219081e-01) (1.060818e-07, 2.216342e-01) (1.701254e-07, 2.211971e-01) (2.728333e-07, 2.205012e-01) (4.375479e-07, 2.193981e-01) (7.017038e-07, 2.176615e-01) (1.125336e-06, 2.149566e-01) (1.804722e-06, 2.108120e-01) (2.894266e-06, 2.046150e-01) (4.641589e-06, 1.956724e-01) (7.443803e-06, 1.833849e-01) (1.193777e-05, 1.675369e-01) (1.914482e-05, 1.485745e-01) (3.070291e-05, 1.276496e-01) (4.923883e-05, 1.063451e-01) (7.896523e-05, 8.621245e-02) (1.266380e-04, 6.842284e-02) (2.030918e-04, 5.360325e-02) (3.257021e-04, 4.184911e-02) (5.223345e-04, 3.286432e-02) (8.376776e-04, 2.617905e-02) (1.343399e-03, 2.129774e-02) (2.154435e-03, 1.777732e-02) (3.455107e-03, 1.524869e-02) (5.541020e-03, 1.342993e-02) (8.886238e-03, 1.210373e-02) (1.425103e-02, 1.111890e-02) (2.285464e-02, 1.036436e-02) (3.665241e-02, 9.723473e-03) (5.878016e-02, 9.087431e-03) (9.426685e-02, 8.421136e-03) (1.511775e-01, 7.968666e-03) (2.424462e-01, 8.789750e-03) (3.888155e-01, 2.305196e-02) (6.235507e-01, 1.342159e-01) (1.000000e+00, 5.917121e-01)
            };

        % h=0.3
        \addplot[
            color=red,
            mark = square*,
            mark size=1.5pt,
            line width=1pt
            ]
            coordinates {
                (1.000000e-08, 2.151690e-01) (1.603719e-08, 2.150373e-01) (2.571914e-08, 2.148267e-01) (4.124626e-08, 2.144900e-01) (6.614741e-08, 2.139532e-01) (1.060818e-07, 2.131000e-01) (1.701254e-07, 2.117510e-01) (2.728333e-07, 2.096359e-01) (4.375479e-07, 2.063616e-01) (7.017038e-07, 2.013912e-01) (1.125336e-06, 1.940609e-01) (1.804722e-06, 1.836850e-01) (2.894266e-06, 1.697843e-01) (4.641589e-06, 1.523896e-01) (7.443803e-06, 1.322406e-01) (1.193777e-05, 1.106938e-01) (1.914482e-05, 8.937394e-02) (3.070291e-05, 6.977891e-02) (4.923883e-05, 5.296089e-02) (7.896523e-05, 3.936884e-02) (1.266380e-04, 2.892070e-02) (2.030918e-04, 2.118211e-02) (3.257021e-04, 1.558916e-02) (5.223345e-04, 1.160895e-02) (8.376776e-04, 8.799827e-03) (1.343399e-03, 6.820616e-03) (2.154435e-03, 5.434610e-03) (3.455107e-03, 4.465031e-03) (5.541020e-03, 3.792175e-03) (8.886238e-03, 3.327173e-03) (1.425103e-02, 3.000615e-03) (2.285464e-02, 2.763315e-03) (3.665241e-02, 2.575243e-03) (5.878016e-02, 2.402008e-03) (9.426685e-02, 2.216547e-03) (1.511775e-01, 2.069777e-03) (2.424462e-01, 2.125649e-03) (3.888155e-01, 2.794032e-03) (6.235507e-01, 7.781993e-03) (1.000000e+00, 5.254775e-02)
            };

        % ref h=0.2
        \addplot[color=blue, dashed]
            coordinates{
                (1.000000e-08, 4.354997e-03) (1, 4.354997e-03) 
            };

        % ref h=0.15
        \addplot[color=orange, dashed]
            coordinates{
                (1.000000e-08, 9.652975e-03) (1, 9.652975e-03) 
            };

        % ref h=0.3
        \addplot[color=red, dashed]
            coordinates{
                (1.000000e-08, 2.465000e-03) (1.000000e+00, 2.465000e-03)
            };

        \addplot[color = black, dashed]
        coordinates{(0.125, 0.001) (0.125, 0.1)
        };

        \node [black] at (rel axis cs: 
            0.76, 0.73) { $H=8$ };

        \node [orange] at (rel axis cs: 
            0.915, 0.88) { $h=0.3$ };

        \node [blue] at (rel axis cs: 
            0.915, 0.77) { $h=0.2$ };

        \node [red] at (rel axis cs: 
            0.92, 0.62) { $h=0.15$ };

        \node [black] at (rel axis cs: 
            0.25, 0.26) { Regular meshes };

        \end{axis}
    \end{tikzpicture} 
    \caption{Optimal choice of the $H$ parameter for a Laplacian stress test.}
    \label{fig:optimal_choice}
\end{figure}
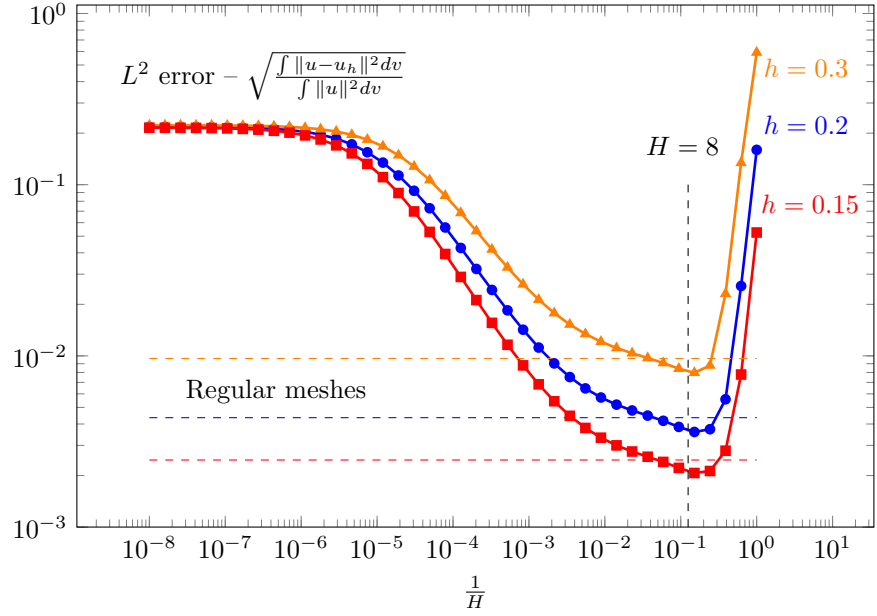

%% file: physics.tex
\section{Applications to physical problems}
\label{sec:physics}
In the seminal TFEM paper \cite{quiriny2024tempered}, the focus was primarily on the theoretical foundations of the method, with numerical illustrations limited to simplified configurations aimed at demonstrating its convergence properties. In the present section, we move beyond these preliminary results and investigate the practical performance of TFEM on a range of physically relevant problems. To this end, we consider four representative applications: incompressible fluid flows, Cahn–Hilliard phase-field dynamics, wave propagation, and vibro-acoustic fluid–structure interactions in the frequency domain.

These test cases have been selected to span a wide variety of mathematical and physical features, including incompressibility constraints, nonlinear and higher-order operators, transient wave phenomena, and coupled multiphysics interactions across interfaces. For each problem, we deliberately introduce sliver elements into the computational mesh in order to assess the impact of severe element distortion on the numerical solution. The results obtained with TFEM are systematically compared to those of standard finite element discretizations on the same meshes. Through these examples, we demonstrate that TFEM consistently delivers accurate and stable solutions in situations where classical FEM suffers from locking and loss of accuracy, thereby highlighting the robustness and versatility of the approach across a broad spectrum of applications. 

All incompressible Navier–Stokes simulations presented in this section are performed using the MigFlow solver \cite{CONSTANT2019213}, while the remaining applications—namely the Cahn–Hilliard, wave propagation, and vibro-acoustic problems—are implemented using AutoFreeFem \cite{Allaire2024} for automated code generation, and subsequently solved with FreeFEM \cite{MR3043640}. This combination allows us to rely on well-established and efficient computational frameworks while ensuring a consistent and reproducible implementation of both classical FEM and TFEM formulations across all considered physical models.
\subsection{Incompressible fluid flow}
In this example, we solve the incompressible Navier-Stokes equations on a 3D domain with slivers. The Navier-Stokes equations describe the motion of viscous fluids and are given by
\begin{align}
\rho \left( \frac{\partial \mathbf{u}}{\partial t} + \mathbf{u} \cdot \nabla \mathbf{u} \right) &= -\nabla p + \mu \nabla^2 \mathbf{u} + \rho \mathbf{g} \\
\nabla \cdot \mathbf{u} &= 0
\end{align}
where $\mathbf{u}$ is the velocity field, $p$ is the pressure, $\rho$ is the density, $\mu$ is the dynamic viscosity, and $\mathbf{g}$ is the gravity. \\
These equations are discretized in time using a simple implicit Euler scheme, and the spatial discretization is performed using P1 finite elements for both velocity and pressure. This spatial discretization is known to be unstable because it does not satisfy the inf-sup condition; we therefore stabilize the problem using a simple PSPG stabilization. The advection term with continuous finite elements can lead to oscillations, so we also stabilize it using a SUPG stabilization. The resulting linear system is solved using a direct solver.  \\ \\
\\
\\
We consider a simple Poiseuille flow in a cylindrical pipe. The domain is a cylinder of length $L = 2 \, [m]$ and radius $R=0.4 \, [m]$. The boundary conditions are no-slip on the pipe walls and a parabolic velocity profile at the inlet. The outlet is treated with a simple outflow condition with fixed pressure, and gravity is set to zero. This problem has an analytical solution, which allows us to compare the results obtained with TFEM and traditional FEM through a convergence study.\\
We compare the results obtained with classical FEM and TFEM on two deliberately degraded meshes. In the first configuration, a spherical region of sliver elements is artificially introduced at the center of the pipe. This sphere is constructed as described in Section~\ref{sec:tfem}, with radius $r=\frac{1}{2}R$. In the second configuration, the degradation is more severe: an entire central section of the pipe is replaced by fully degenerated elements.

The results are reported in Figure~\ref{fig:poiseuille_comparison}. In both configurations, the velocity profile obtained with TFEM remains in excellent agreement with the analytical parabolic solution. Conversely, the classical FEM fails to accurately represent the solution within the degenerated regions. In the presence of sliver elements, the strong constraints induced by the classical formulation drastically reduce the effective number of degrees of freedom available to the solution, thereby preventing an accurate representation of the parabolic velocity profile and leading to significant errors. This limitation becomes even more pronounced when an entire section of the pipe is affected, resulting in a severely degraded solution.

By contrast, TFEM retains its ability to capture the correct velocity profile even in highly distorted regions, demonstrating its robustness with respect to mesh quality. The convergence study presented in Figure \ref{fig:NS_conv} confirms that TFEM achieves the expected convergence rates for the velocity field at the outlet of the pipe in both configurations, whereas the classical FEM fails to converge.

These results highlight the robustness of TFEM for incompressible flow problems in the presence of severely distorted or degenerate elements, in contrast with the standard FEM, which may lose both accuracy and convergence.
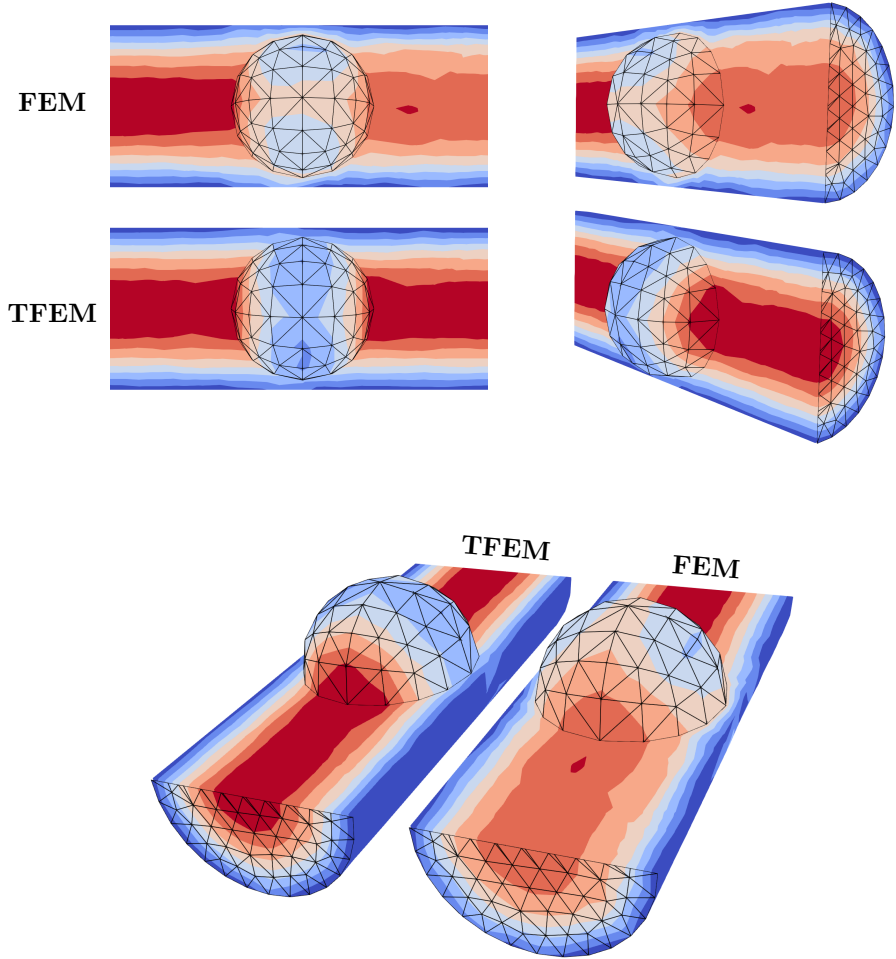
\begin{figure}[H]
    \centering
    \begin{tikzpicture}
        \node[anchor=center] at (0,0) {
        \includegraphics[width=0.45\textwidth,trim={600 0 500 0},clip]{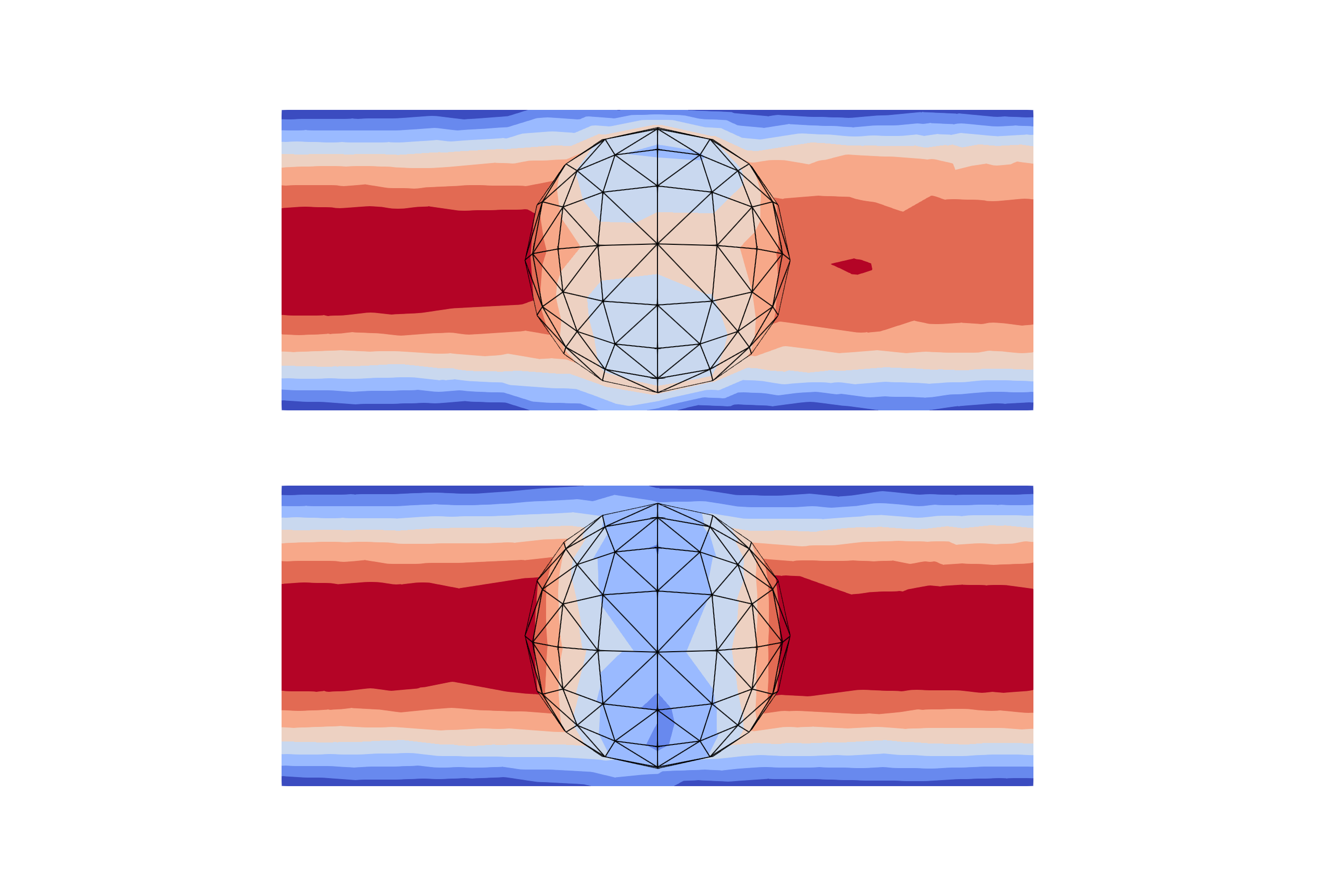}
        };
        \node[anchor=center] at (-3.5,1.4) {\textbf{FEM}};
        \node[anchor=center] at (-3.5,-1.4) {\textbf{TFEM}};
        \node[anchor=center] at (5,0) {
        \includegraphics[width=0.45\textwidth,trim={600 0 500 0},clip]{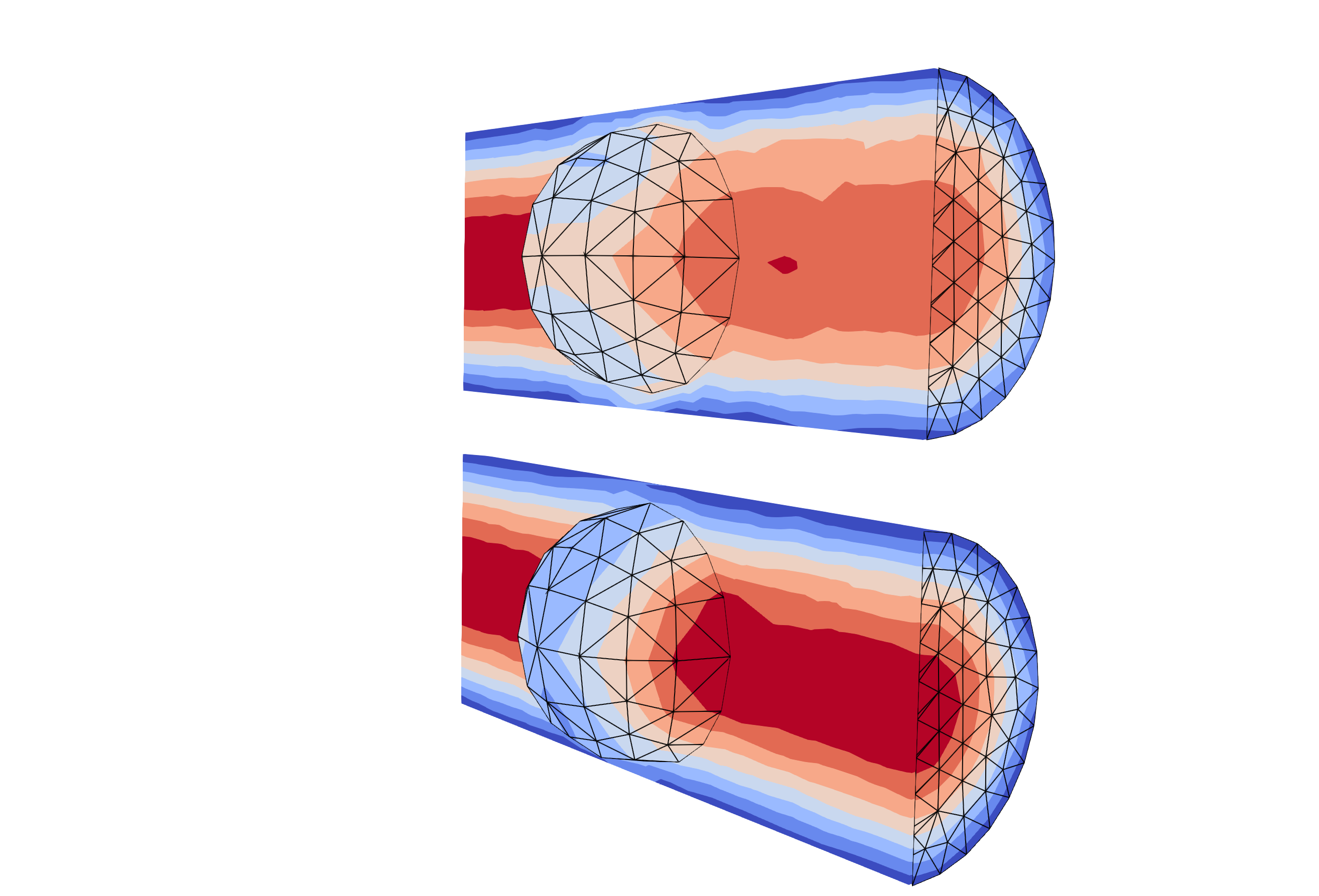}
        };
        
        \node[anchor=center] at (2.5,-7) {
        \includegraphics[width=\textwidth,trim={0 0 0 300},clip]{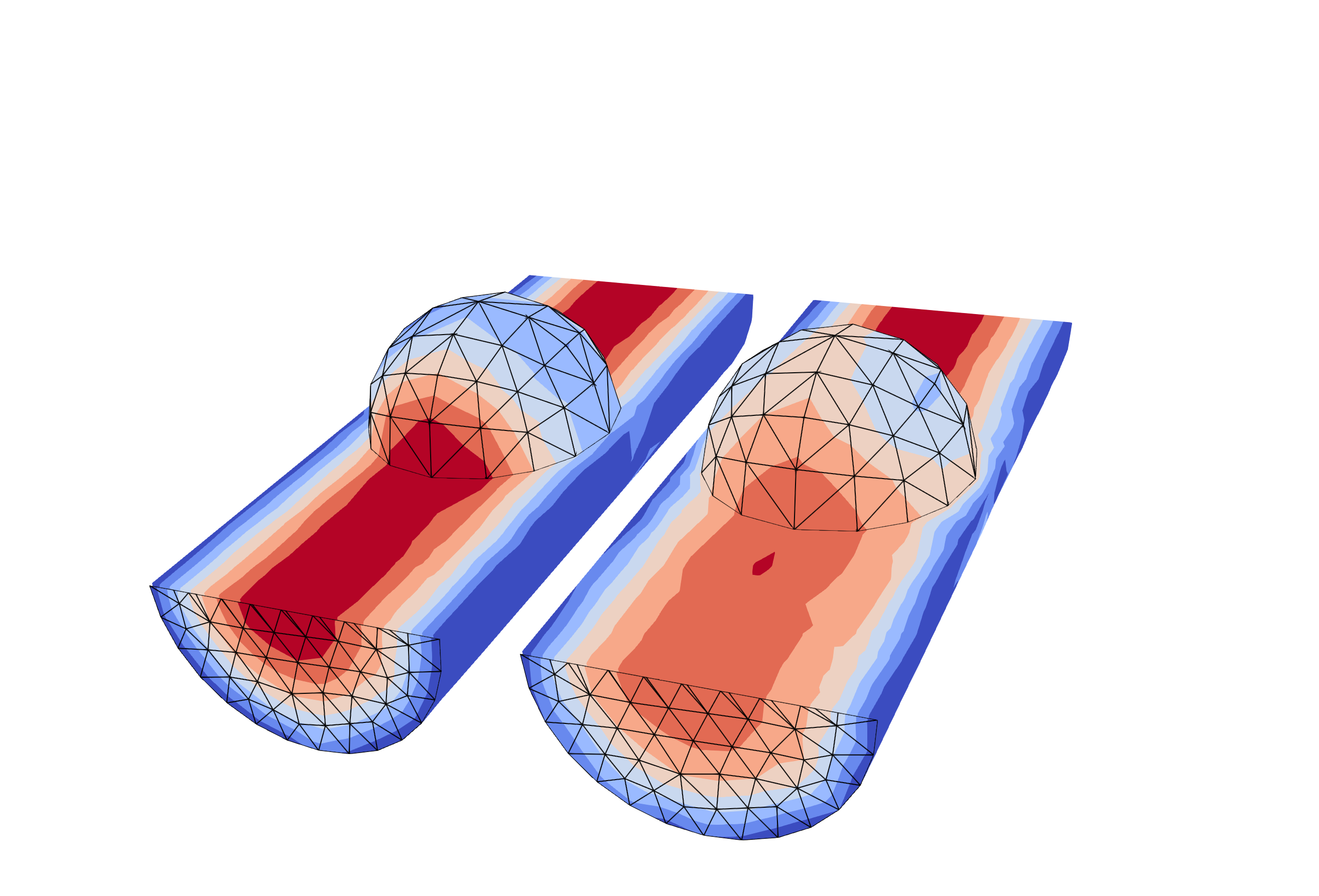}
        };
        \node[anchor=center, rotate=-5] at (5.15,-4.75) {\textbf{FEM}};
        \node[anchor=center, rotate=-5] at (2.5,-4.5) {\textbf{TFEM}};
    \end{tikzpicture}
    \caption{Multiple views of the velocity field for the Poiseuille flow problem in a cylindrical pipe with a sphere of slivers. Solution obtained with traditional FEM and with TFEM on the same mesh.}
    \label{fig:poiseuille_comparison}
\end{figure}

\input{fluid_convergence}

\newpage
\subsection{Cahn-Hilliard phase-field problem}
In this example, we consider the Cahn–Hilliard phase-field problem in a three-dimensional domain containing sliver elements. The computational mesh is constructed by inserting a thin sheet of slivers in the middle of a cubic domain, as illustrated in Fig.\ref{fig:CH_wave_mesh}. The Cahn–Hilliard equation is a nonlinear, fourth-order partial differential equation that models phase separation phenomena in binary mixtures and has become a standard benchmark in diffuse-interface modeling. Its intrinsic multiscale nature, combining sharp interfacial regions with smooth bulk phases, makes it particularly well-suited for assessing the robustness and accuracy of numerical discretizations on distorted meshes.

\begin{figure}[H]
    \centering
    \includegraphics[width=\textwidth]{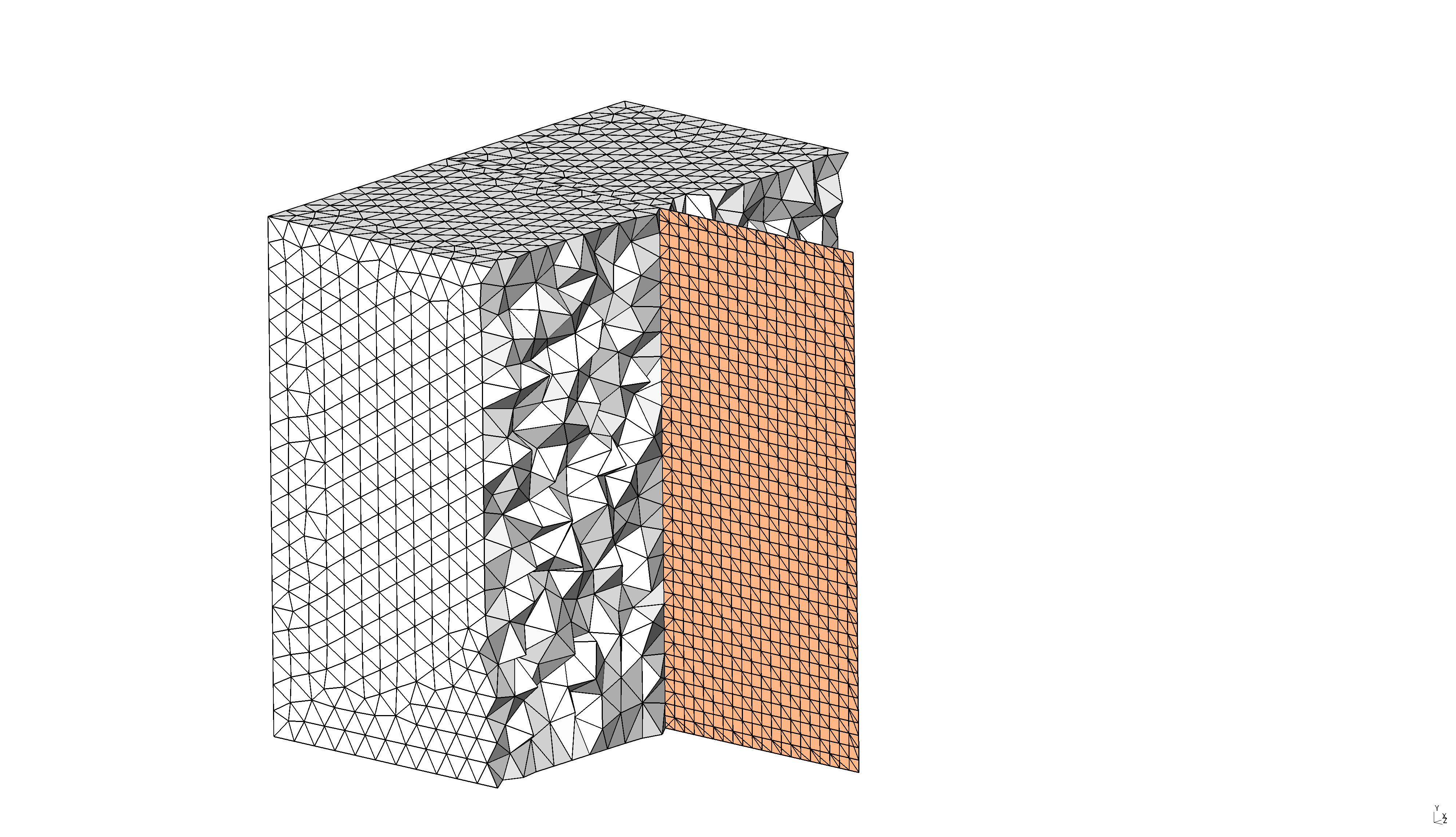}
    \caption{Mesh of a 3D cube with a sheet of slivers in the middle.}
    \label{fig:CH_wave_mesh}
\end{figure}

From a methodological standpoint, this problem provides a stringent test for the proposed TFEM approach. The presence of fourth-order derivatives (treated here via a mixed formulation) and the strong coupling between the concentration and chemical potential fields require an accurate representation of gradients and curvature. As such, any numerical locking induced by poorly shaped elements—such as slivers—can significantly degrade the solution quality. Demonstrating accurate and stable results for the Cahn–Hilliard system therefore supports the claim that TFEM remains reliable across a broad class of problems, including nonlinear, higher-order, and physically relevant applications.

The governing equations are written in mixed form as
\begin{align}
\frac{\partial c}{\partial t} &= \nabla \cdot (M \nabla \mu), \\
\mu &= \frac{\partial f}{\partial c} - \kappa \nabla^2 c,
\end{align}
where $c$ denotes the concentration field, $\mu$ the chemical potential, $M$ the mobility, $f$ the free-energy density, and $\kappa$ the gradient energy coefficient. In this work, we adopt a standard double-well potential of the form
\begin{equation}
f(c) = \frac{1}{4} c^2 (1 - c)^2.
\end{equation}
Time integration is performed using an implicit Euler scheme, while space is discretized using linear (P1) finite elements for both the concentration and the chemical potential. The resulting nonlinear system is solved using a Newton–Raphson method, and the linearized systems at each iteration are solved using a direct solver. The same mesh is used to compare the results obtained with traditional FEM and with TFEM, allowing us to isolate the effect of the slivers on the solution quality.\\

We initialize the concentration field with random values between 0 and 1, and we let the system evolve in time. Figure \ref{fig:cahn_hilliard_time} presents a series of snapshots of the Cahn–Hilliard simulation at successive times. The top row corresponds to the solution obtained using a standard finite element method (FEM), while the bottom row shows the result computed with the proposed TFEM approach on the same mesh, which contains a band of sliver elements. A clear qualitative difference is observed: in the classical FEM case, the solution is artificially constrained and exhibits a nearly linear variation across the sliver region, indicating numerical locking. In contrast, the TFEM solution preserves smooth variations and allows for curvature along the sheet of slivers, demonstrating its ability to accurately represent the underlying physics despite the presence of highly distorted elements.

\begin{figure}[H]
    \centering
    \begin{tikzpicture}
        \node[anchor=center] at (0,0) {
        \includegraphics[width=\textwidth,trim={200 375 200 525},clip]{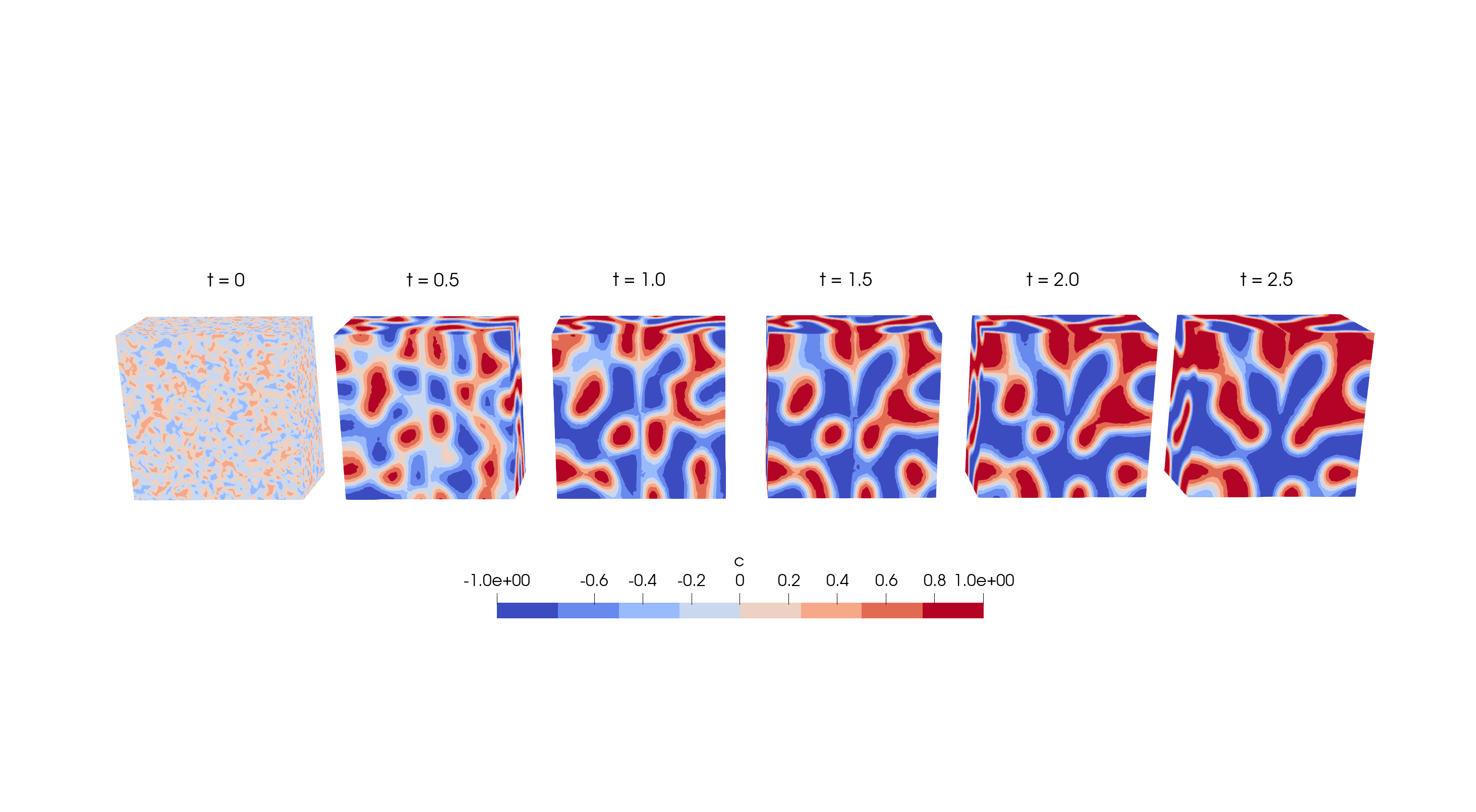}
        };
        \draw[fill=white, draw=none] (-6,1.3) rectangle (6,2);
       \draw[fill=white, draw=none] (-4,-2) rectangle (4,-1);
    \end{tikzpicture}
    \\ \vspace{-1cm}
    \begin{tikzpicture}
        \node[anchor=center] at (0,0) {
        \includegraphics[width=\textwidth,trim={200 375 200 525},clip]{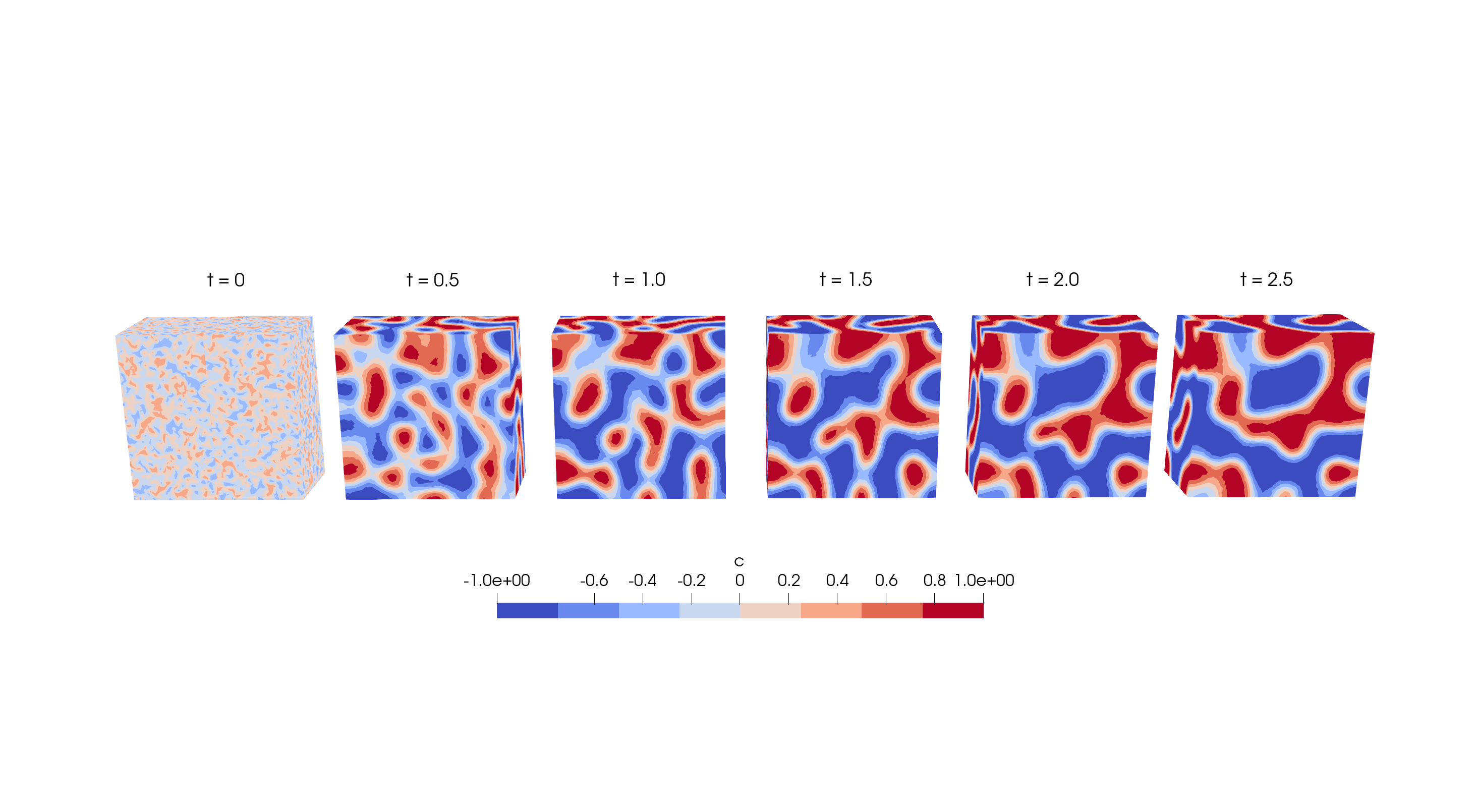}
        };
         \draw[fill=white, draw=none] (-6,1.3) rectangle (6,2);
         \draw[fill=white, draw=none] (-4,-2) rectangle (4,-1);
        \node[anchor=center] at (0,-1.2) {
         \includegraphics[width=0.5\textwidth]{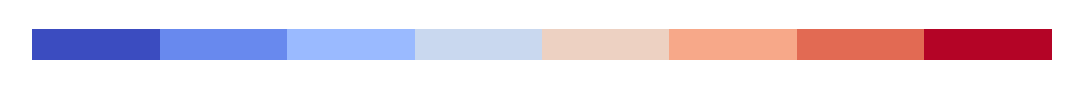}
        };
        \node[anchor=center] at (0,-1.7) {$c$};
        \node[anchor=center] at (-2.75,-1.7) {$-1$};
        \node[anchor=center] at (2.7,-1.7) {$1$};

        \node[anchor=center] at (-5,2.0) {$t=0$};
        \node[anchor=center] at (-3,2.0) {$t=0.5$};
        \node[anchor=center] at (-1,2.0) {$t=1.0$};
        \node[anchor=center] at (1,2.0) {$t=1.5$};
        \node[anchor=center] at (3,2.0) {$t=2.0$};
        \node[anchor=center] at (5,2.0) {$t=2.5$};
    \end{tikzpicture}
    \caption{Time evolution of the concentration field for the Cahn-Hilliard phase-field problem in a 3D domain with slivers. (top) locking of the solution in the sliver region with traditional FEM. (bottom) solution obtained with TFEM on the same mesh.}
    \label{fig:cahn_hilliard_time}
\end{figure}
Figure \ref{fig:cahn_hilliard_mesh_coupes} further highlights the locking phenomenon by showing planar cuts of the solution at time $t=2$, taken at $x=0.25$, $x=0.5$, and $x=0.75$. The middle slice ($x=0.5$) intersects the sheet of sliver elements, while the other two slices lie in regions of well-shaped elements. In the standard FEM solution (top row), a pronounced locking effect is clearly visible in the central cut, where the concentration field exhibits an artificial, nearly linear profile across the sliver region, in stark contrast with the surrounding patterns. This behavior is not observed in the slices away from the slivers, indicating that it is purely mesh-induced. By contrast, the TFEM solution (bottom row) remains smooth and physically consistent across all cut planes, including at $x=0.5$, where curvature and fine-scale structures are preserved despite the presence of highly distorted elements. These observations emphasize the ability of TFEM to eliminate locking and recover an accurate representation of the phase field even in the presence of slivers.
\begin{figure}[H]
    \centering
    \begin{tikzpicture}
        \node[anchor=center] at (0,0) {
        \includegraphics[width=0.8\textwidth]{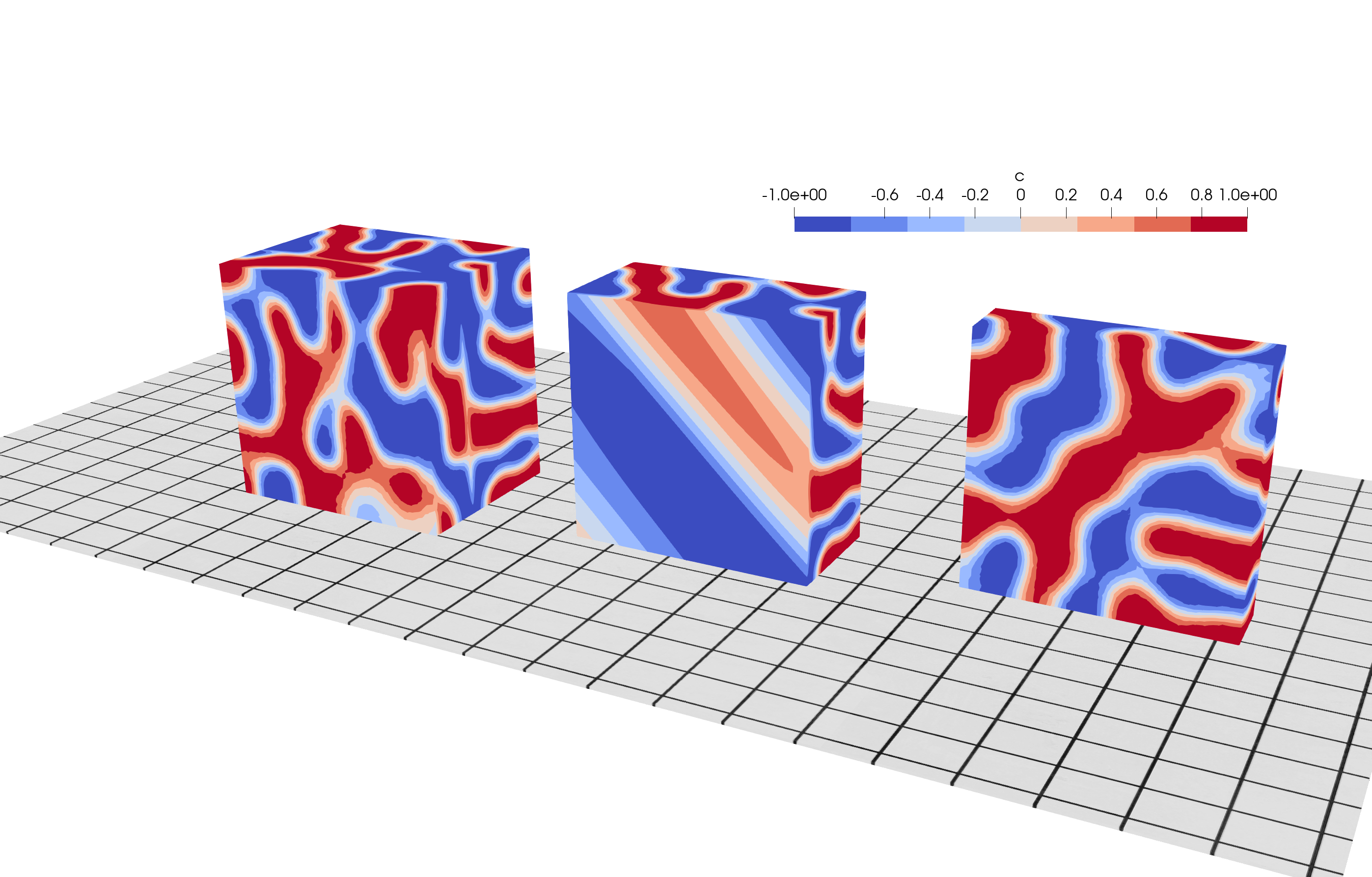}
        };
         \draw[fill=white, draw=none] (0.5,1.3) rectangle (5,2);
        \node[anchor=center] at (2.75,1.65) {
         \includegraphics[width=0.3\textwidth]{images/colorbar.png}
        };
        \node[anchor=center] at (2.75,2.0) {$c$};
        \node[anchor=center] at (1.,2.0) {$-1$};
        \node[anchor=center] at (4.25,2.0) {$1$};
    \end{tikzpicture}
    \\
    \begin{tikzpicture}
        \node[anchor=center] at (0,0) {
        \includegraphics[width=0.8\textwidth]{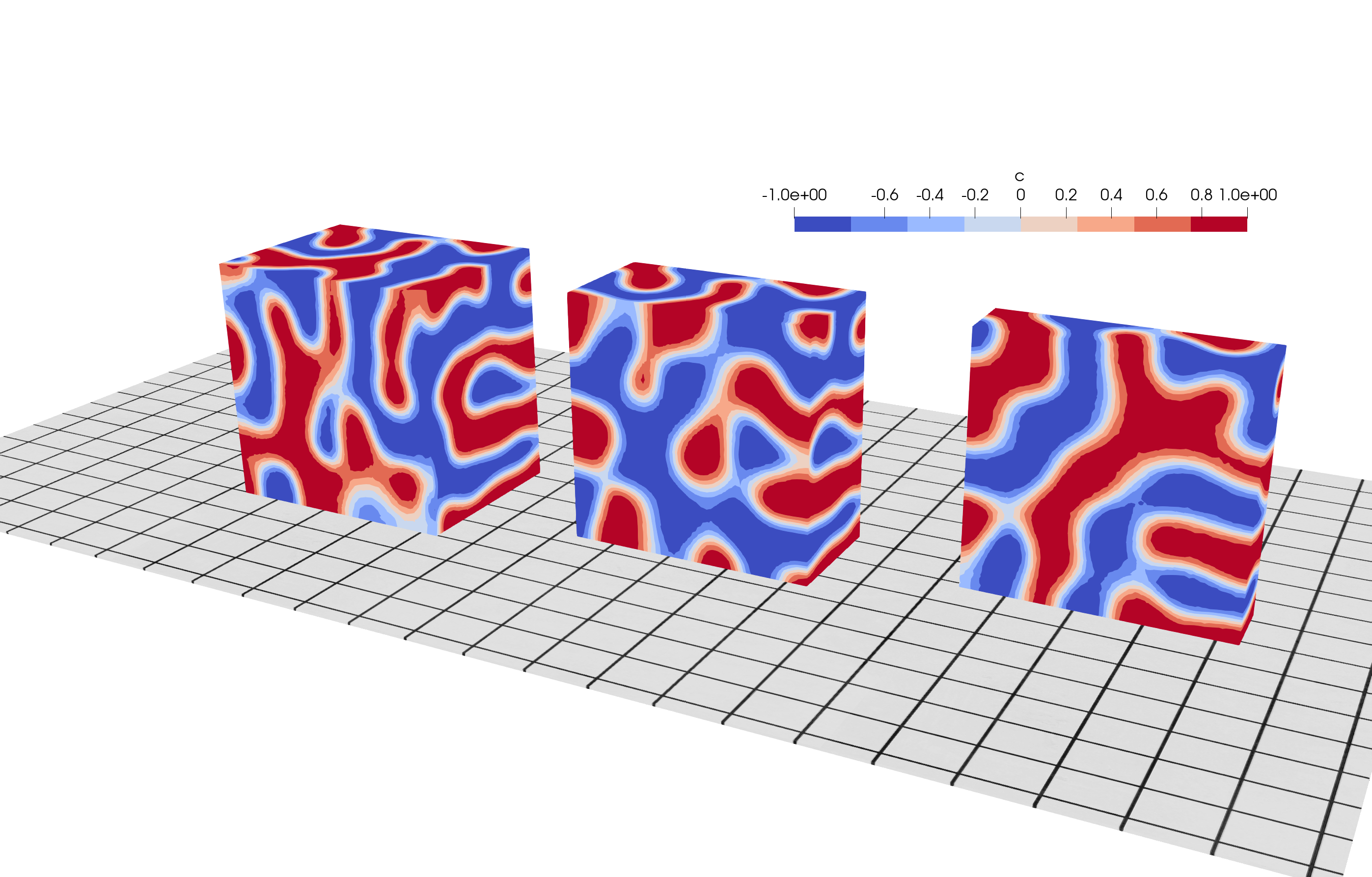}
        };
        \draw[fill=white, draw=none] (0.5,1.3) rectangle (5,2);
        \node[anchor=center] at (2.75,1.65) {
         \includegraphics[width=0.3\textwidth]{images/colorbar.png}
        };
        \node[anchor=center] at (2.75,2.0) {$c$};
        \node[anchor=center] at (1.,2.0) {$-1$};
        \node[anchor=center] at (4.25,2.0) {$1$};
    \end{tikzpicture}
    \caption{Cahn-Hilliard phase-field problem in a 3D domain with slivers. (top) locking of the solution in the sliver region with traditional FEM. (bottom) solution obtained with TFEM on the same mesh at time $t=2$.}
    \label{fig:cahn_hilliard_mesh_coupes}
\end{figure}

\newpage
\subsection{Wave equation}
Next, we consider the classical wave equation in a three-dimensional domain containing sliver elements, using the same mesh configuration as in the Cahn–Hilliard example (Fig.~\ref{fig:CH_wave_mesh}). The wave equation is a second-order hyperbolic partial differential equation that describes the propagation of disturbances in a medium and serves as a canonical model for transient dynamics. In contrast to the diffusive nature of the Cahn–Hilliard system, it involves the propagation of sharp features over time, making it a good complementary test for numerical methods.

From the perspective of the proposed TFEM framework, this problem is particularly relevant as it probes the ability of the method to accurately capture wave propagation in the presence of poorly shaped elements. In standard finite element discretizations, sliver elements can induce spurious stiffness and numerical locking, leading to artificial wave distortion, loss of amplitude, or incorrect propagation speeds. Demonstrating robust performance for the wave equation therefore provides further evidence that TFEM can handle a wide range of physical regimes, including dynamic, transport-dominated problems.

The governing equation is given by
\begin{equation}
\frac{\partial^2 u}{\partial t^2} = c^2 \nabla^2 u,
\end{equation}
where $u$ denotes the wave field and $c$ is the wave speed. For the purpose of time integration, the equation is recast as a first-order system by introducing the auxiliary velocity field $v$:
\begin{equation}
\begin{cases}
\frac{\partial u}{\partial t} = v, \\
\frac{\partial v}{\partial t} = c^2 \nabla^2 u.
\end{cases}
\end{equation}
Time discretization is performed using an implicit Euler scheme, while space is discretized using linear (P1) finite elements for both fields.\\
\\
Figure \ref{fig:wave_time} shows the time evolution of the wave field for an initial Gaussian profile centered in the domain. The top row corresponds to the reference solution obtained on a regular mesh, while the middle and bottom rows present results on the mesh containing a sheet of sliver elements, using standard FEM and TFEM, respectively. A clear degradation of the solution is observed in the classical FEM case: as the wave propagates, strong locking effects appear in the vicinity of the sliver region, leading to distorted wave fronts compared to the reference solution. In contrast, the TFEM results on the sliver mesh remain in excellent agreement with those obtained on the regular mesh. The two solutions are almost indistinguishable throughout the simulation. This comparison highlights the severe impact of sliver-induced locking in standard FEM for transient wave problems, and demonstrates the ability of TFEM to recover accurate wave propagation even on highly distorted meshes.

\begin{figure}[H]
    \centering
    \begin{tikzpicture}
        \node[anchor=center] at (0,0) {
        \includegraphics[width=\textwidth,trim={170 375 145 525},clip]{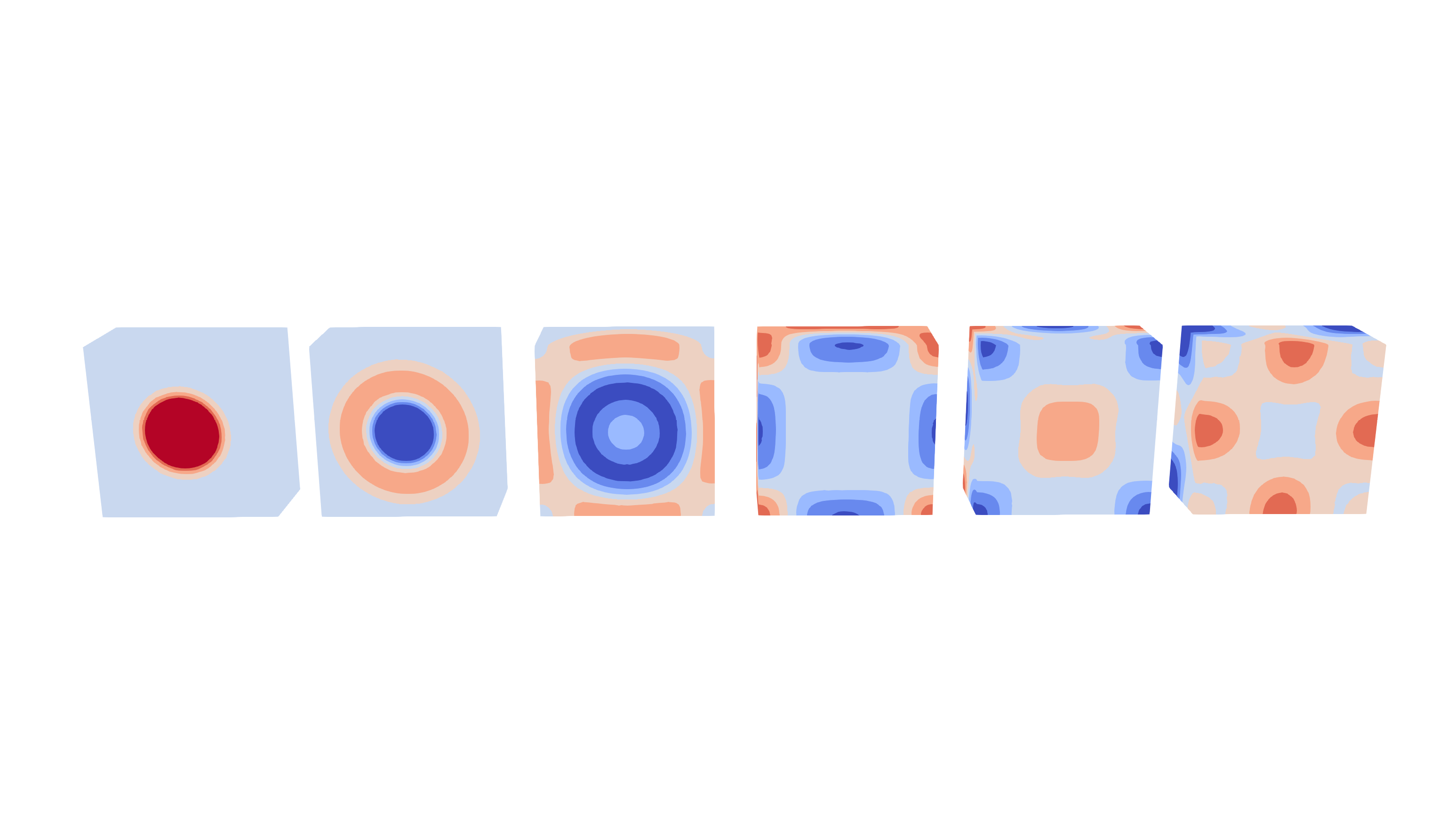}
        };
    \end{tikzpicture}
    \\ \vspace{-1cm}
    \begin{tikzpicture}
        \node[anchor=center] at (0,0) {
        \includegraphics[width=\textwidth,trim={170 375 145 525},clip]{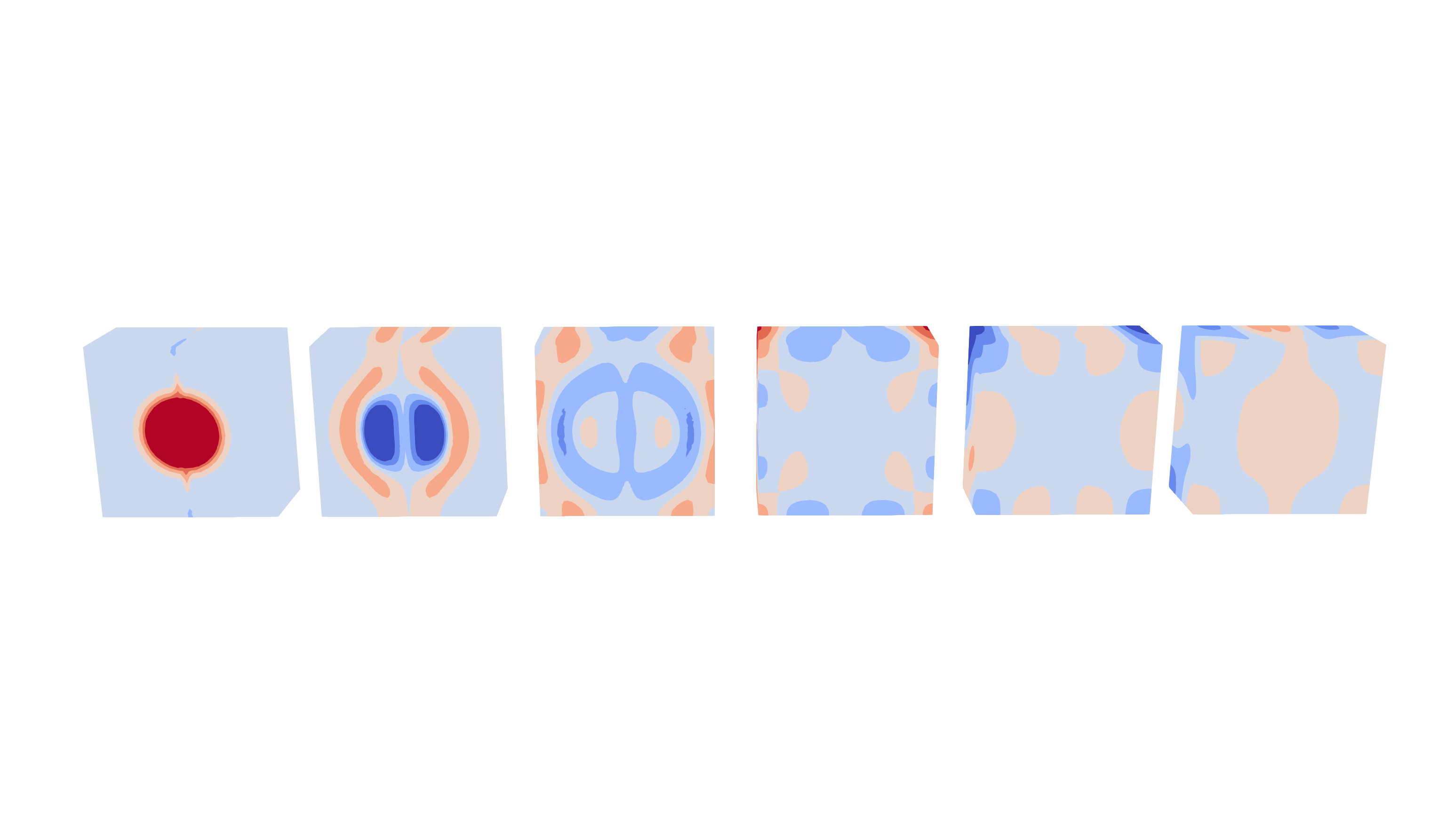}
        };

        \node[anchor=center] at (-5,2.0) {$t=0$};
        \node[anchor=center] at (-3,2.0) {$t=0.2$};
        \node[anchor=center] at (-1,2.0) {$t=0.4$};
        \node[anchor=center] at (1,2.0) {$t=0.6$};
        \node[anchor=center] at (3,2.0) {$t=0.8$};
        \node[anchor=center] at (5,2.0) {$t=1.0$};
    \end{tikzpicture}
    \\ \vspace{-1.2cm}
    \begin{tikzpicture}
        \node[anchor=center] at (0,0) {
        \includegraphics[width=\textwidth,trim={170 650 145 525},clip]{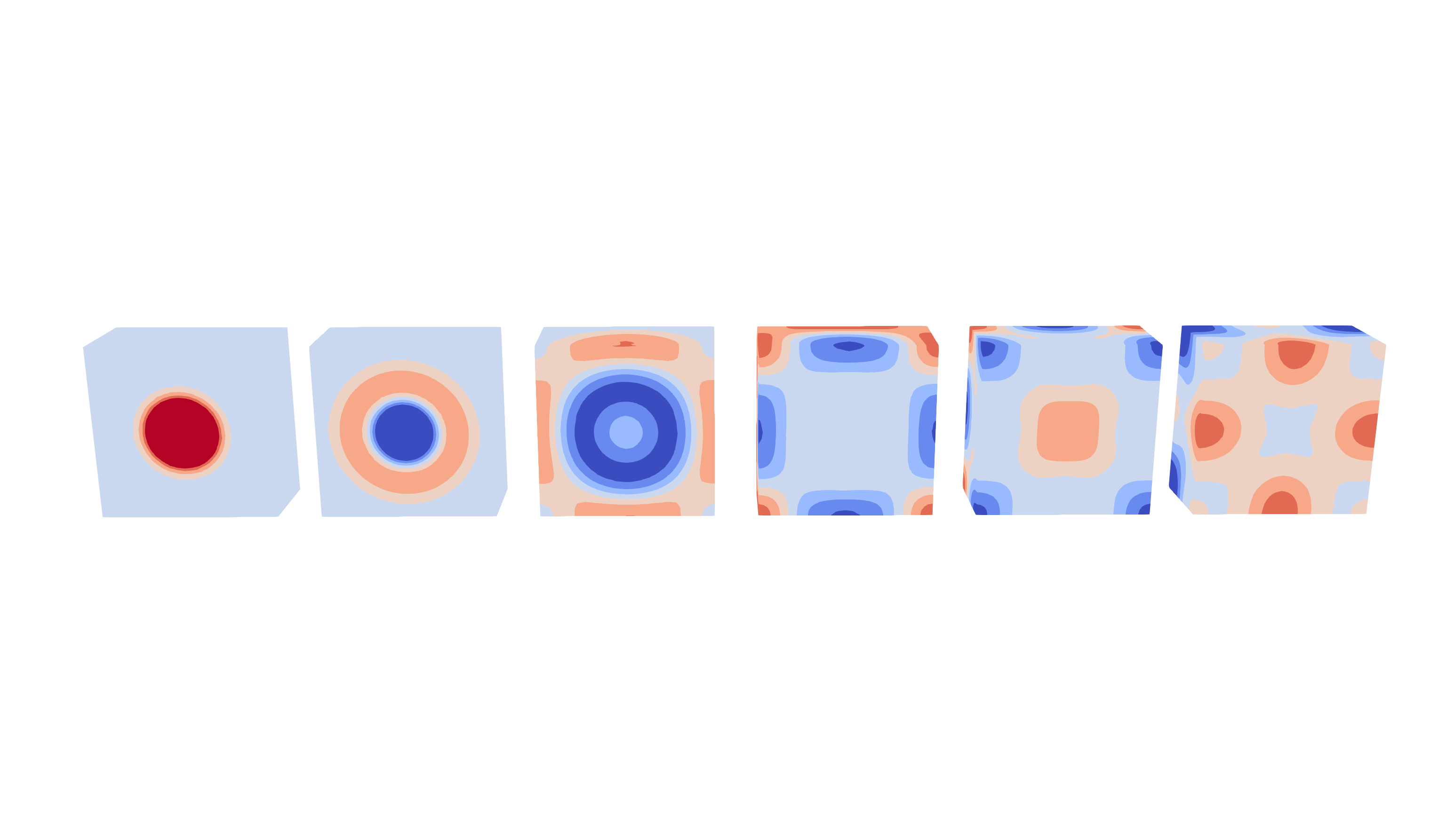}
        };
        \node[anchor=center] at (0,1.15) {
         \includegraphics[width=0.5\textwidth]{images/colorbar.png}
        };
        \node[anchor=center] at (0,1.5) {$u$};
        \node[anchor=center] at (-2.75,1.5) {$-0.1$};
        \node[anchor=center] at (2.7,1.5) {$0.15$};
    \end{tikzpicture}
    \caption{Time evolution of the wave field in a 3D domain on a regular mesh (top), on a mesh with slivers using traditional FEM (middle), and on the same mesh with TFEM (bottom).}
    \label{fig:wave_time}
\end{figure}
Figure \ref{fig:wave_slivers_cut_TFEM} provides further insight into the impact of sliver elements by showing planar cuts of the solution at time $t = 1$, taken at $x = 0.25$, $x = 0.5$, and $x = 0.75$. In contrast to the Cahn–Hilliard case, the degradation induced by standard FEM is no longer confined to the immediate vicinity of the sliver sheet. Instead, the perturbation propagates throughout the domain: even the slices at $x = 0.25$ and $x = 0.75$, which lie away from the sliver region, exhibit noticeable differences from the reference solution. The central slice at $x = 0.5$, intersecting the sheet of slivers, is particularly affected, with the solution exhibiting an almost linear profile characteristic of severe numerical locking. By contrast, the TFEM solution on the sliver mesh remains in excellent agreement with the reference solution across all cut planes.

\begin{figure}[H]
    \centering
    \begin{tikzpicture}
        \node[anchor=center] at (0,0) {
        \includegraphics[width=0.8\textwidth,trim={100 200 100 200},clip]{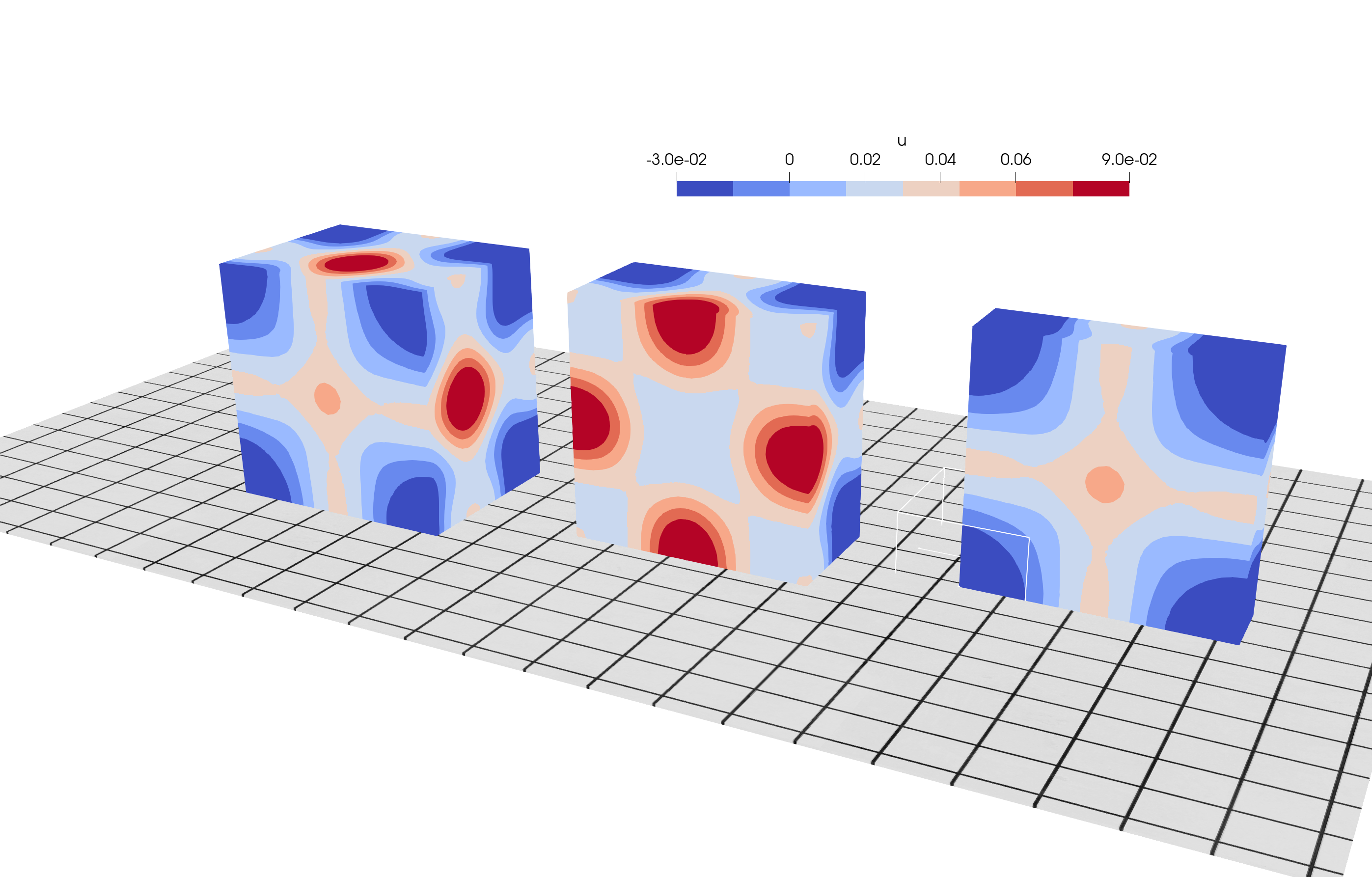}
        };
         \draw[fill=white, draw=none] (-0.5,1.3) rectangle (5,2.5);
        \node[anchor=center] at (2.75,1.65) {
         \includegraphics[width=0.3\textwidth]{images/colorbar.png}
        };
        \node[anchor=center] at (2.75,2.0) {$u$};
        \node[anchor=center] at (1.,2.0) {$-0.03$};
        \node[anchor=center] at (4.25,2.0) {$0.09$};
    \end{tikzpicture}
    \\
    \begin{tikzpicture}
        \node[anchor=center] at (0,0) {
        \includegraphics[width=0.8\textwidth,trim={100 200 100 200},clip]{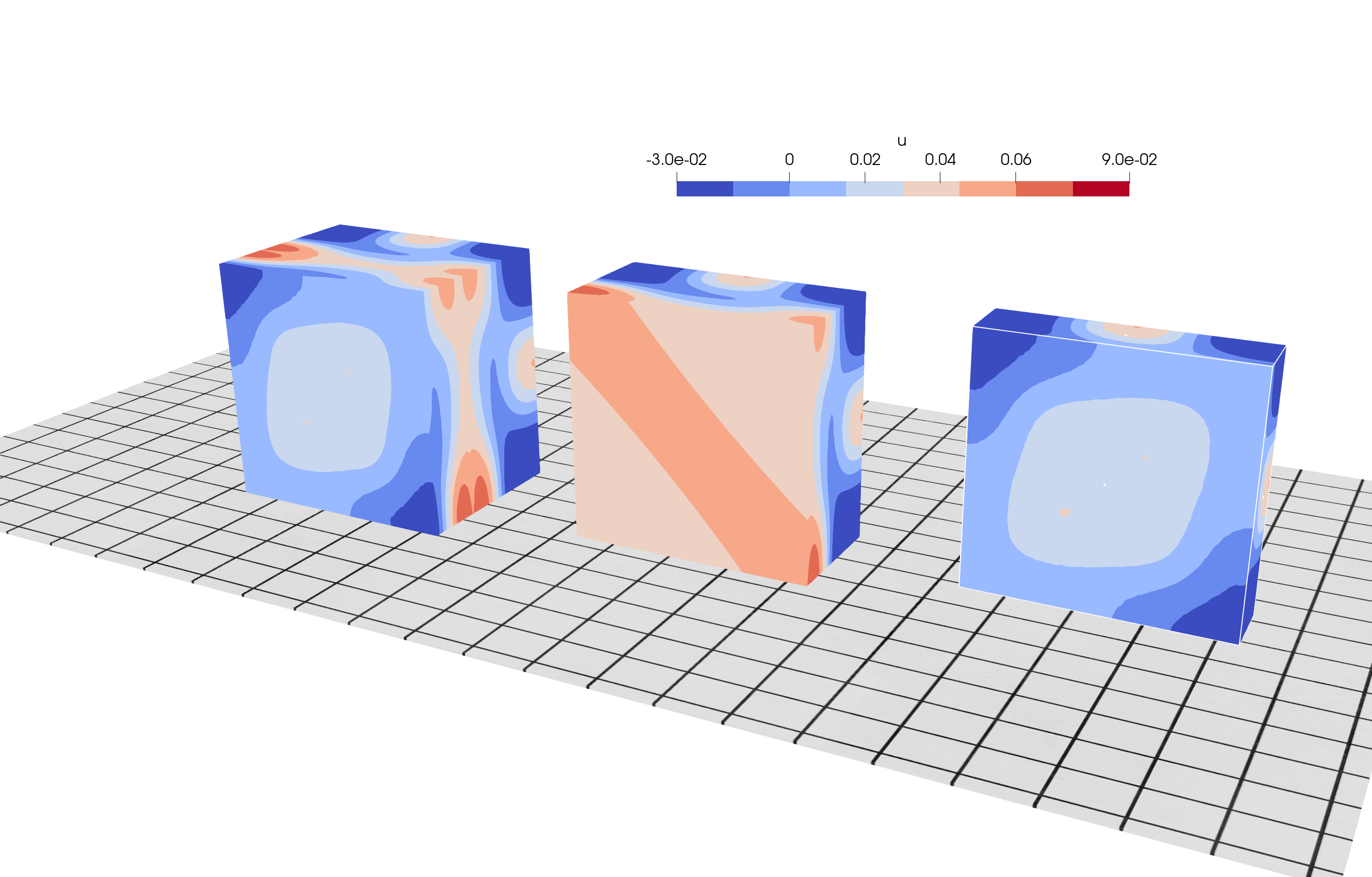}
        };
         \draw[fill=white, draw=none] (-0.5,1.3) rectangle (5,2.5);
        \node[anchor=center] at (2.75,1.65) {
         \includegraphics[width=0.3\textwidth]{images/colorbar.png}
        };
        \node[anchor=center] at (2.75,2.0) {$u$};
        \node[anchor=center] at (1.,2.0) {$-0.03$};
        \node[anchor=center] at (4.25,2.0) {$0.09$};
    \end{tikzpicture}
    \\
    \begin{tikzpicture}
        \node[anchor=center] at (0,0) {
        \includegraphics[width=0.8\textwidth,trim={100 200 100 200},clip]{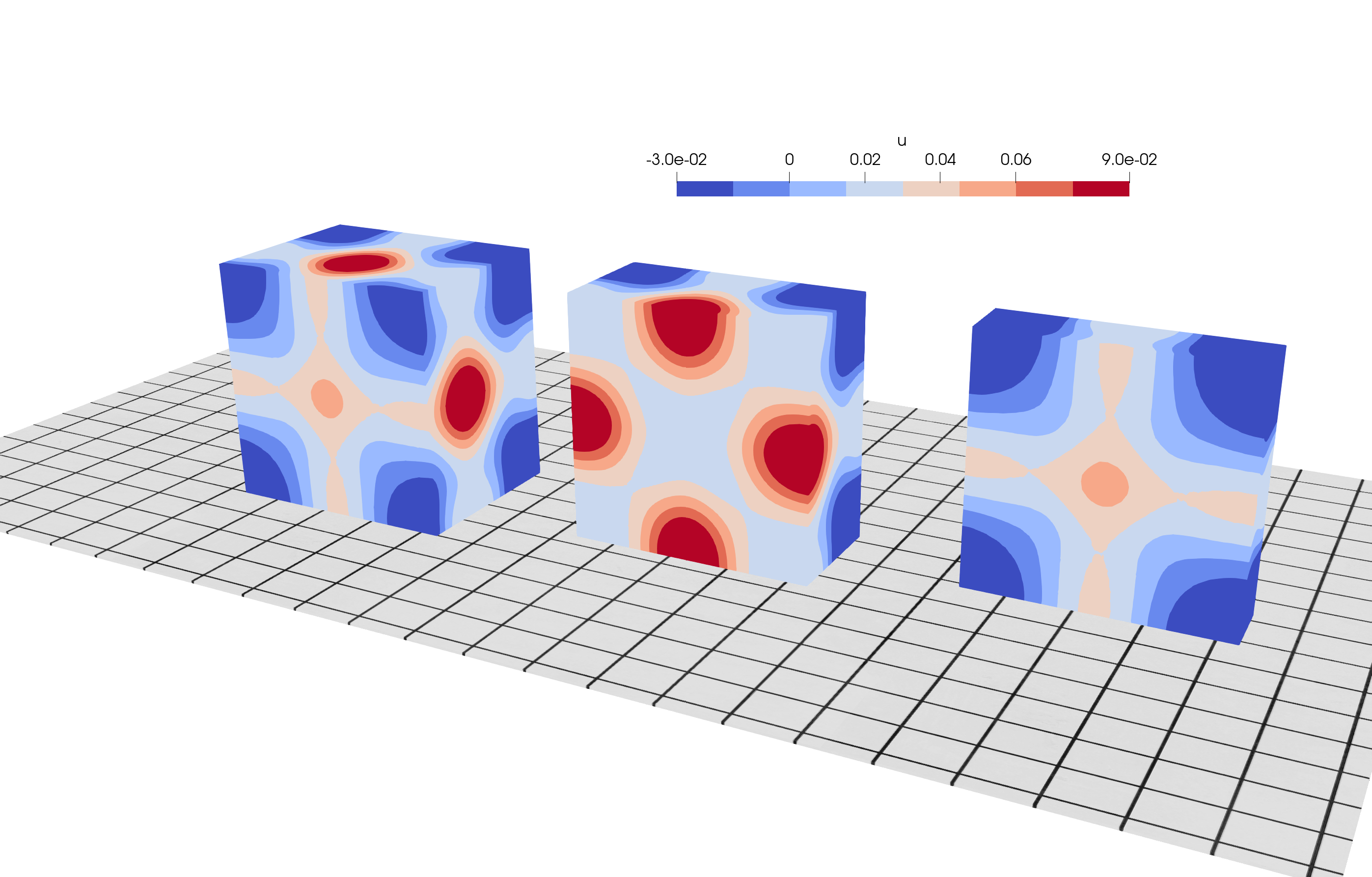}
        };
         \draw[fill=white, draw=none] (-0.5,1.3) rectangle (5,2.5);
        \node[anchor=center] at (2.75,1.65) {
         \includegraphics[width=0.3\textwidth]{images/colorbar.png}
        };
        \node[anchor=center] at (2.75,2.0) {$u$};
        \node[anchor=center] at (1.,2.0) {$-0.03$};
        \node[anchor=center] at (4.25,2.0) {$0.09$};
    \end{tikzpicture}
    \caption{Wave propagation in a 3D domain on a regular mesh (top), on a mesh with slivers using traditional FEM (middle), and on the same mesh with TFEM (bottom).}
    \label{fig:wave_slivers_cut_TFEM}
\end{figure}

\subsection{Vibro-acoustic harmonic study}
To further assess the robustness of the proposed TFEM formulation in a multiphysics context, we consider a vibro-acoustic problem involving fluid–structure interaction in the presence of sliver elements. This class of problems is particularly relevant as it combines different physical models across an interface, where accurate transmission of stresses and kinematics is critical. From a numerical standpoint, interfaces are known to be especially sensitive to mesh quality, and the presence of sliver elements can severely deteriorate the accuracy of standard finite element approximations. This makes vibro-acoustic coupling an ideal and stringent test case to demonstrate that TFEM remains reliable not only for single-physics problems, but also across coupled systems representative of practical engineering applications.

The computational domain consists of a $5\,\text{m} \times 1\,\text{m} \times 1\,\text{m}$ rectangular acoustic canal $\Omega_f$, in which waves propagate, and a linear elastic beam $\Omega_s$ of length $0.9\,\text{m}$, width $0.4\,\text{m}$, and thickness $0.1\,\text{m}$, located at the center of the canal, as shown in Fig.\ref{fig:vibro_mesh}. The mesh is deliberately constructed by inserting sheets of sliver elements along the fluid–structure interface $\Gamma_{fs} = \partial \Omega_s \cap \partial \Omega_f$, thereby introducing a challenging discretization in the region where coupling conditions must be enforced.

\begin{figure}[H]
    \centering
    \includegraphics[width=\textwidth]{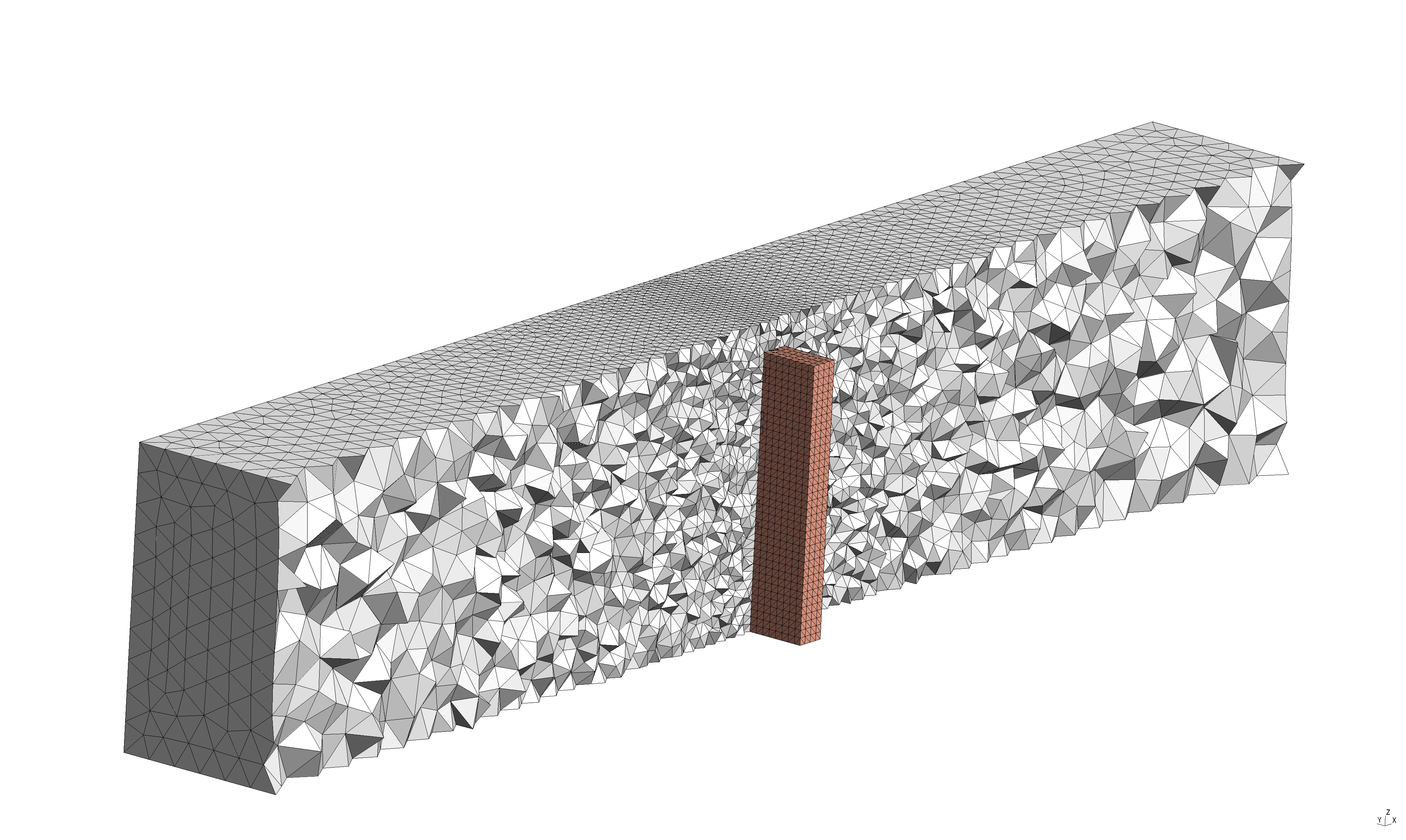}
    \caption{Mesh of the vibro-acoustic problem with sliver elements along the fluid–structure interface.}
    \label{fig:vibro_mesh}
\end{figure}

The unknowns of the problem are the structural displacement field $\mathbf{u}(\mathbf{x},t)$ in $\Omega_s$ and the acoustic pressure field $p(\mathbf{x},t)$ in $\Omega_f$. In the solid, the motion is governed by the linear elastodynamic equations
\begin{equation}
\rho_s \frac{\partial^2 \mathbf{u}}{\partial t^2} - \nabla \cdot \boldsymbol{\sigma}(\mathbf{u}) = \mathbf{0} \qquad \text{in } \Omega_s,
\end{equation}
where $\rho_s$ is the density of the solid and $\boldsymbol{\sigma}$ is the Cauchy stress tensor given by the linear isotropic constitutive law
\begin{equation}
\boldsymbol{\sigma}(\mathbf{u}) = \lambda \, (\nabla \cdot \mathbf{u}) \, \mathbf{I} + 2\mu \, \boldsymbol{\varepsilon}(\mathbf{u}),
\end{equation}
with the strain tensor
\begin{equation}
\boldsymbol{\varepsilon}(\mathbf{u}) = \frac{1}{2}\left( \nabla \mathbf{u} + \nabla \mathbf{u}^\top \right).
\end{equation}

In the fluid, assuming a compressible, inviscid medium under small perturbations, the acoustic pressure satisfies
\begin{equation}
\frac{1}{K_f} \frac{\partial^2 p}{\partial t^2} - \nabla^2 p = 0 \qquad \text{in } \Omega_f,
\end{equation}
where $K_f$ denotes the bulk modulus of the fluid.

At the fluid–structure interface $\Gamma_{fs}$, coupling conditions enforce both dynamic equilibrium and kinematic compatibility. The continuity of normal stresses yields
\begin{equation}
\boldsymbol{\sigma}(\mathbf{u}) \, \mathbf{n} = - p \, \mathbf{n} \qquad \text{on } \Gamma_{fs},
\end{equation}
while the continuity of normal accelerations reads
\begin{equation}
\frac{\partial^2 \mathbf{u}}{\partial t^2} \cdot \mathbf{n} = - \frac{1}{\rho_f} \nabla p \cdot \mathbf{n} \qquad \text{on } \Gamma_{fs}.
\end{equation}

We restrict the analysis to time-harmonic oscillations at a prescribed angular frequency $\omega$, assuming all fields vary as
\begin{equation}
\mathbf{u}(\mathbf{x},t) = \Re \left( \hat{\mathbf{u}}(\mathbf{x}) e^{i \omega t} \right),
\qquad
p(\mathbf{x},t) = \Re \left( \hat{p}(\mathbf{x}) e^{i \omega t} \right),
\end{equation}
which transforms time derivatives according to $\partial^2 / \partial t^2 \rightarrow -\omega^2$. The resulting frequency-domain problem is therefore
\begin{equation}
\begin{cases}

 \omega^2 \rho_s \, \hat{\mathbf{u}} - \nabla \cdot \boldsymbol{\sigma}(\hat{\mathbf{u}}) &= \mathbf{0} \qquad \text{in } \Omega_s, \\
 \nabla^2 \hat{p} - \dfrac{\omega^2 \rho_f}{K_f} \hat{p} &= 0 \qquad \text{in } \Omega_f.
  \end{cases}
  \end{equation}
  The corresponding frequency-domain coupling conditions on ($\Gamma_{fs}$) are
\begin{equation}
\begin{cases}
\boldsymbol{\sigma}(\hat{\mathbf{u}}) , \mathbf{n} = - \hat{p}  \mathbf{n},
& \qquad \text{on } \Gamma_{fs}, \\
\omega^2 \hat{\mathbf{u}} \cdot \mathbf{n}
= \dfrac{1}{\rho_f} \nabla \hat{p} \cdot \mathbf{n},
& \qquad \text{on } \Gamma_{fs}.
\end{cases}
\end{equation}
The system is excited by a prescribed harmonic pressure at the inlet and outlet of the canal, generating an acoustic wave that interacts with the elastic beam. Homogeneous Dirichlet boundary conditions are imposed on the base of the structure. The problem is discretized using linear (P1) finite elements for both the acoustic pressure and structural displacement, and solved in a fully monolithic manner, ensuring a consistent treatment of the coupling at the interface. The resulting linear system is solved using a direct solver.\\
\\
\begin{figure}[H]
    \centering
    \begin{tikzpicture}
        \node[anchor=center] at (0,0) {
        \includegraphics[width=0.8\textwidth,trim={600 100 680 230},clip]{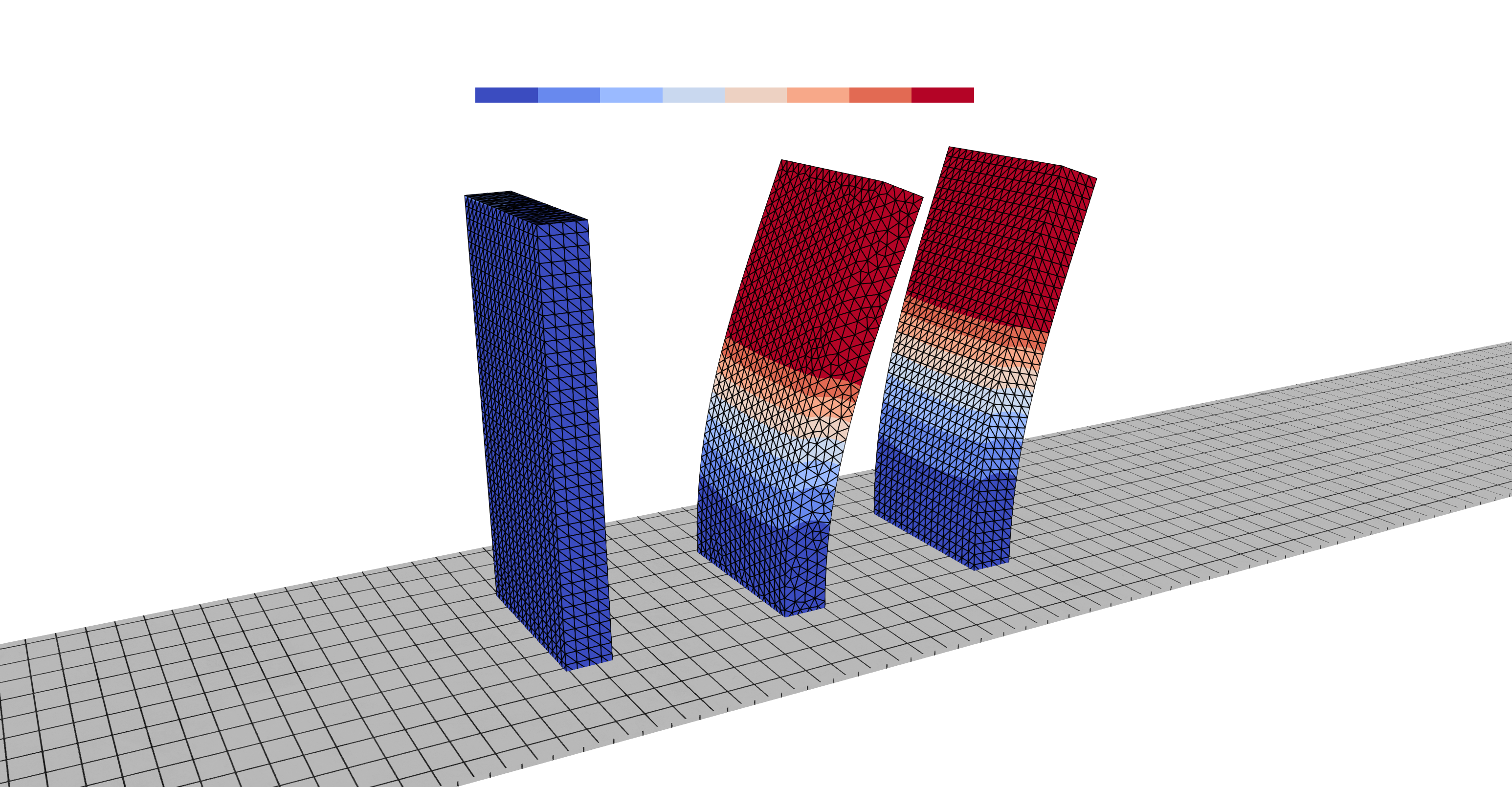}
        };
    \end{tikzpicture}
    \\ \vspace{0.5cm}
    \begin{tikzpicture}
        \node[anchor=center] at (0,0) {
        \includegraphics[width=0.8\textwidth,trim={600 100 680 230},clip]{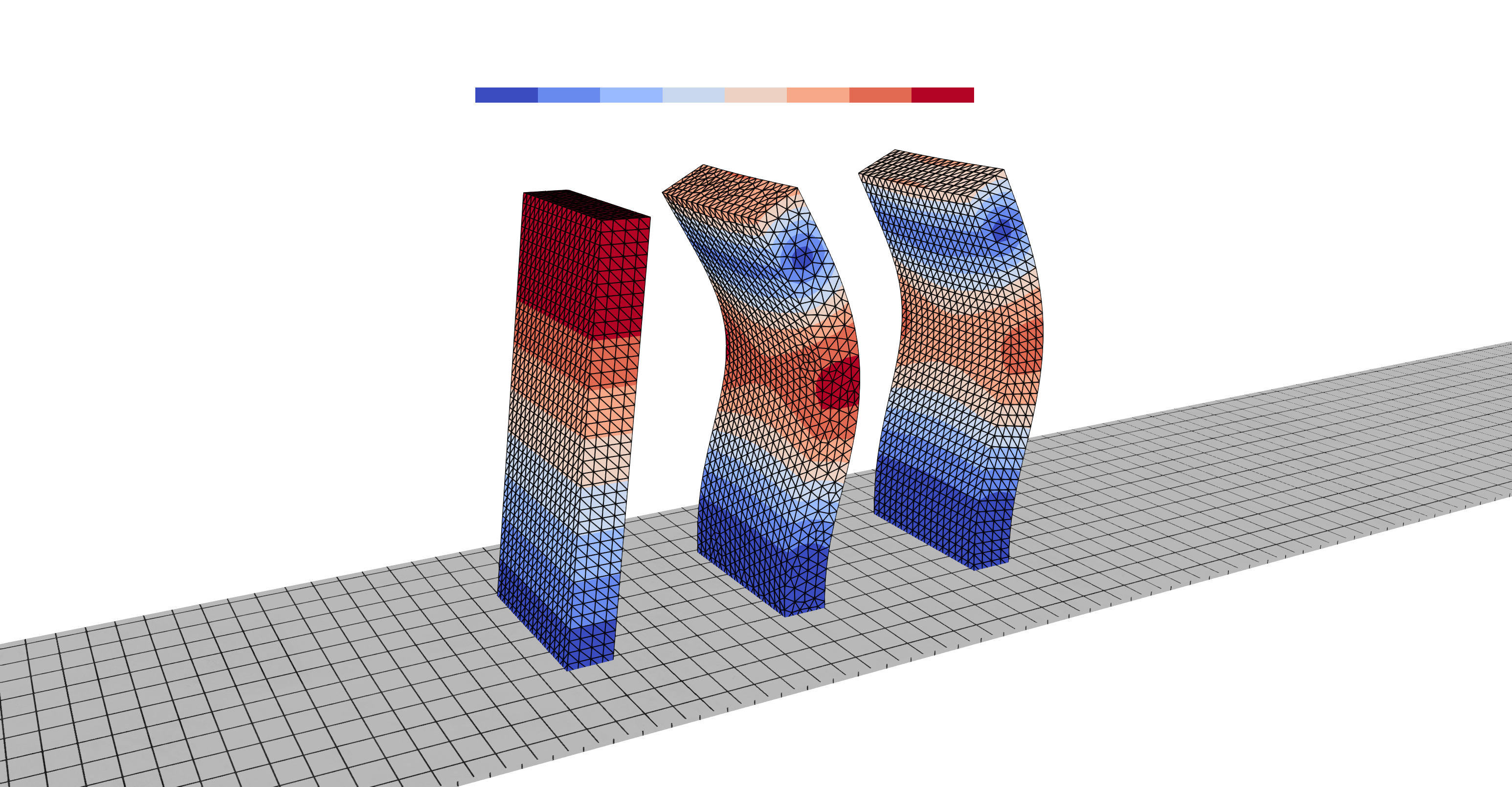}
        };
        \node[anchor=center] at (0,3.7) {
         \includegraphics[width=0.5\textwidth]{images/colorbar.png}
        };
        \node[anchor=center] at (0,4.05) {$\|\boldsymbol{u}\|$};
        \node[anchor=center] at (-2.75,4.05) {$0$};
        \node[anchor=center] at (2.7,4.05) {$10^{-8}$};
    \end{tikzpicture}
    \caption{Beam displacement field for the vibro-acoustic problem at 1000 Hz (top row) and 5000 Hz (bottom row). The columns show standard FEM on a mesh with slivers (left), standard FEM on a regular mesh (center), and TFEM on the same mesh with slivers (right).}
    \label{fig:vibro_comparison}
\end{figure}
Figure \ref{fig:vibro_comparison} shows the structural displacement field for two excitation frequencies, $\omega=1000$ Hz and $\omega=5000$ Hz, for three different configurations: a reference solution on a regular mesh (center), a mesh containing slivers discretized with standard FEM (left), and the same sliver mesh using TFEM (right). For visualization purposes, the deformation is amplified by a factor of $10^7$, as the problem is set in the small-deformation regime.\\
\\
A clear discrepancy is observed in the classical FEM solution on the sliver mesh, where the deformation pattern is significantly altered compared to the reference configuration. The structural response appears overly stiff and is forced to be linear. In contrast, the TFEM solution on the same mesh allows the structure to exhibit curvature and matches the reference solution obtained on the regular mesh for both frequencies. The deformation profiles, including both amplitude distribution and bending shape, are recovered, demonstrating that TFEM effectively preserves the coupled vibro-acoustic response despite the presence of slivers at the fluid–structure interface.\\
\\
To further quantify the impact of sliver-induced mesh distortion on the vibro-acoustic response, we perform a harmonic analysis based on two scalar indicators extracted from the structural displacement field. The first quantity is the mean displacement at the top of the beam $u_{\text{top}}$, which provides a global measure of the structural response amplitude under harmonic excitation. The second quantity is a curvature-based indicator $\kappa$ computed along a vertical line at the mid-span of the structure, which serves as a sensitive measure of the locking phenomenon induced by sliver elements. In the presence of locking, the solution is forced to be linear in the sliver region, which leads to a significant reduction in curvature and an artificial stiffening of the response. By contrast, a well-resolved solution should exhibit non-zero curvature along the beam, reflecting the physical bending behavior under harmonic loading.\\
\\
To compute this second metric, the displacement field is first interpolated along a vertical line located at the mid-span of the structure, i.e., at fixed coordinates $x = 2.5$ and $y = 0.5$, using a one-dimensional sampling of $u_x(2.5, 0.5, z)$. The resulting discrete profile is then used to approximate the second derivative of the $x$-displacement in the vertical direction via a standard finite-difference stencil, providing an estimate of the local curvature. A global curvature indicator is finally obtained by integrating the squared curvature along the line. \\
\\
This post-processing strategy allows us to systematically compare the reference solution, the standard FEM solution on sliver meshes, and the TFEM formulation in a unified manner across excitation frequencies. \\
\\
\input{vibro_acoustic_3D_graph}
Figure \ref{fig:vibro_graph} reports the frequency response of the system in terms of the two previously defined indicators. The top plot shows the evolution of the mean displacement at the top of the beam as a function of the excitation frequency, while the bottom plot presents the corresponding curvature indicator. In both cases, three configurations are compared: the reference solution obtained on a regular mesh, the standard FEM solution on a mesh containing slivers, and the TFEM solution on the same distorted mesh.\\
\\
A clear discrepancy is observed for the standard FEM discretization on the sliver mesh. The predicted response significantly deviates from the reference one over the entire frequency range. This behavior is consistent with the presence of locking at the fluid–structure interface, which artificially stiffens the structural response and alters the deformation patterns, ultimately leading to an incorrect estimation of both global and local quantities.\\
\\
In contrast, the TFEM results obtained on the sliver mesh follow the reference solution computed on the regular mesh. The overall trends, including the variations with frequency and the relative amplitude levels, are well captured for both the mean displacement and the curvature indicator. Although small discrepancies can still be observed, they remain limited and can be reasonably attributed to standard mesh-dependent effects rather than to pathological behavior induced by element distortion. These results confirm that TFEM effectively mitigates locking and provides a reliable prediction of both global and local response characteristics.

%% file: fluid_convergence.tex
\begin{figure}[!ht]
    \centering
    \begin{tikzpicture}
        \begin{axis}[
            xlabel={$h$},
            ylabel={$\| u_{h} - u_{\text{analytic}}\|_{L^2, \text{ out}}$},
            legend style={at={(0.65,0.4)},anchor=west},
            ymajorgrids=true,
            grid=none,
            ymode=log,
            xmode=log,
            xtick={-1, -2},
            xmin=1e-2,
            ymax=20,
            ymin=0.01,
            width=\textwidth,
            height=0.7\textwidth,
            y label style={at={(axis description cs:0.30,0.85)},rotate=-90,anchor=south},
            legend style={draw=none}
        ]
        %regular
        \addplot[
            color=blue,
            mark = *,
            mark size=1.5pt,
            line width=1pt
            ]
            coordinates {
                (0.1, 1.29952) (0.07071, 0.54604) (0.05, 0.29256) (0.03536, 0.141126) (0.025, 0.067251) (0.017676, 0.032479)
            };
        
        % monstre TFEM
        \addplot[
            color=orange,
            mark = triangle*,
            mark size=1.5pt,
            line width=1pt
            ]
            coordinates {
               (0.1, 1.596277) (0.07071, 0.77209) (0.05, 0.65507) (0.03536, 0.16122) (0.025, 0.07759) (0.017676, 0.0354567)
            };

        % monstre FEM
        \addplot[
            color=orange,
            mark = square*,
            mark size=1.5pt,
            line width=1pt
            ]
            coordinates {
                (0.1, 3.696452) (0.07071, 1.83864) (0.05, 1.81821) (0.03536, 0.63528) (0.025, 0.5707) (0.017676, 0.275354)
            };

        % plane TFEM
        \addplot[
            color=red,
            mark = triangle*,
            mark size=1.5pt,
            line width=1pt
            ]
            coordinates {
               (0.1, 1.258178) (0.07071, 0.573994) (0.05, 0.33170) (0.03536, 0.13917) (0.025, 0.072825) (0.017676, 0.03196)
            };

        % plane FEM
        \addplot[
            color=red,
            mark = square*,
            mark size=1.5pt,
            line width=1pt
            ]
            coordinates {
                (0.1, 9.71457) (0.07071, 8.30983) (0.05, 8.73795) (0.03536, 7.67285) (0.025, 6.346454) (0.017676, 2.533)
            };

        \addplot[color=black]
            coordinates{
                (0.1, 0.5) (0.035, 0.061256)
            };

        \node [orange, align=center] at (rel axis cs: 
            0.15, 0.47) { Sphere Slivers \\ FEM };

        \node [blue, align=center, rotate=23] at (rel axis cs: 
            0.25, 0.11) { Regular FEM };

        \node [red, align=center] at (rel axis cs: 
            0.15, 0.77) { Plane Slivers\\ FEM };

        \node [orange, align=center, rotate=23] at (rel axis cs: 
            0.2, 0.3) { Sphere Slivers TFEM };

        \node [red, align=center, rotate=23] at (rel axis cs: 
            0.22, 0.24) { Plane Slivers TFEM };

        \node [black, align=center, rotate=24] at (rel axis cs: 
            0.75, 0.35) { Slope 2 };

        \end{axis}
    \end{tikzpicture} 
    \caption{Convergence study of the 3D poiseuille problem in a pipe for FEM and TFEM for a mesh with a sphere of slivers, a mesh with a plane of slivers and a regular mesh.}
    \label{fig:NS_conv}
\end{figure}

%% file: vibro_acoustic_3D_graph.tex
\begin{figure}[!ht]
    \centering
    \begin{tikzpicture}
        \begin{axis}[
            xlabel={},
            ylabel={$u_{\text{top}}$},
            legend style={at={(0.65,0.4)},anchor=west},
            ymajorgrids=true,
            grid=none,
            ymode=log,
            width=\textwidth,
            height=0.6\textwidth,
            xmin=0, xmax=5000,
            xtick={-1000, 10000},
            y label style={at={(axis description cs:0.19,0.85)},rotate=-90,anchor=south},
            legend style={draw=none}
        ]
        %regular
        \addplot[
            color=blue,
            mark = triangle*,
            mark size=1pt
            ]
            coordinates {
               (50.000000000000, 0.000000004818) (100.000000000000, 0.000000004961) (150.000000000000, 0.000000005216) (200.000000000000, 0.000000005615) (250.000000000000, 0.000000006215) (300.000000000000, 0.000000007124) (350.000000000000, 0.000000008562) (400.000000000000, 0.000000011051) (450.000000000000, 0.000000016173) (500.000000000000, 0.000000032031) (550.000000000000, 0.000001100740) (600.000000000000, 0.000000029332) (650.000000000000, 0.000000014737) (700.000000000000, 0.000000009832) (750.000000000000, 0.000000007398) (800.000000000000, 0.000000005959) (850.000000000000, 0.000000005016) (900.000000000000, 0.000000004328) (950.000000000000, 0.000000004462) (1000.000000000000, 0.000000003681) (1050.000000000000, 0.000000003414) (1100.000000000000, 0.000000003233) (1150.000000000000, 0.000000003115) (1200.000000000000, 0.000000003048) (1250.000000000000, 0.000000003031) (1300.000000000000, 0.000000003063) (1350.000000000000, 0.000000003153) (1400.000000000000, 0.000000003314) (1450.000000000000, 0.000000003569) (1500.000000000000, 0.000000003957) (1550.000000000000, 0.000000004552) (1600.000000000000, 0.000000005508) (1650.000000000000, 0.000000007187) (1700.000000000000, 0.000000010701) (1750.000000000000, 0.000000022020) (1800.000000000000, 0.000000340912) (1850.000000000000, 0.000000019287) (1900.000000000000, 0.000000010015) (1950.000000000000, 0.000000006857) (2000.000000000000, 0.000000005289) (2050.000000000000, 0.000000004367) (2100.000000000000, 0.000000003776) (2150.000000000000, 0.000000003378) (2200.000000000000, 0.000000003949) (2250.000000000000, 0.000000002927) (2300.000000000000, 0.000000002786) (2350.000000000000, 0.000000002705) (2400.000000000000, 0.000000002670) (2450.000000000000, 0.000000002675) (2500.000000000000, 0.000000002718) (2550.000000000000, 0.000000002803) (2600.000000000000, 0.000000002932) (2650.000000000000, 0.000000003115) (2700.000000000000, 0.000000003364) (2750.000000000000, 0.000000003697) (2800.000000000000, 0.000000004135) (2850.000000000000, 0.000000004839) (2900.000000000000, 0.000000005671) (2950.000000000000, 0.000000006946) (3000.000000000000, 0.000000008936) (3050.000000000000, 0.000000012348) (3100.000000000000, 0.000000019238) (3150.000000000000, 0.000000039049) (3200.000000000000, 0.000000390328) (3250.000000000000, 0.000000060524) (3300.000000000000, 0.000000032192) (3350.000000000000, 0.000000024664) (3400.000000000000, 0.000000022459) (3450.000000000000, 0.000000023449) (3500.000000000000, 0.000000028873) (3550.000000000000, 0.000000049234) (3600.000000000000, 0.000002483190) (3650.000000000000, 0.000000036383) (3700.000000000000, 0.000000016444) (3750.000000000000, 0.000000009984) (3800.000000000000, 0.000000006888) (3850.000000000000, 0.000000005116) (3900.000000000000, 0.000000003992) (3950.000000000000, 0.000000003230) (4000.000000000000, 0.000000002687) (4050.000000000000, 0.000000002287) (4100.000000000000, 0.000000001984) (4150.000000000000, 0.000000001749) (4200.000000000000, 0.000000001564) (4250.000000000000, 0.000000001417) (4300.000000000000, 0.000000001299) (4350.000000000000, 0.000000001205) (4400.000000000000, 0.000000001129) (4450.000000000000, 0.000000001068) (4500.000000000000, 0.000000001021) (4550.000000000000, 0.000000000986) (4600.000000000000, 0.000000000964) (4650.000000000000, 0.000000000969) (4700.000000000000, 0.000000004960) (4750.000000000000, 0.000000003105) (4800.000000000000, 0.000000000946) (4850.000000000000, 0.000000000972) (4900.000000000000, 0.000000001030) (4950.000000000000, 0.000000001149) (5000.000000000000, 0.000000001922) (5050.000000000000, 0.000000000844) (5100.000000000000, 0.000000001125) (5150.000000000000, 0.000000001345) (5200.000000000000, 0.000000001635) (5250.000000000000, 0.000000002097) (5300.000000000000, 0.000000002982) (5350.000000000000, 0.000000005362) (5400.000000000000, 0.000000032653) (5450.000000000000, 0.000000007701) (5500.000000000000, 0.000000003406) (5550.000000000000, 0.000000002192) (5600.000000000000, 0.000000001627) (5650.000000000000, 0.000000001308) (5700.000000000000, 0.000000001114) (5750.000000000000, 0.000000000995) (5800.000000000000, 0.000000000938) (5850.000000000000, 0.000000000980) (5900.000000000000, 0.000000001550) (5950.000000000000, 0.000000000912) (6000.000000000000, 0.000000000403) (6050.000000000000, 0.000000000357) (6100.000000000000, 0.000000000376) (6150.000000000000, 0.000000000394) (6200.000000000000, 0.000000000414) (6250.000000000000, 0.000000000440) (6300.000000000000, 0.000000000480) (6350.000000000000, 0.000000000564) (6400.000000000000, 0.000000000714) (6450.000000000000, 0.000000001002) (6500.000000000000, 0.000000001710) (6550.000000000000, 0.000000004605) (6600.000000000000, 0.000000183745) (6650.000000000000, 0.000000042734) (6700.000000000000, 0.000000033366) (6750.000000000000, 0.000000002264) (6800.000000000000, 0.000000001543) (6850.000000000000, 0.000000001310) (6900.000000000000, 0.000000001470) (6950.000000000000, 0.000000001777) (7000.000000000000, 0.000000002383) (7050.000000000000, 0.000000003621) (7100.000000000000, 0.000000007067) (7150.000000000000, 0.000000046657) (7200.000000000000, 0.000000014025) (7250.000000000000, 0.000000012235) (7300.000000000000, 0.000000003742) (7350.000000000000, 0.000000002121) (7400.000000000000, 0.000000001981) (7450.000000000000, 0.000000001864) (7500.000000000000, 0.000000001784) (7550.000000000000, 0.000000001741) (7600.000000000000, 0.000000001744) (7650.000000000000, 0.000000001777) (7700.000000000000, 0.000000001845) (7750.000000000000, 0.000000001950) (7800.000000000000, 0.000000002098) (7850.000000000000, 0.000000002305) (7900.000000000000, 0.000000002592) (7950.000000000000, 0.000000003003) (8000.000000000000, 0.000000003592) (8050.000000000000, 0.000000004475) (8100.000000000000, 0.000000005810) (8150.000000000000, 0.000000008694) (8200.000000000000, 0.000000015301) (8250.000000000000, 0.000000051278) (8300.000000000000, 0.000000046888) (8350.000000000000, 0.000000017840) (8400.000000000000, 0.000000011964) (8450.000000000000, 0.000000009731) (8500.000000000000, 0.000000013935) (8550.000000000000, 0.000000009396) (8600.000000000000, 0.000000011764) (8650.000000000000, 0.000000030497) (8700.000000000000, 0.000000010158) (8750.000000000000, 0.000000001268) (8800.000000000000, 0.000000001187) (8850.000000000000, 0.000000002632) (8900.000000000000, 0.000000004183) (8950.000000000000, 0.000000006756) (9000.000000000000, 0.000000014170) (9050.000000000000, 0.000000323944) (9100.000000000000, 0.000000013107) (9150.000000000000, 0.000000006661) (9200.000000000000, 0.000000004450) (9250.000000000000, 0.000000003337) (9300.000000000000, 0.000000002669) (9350.000000000000, 0.000000002218) (9400.000000000000, 0.000000002561) (9450.000000000000, 0.000000001788) (9500.000000000000, 0.000000001621) (9550.000000000000, 0.000000001549) (9600.000000000000, 0.000000000131) (9650.000000000000, 0.000000001108) (9700.000000000000, 0.000000001134) (9750.000000000000, 0.000000001109) (9800.000000000000, 0.000000001102) (9850.000000000000, 0.000000001972) (9900.000000000000, 0.000000001265) (9950.000000000000, 0.000000001402) (10000.000000000000, 0.000000001730)
            };
        
        %TFEM
        \addplot[
            color=green!50!black,
            mark = *,
            mark size=1pt
            ]
            coordinates {
               (50.000000000000, 0.000000005882) (100.000000000000, 0.000000006068) (150.000000000000, 0.000000006401) (200.000000000000, 0.000000006927) (250.000000000000, 0.000000007729) (300.000000000000, 0.000000008974) (350.000000000000, 0.000000011016) (400.000000000000, 0.000000014780) (450.000000000000, 0.000000023597) (500.000000000000, 0.000000065521) (550.000000000000, 0.000000074969) (600.000000000000, 0.000000023229) (650.000000000000, 0.000000013610) (700.000000000000, 0.000000009595) (750.000000000000, 0.000000007414) (800.000000000000, 0.000000006056) (850.000000000000, 0.000000005135) (900.000000000000, 0.000000004446) (950.000000000000, 0.000000004716) (1000.000000000000, 0.000000003793) (1050.000000000000, 0.000000003500) (1100.000000000000, 0.000000003299) (1150.000000000000, 0.000000003159) (1200.000000000000, 0.000000003071) (1250.000000000000, 0.000000003029) (1300.000000000000, 0.000000003033) (1350.000000000000, 0.000000003092) (1400.000000000000, 0.000000003218) (1450.000000000000, 0.000000003434) (1500.000000000000, 0.000000003767) (1550.000000000000, 0.000000004282) (1600.000000000000, 0.000000005109) (1650.000000000000, 0.000000006561) (1700.000000000000, 0.000000009565) (1750.000000000000, 0.000000019230) (1800.000000000000, 0.000000342022) (1850.000000000000, 0.000000017708) (1900.000000000000, 0.000000009339) (1950.000000000000, 0.000000006510) (2000.000000000000, 0.000000005148) (2050.000000000000, 0.000000004517) (2100.000000000000, 0.000000006167) (2150.000000000000, 0.000000003477) (2200.000000000000, 0.000000003173) (2250.000000000000, 0.000000003002) (2300.000000000000, 0.000000002906) (2350.000000000000, 0.000000002871) (2400.000000000000, 0.000000002887) (2450.000000000000, 0.000000002951) (2500.000000000000, 0.000000003074) (2550.000000000000, 0.000000003262) (2600.000000000000, 0.000000003520) (2650.000000000000, 0.000000003878) (2700.000000000000, 0.000000004360) (2750.000000000000, 0.000000005025) (2800.000000000000, 0.000000005985) (2850.000000000000, 0.000000007381) (2900.000000000000, 0.000000009539) (2950.000000000000, 0.000000013350) (3000.000000000000, 0.000000021369) (3050.000000000000, 0.000000047464) (3100.000000000000, 0.000000583150) (3150.000000000000, 0.000000047558) (3200.000000000000, 0.000000027396) (3250.000000000000, 0.000000020902) (3300.000000000000, 0.000000018243) (3350.000000000000, 0.000000017466) (3400.000000000000, 0.000000018166) (3450.000000000000, 0.000000020775) (3500.000000000000, 0.000000027367) (3550.000000000000, 0.000000048981) (3600.000000000000, 0.000013093800) (3650.000000000000, 0.000000040462) (3700.000000000000, 0.000000018790) (3750.000000000000, 0.000000011729) (3800.000000000000, 0.000000008302) (3850.000000000000, 0.000000006315) (3900.000000000000, 0.000000005037) (3950.000000000000, 0.000000004159) (4000.000000000000, 0.000000003526) (4050.000000000000, 0.000000003055) (4100.000000000000, 0.000000002694) (4150.000000000000, 0.000000002413) (4200.000000000000, 0.000000002190) (4250.000000000000, 0.000000002012) (4300.000000000000, 0.000000001869) (4350.000000000000, 0.000000001754) (4400.000000000000, 0.000000001662) (4450.000000000000, 0.000000001589) (4500.000000000000, 0.000000001534) (4550.000000000000, 0.000000001493) (4600.000000000000, 0.000000001467) (4650.000000000000, 0.000000001611) (4700.000000000000, 0.000000008641) (4750.000000000000, 0.000000005149) (4800.000000000000, 0.000000001523) (4850.000000000000, 0.000000001564) (4900.000000000000, 0.000000001649) (4950.000000000000, 0.000000001803) (5000.000000000000, 0.000000002570) (5050.000000000000, 0.000000001657) (5100.000000000000, 0.000000002023) (5150.000000000000, 0.000000002381) (5200.000000000000, 0.000000002881) (5250.000000000000, 0.000000003688) (5300.000000000000, 0.000000005229) (5350.000000000000, 0.000000009327) (5400.000000000000, 0.000000051196) (5450.000000000000, 0.000000014058) (5500.000000000000, 0.000000006081) (5550.000000000000, 0.000000003860) (5600.000000000000, 0.000000002825) (5650.000000000000, 0.000000002234) (5700.000000000000, 0.000000001863) (5750.000000000000, 0.000000001631) (5800.000000000000, 0.000000001489) (5850.000000000000, 0.000000001455) (5900.000000000000, 0.000000001866) (5950.000000000000, 0.000000001207) (6000.000000000000, 0.000000000852) (6050.000000000000, 0.000000000782) (6100.000000000000, 0.000000000786) (6150.000000000000, 0.000000000869) (6200.000000000000, 0.000000001037) (6250.000000000000, 0.000000001399) (6300.000000000000, 0.000000002501) (6350.000000000000, 0.000000024493) (6400.000000000000, 0.000000004735) (6450.000000000000, 0.000000003688) (6500.000000000000, 0.000000003630) (6550.000000000000, 0.000000005533) (6600.000000000000, 0.000000176699) (6650.000000000000, 0.000000013010) (6700.000000000000, 0.000000001086) (6750.000000000000, 0.000000001198) (6800.000000000000, 0.000000001345) (6850.000000000000, 0.000000001550) (6900.000000000000, 0.000000001856) (6950.000000000000, 0.000000002362) (7000.000000000000, 0.000000003215) (7050.000000000000, 0.000000004918) (7100.000000000000, 0.000000009720) (7150.000000000000, 0.000000056896) (7200.000000000000, 0.000000021040) (7250.000000000000, 0.000000016221) (7300.000000000000, 0.000000007931) (7350.000000000000, 0.000000004500) (7400.000000000000, 0.000000003893) (7450.000000000000, 0.000000003759) (7500.000000000000, 0.000000003771) (7550.000000000000, 0.000000003899) (7600.000000000000, 0.000000004161) (7650.000000000000, 0.000000004598) (7700.000000000000, 0.000000005292) (7750.000000000000, 0.000000006418) (7800.000000000000, 0.000000008406) (7850.000000000000, 0.000000012587) (7900.000000000000, 0.000000026075) (7950.000000000000, 0.000000302838) (8000.000000000000, 0.000000022461) (8050.000000000000, 0.000000011857) (8100.000000000000, 0.000000008179) (8150.000000000000, 0.000000006415) (8200.000000000000, 0.000000005374) (8250.000000000000, 0.000000004733) (8300.000000000000, 0.000000004332) (8350.000000000000, 0.000000004099) (8400.000000000000, 0.000000003995) (8450.000000000000, 0.000000003973) (8500.000000000000, 0.000000042649) (8550.000000000000, 0.000000005506) (8600.000000000000, 0.000000007241) (8650.000000000000, 0.000000018347) (8700.000000000000, 0.000000008008) (8750.000000000000, 0.000000001164) (8800.000000000000, 0.000000001365) (8850.000000000000, 0.000000002351) (8900.000000000000, 0.000000003585) (8950.000000000000, 0.000000005745) (9000.000000000000, 0.000000011832) (9050.000000000000, 0.000000787014) (9100.000000000000, 0.000000012657) (9150.000000000000, 0.000000006351) (9200.000000000000, 0.000000004270) (9250.000000000000, 0.000000003239) (9300.000000000000, 0.000000002633) (9350.000000000000, 0.000000002256) (9400.000000000000, 0.000000002616) (9450.000000000000, 0.000000001793) (9500.000000000000, 0.000000001664) (9550.000000000000, 0.000000001681) (9600.000000000000, 0.000000000391) (9650.000000000000, 0.000000001105) (9700.000000000000, 0.000000001152) (9750.000000000000, 0.000000001170) (9800.000000000000, 0.000000001205) (9850.000000000000, 0.000000003201) (9900.000000000000, 0.000000001326) (9950.000000000000, 0.000000001482) (10000.000000000000, 0.000000001835)
            };

        %slivers
        \addplot[
            color=orange,
            mark = square*,
            mark size=1pt
            ]
            coordinates {
                (50.000000000000, 0.000000000147) (100.000000000000, 0.000000000148) (150.000000000000, 0.000000000149) (200.000000000000, 0.000000000151) (250.000000000000, 0.000000000153) (300.000000000000, 0.000000000156) (350.000000000000, 0.000000000159) (400.000000000000, 0.000000000164) (450.000000000000, 0.000000000169) (500.000000000000, 0.000000000175) (550.000000000000, 0.000000000182) (600.000000000000, 0.000000000191) (650.000000000000, 0.000000000202) (700.000000000000, 0.000000000215) (750.000000000000, 0.000000000233) (800.000000000000, 0.000000000260) (850.000000000000, 0.000000000306) (900.000000000000, 0.000000000447) (950.000000000000, 0.000000001178) (1000.000000000000, 0.000000000415) (1050.000000000000, 0.000000000358) (1100.000000000000, 0.000000000351) (1150.000000000000, 0.000000000363) (1200.000000000000, 0.000000000385) (1250.000000000000, 0.000000000417) (1300.000000000000, 0.000000000459) (1350.000000000000, 0.000000000515) (1400.000000000000, 0.000000000587) (1450.000000000000, 0.000000000682) (1500.000000000000, 0.000000000816) (1550.000000000000, 0.000000001008) (1600.000000000000, 0.000000001302) (1650.000000000000, 0.000000001801) (1700.000000000000, 0.000000002823) (1750.000000000000, 0.000000006036) (1800.000000000000, 0.000000111538) (1850.000000000000, 0.000000005909) (1900.000000000000, 0.000000003175) (1950.000000000000, 0.000000002242) (2000.000000000000, 0.000000001778) (2050.000000000000, 0.000000001505) (2100.000000000000, 0.000000001330) (2150.000000000000, 0.000000001212) (2200.000000000000, 0.000000001130) (2250.000000000000, 0.000000001074) (2300.000000000000, 0.000000001037) (2350.000000000000, 0.000000001014) (2400.000000000000, 0.000000001005) (2450.000000000000, 0.000000001006) (2500.000000000000, 0.000000001019) (2550.000000000000, 0.000000001042) (2600.000000000000, 0.000000001077) (2650.000000000000, 0.000000001127) (2700.000000000000, 0.000000001198) (2750.000000000000, 0.000000001315) (2800.000000000000, 0.000000001735) (2850.000000000000, 0.000000001705) (2900.000000000000, 0.000000001584) (2950.000000000000, 0.000000001668) (3000.000000000000, 0.000000001816) (3050.000000000000, 0.000000002026) (3100.000000000000, 0.000000002300) (3150.000000000000, 0.000000002655) (3200.000000000000, 0.000000003123) (3250.000000000000, 0.000000003758) (3300.000000000000, 0.000000004650) (3350.000000000000, 0.000000005965) (3400.000000000000, 0.000000008038) (3450.000000000000, 0.000000011660) (3500.000000000000, 0.000000019233) (3550.000000000000, 0.000000042854) (3600.000000000000, 0.000014258900) (3650.000000000000, 0.000000055105) (3700.000000000000, 0.000000032370) (3750.000000000000, 0.000000026058) (3800.000000000000, 0.000000024550) (3850.000000000000, 0.000000026203) (3900.000000000000, 0.000000032327) (3950.000000000000, 0.000000051639) (4000.000000000000, 0.000000260971) (4050.000000000000, 0.000000065607) (4100.000000000000, 0.000000026744) (4150.000000000000, 0.000000016099) (4200.000000000000, 0.000000011252) (4250.000000000000, 0.000000008539) (4300.000000000000, 0.000000006838) (4350.000000000000, 0.000000005693) (4400.000000000000, 0.000000004887) (4450.000000000000, 0.000000004298) (4500.000000000000, 0.000000003859) (4550.000000000000, 0.000000003527) (4600.000000000000, 0.000000003296) (4650.000000000000, 0.000000003206) (4700.000000000000, 0.000000006091) (4750.000000000000, 0.000000003834) (4800.000000000000, 0.000000002934) (4850.000000000000, 0.000000002888) (4900.000000000000, 0.000000002892) (4950.000000000000, 0.000000002939) (5000.000000000000, 0.000000002895) (5050.000000000000, 0.000000003377) (5100.000000000000, 0.000000003639) (5150.000000000000, 0.000000004083) (5200.000000000000, 0.000000004789) (5250.000000000000, 0.000000005982) (5300.000000000000, 0.000000008311) (5350.000000000000, 0.000000014553) (5400.000000000000, 0.000000077937) (5450.000000000000, 0.000000021825) (5500.000000000000, 0.000000009780) (5550.000000000000, 0.000000008731) (5600.000000000000, 0.000000004220) (5650.000000000000, 0.000000003158) (5700.000000000000, 0.000000002529) (5750.000000000000, 0.000000002100) (5800.000000000000, 0.000000001780) (5850.000000000000, 0.000000001505) (5900.000000000000, 0.000000001081) (5950.000000000000, 0.000000001767) (6000.000000000000, 0.000000001370) (6050.000000000000, 0.000000001230) (6100.000000000000, 0.000000001139) (6150.000000000000, 0.000000001075) (6200.000000000000, 0.000000001028) (6250.000000000000, 0.000000000995) (6300.000000000000, 0.000000000975) (6350.000000000000, 0.000000000966) (6400.000000000000, 0.000000000971) (6450.000000000000, 0.000000000996) (6500.000000000000, 0.000000001058) (6550.000000000000, 0.000000001261) (6600.000000000000, 0.000000014446) (6650.000000000000, 0.000000001274) (6700.000000000000, 0.000000001146) (6750.000000000000, 0.000000001172) (6800.000000000000, 0.000000001227) (6850.000000000000, 0.000000001316) (6900.000000000000, 0.000000001447) (6950.000000000000, 0.000000001641) (7000.000000000000, 0.000000001939) (7050.000000000000, 0.000000002448) (7100.000000000000, 0.000000003549) (7150.000000000000, 0.000000011006) (7200.000000000000, 0.000000001817) (7250.000000000000, 0.000000012939) (7300.000000000000, 0.000000011597) (7350.000000000000, 0.000000004382) (7400.000000000000, 0.000000002746) (7450.000000000000, 0.000000002011) (7500.000000000000, 0.000000001594) (7550.000000000000, 0.000000001326) (7600.000000000000, 0.000000001141) (7650.000000000000, 0.000000001009) (7700.000000000000, 0.000000000909) (7750.000000000000, 0.000000000833) (7800.000000000000, 0.000000000774) (7850.000000000000, 0.000000000729) (7900.000000000000, 0.000000000694) (7950.000000000000, 0.000000000668) (8000.000000000000, 0.000000000651) (8050.000000000000, 0.000000000656) (8100.000000000000, 0.000000000753) (8150.000000000000, 0.000000000596) (8200.000000000000, 0.000000000613) (8250.000000000000, 0.000000000630) (8300.000000000000, 0.000000000656) (8350.000000000000, 0.000000000700) (8400.000000000000, 0.000000000782) (8450.000000000000, 0.000000001001) (8500.000000000000, 0.000000043917) (8550.000000000000, 0.000000001110) (8600.000000000000, 0.000000000917) (8650.000000000000, 0.000000000767) (8700.000000000000, 0.000000001168) (8750.000000000000, 0.000000001183) (8800.000000000000, 0.000000001328) (8850.000000000000, 0.000000001584) (8900.000000000000, 0.000000002035) (8950.000000000000, 0.000000002966) (9000.000000000000, 0.000000005778) (9050.000000000000, 0.000000372441) (9100.000000000000, 0.000000005884) (9150.000000000000, 0.000000002918) (9200.000000000000, 0.000000001947) (9250.000000000000, 0.000000001466) (9300.000000000000, 0.000000001181) (9350.000000000000, 0.000000001055) (9400.000000000000, 0.000000002344) (9450.000000000000, 0.000000000996) (9500.000000000000, 0.000000000866) (9550.000000000000, 0.000000000882) (9600.000000000000, 0.000000000131) (9650.000000000000, 0.000000000547) (9700.000000000000, 0.000000000577) (9750.000000000000, 0.000000000503) (9800.000000000000, 0.000000000525) (9850.000000000000, 0.000000002026) (9900.000000000000, 0.000000000734) (9950.000000000000, 0.000000000777) (10000.000000000000, 0.000000000908)
            };

        \node [orange, align=center] at (rel axis cs: 
            0.33, 0.11) { Slivers \\ FEM };

        \node [blue, align=center] at (rel axis cs: 
            0.16, 0.75) { Regular\\ FEM };

        \node [green!50!black, align=center] at (rel axis cs: 
            0.58, 0.77) { Slivers\\ TFEM };

        \end{axis}
    \end{tikzpicture} \\ \vspace{0.5cm}
    \begin{tikzpicture}
        \begin{axis}[
            xlabel={$\omega$},
            ylabel={$\kappa \cdot 10^6$},
            legend style={at={(0.65,0.4)},anchor=west},
            ymajorgrids=true,
            grid=none,
            ymode=log,
            xmin=0, xmax=5000,
            ymin=1e-13, ymax=5e-1,
            xtick={1000, 2000, 3000, 4000},
            width=\textwidth,
            height=0.6\textwidth,
            x label style={at={(axis description cs:0.5,0.0)},anchor=north},
            y label style={at={(axis description cs:0.2,0.85)},rotate=-90,anchor=south},
            legend style={draw=none}
        ]
        %regular
        \addplot[
            color=blue,
            mark = triangle*,
            mark size=1pt
            ]
            coordinates {
               (50.000000000000, 0.000000005943) (100.000000000000, 0.000000006308) (150.000000000000, 0.000000006988) (200.000000000000, 0.000000008121) (250.000000000000, 0.000000009988) (300.000000000000, 0.000000013187) (350.000000000000, 0.000000019167) (400.000000000000, 0.000000032168) (450.000000000000, 0.000000069527) (500.000000000000, 0.000000275690) (550.000000000000, 0.000329794000) (600.000000000000, 0.000000237782) (650.000000000000, 0.000000061116) (700.000000000000, 0.000000027786) (750.000000000000, 0.000000016130) (800.000000000000, 0.000000010786) (850.000000000000, 0.000000007941) (900.000000000000, 0.000000006326) (950.000000000000, 0.000000005522) (1000.000000000000, 0.000000004542) (1050.000000000000, 0.000000004150) (1100.000000000000, 0.000000003941) (1150.000000000000, 0.000000003883) (1200.000000000000, 0.000000003964) (1250.000000000000, 0.000000004195) (1300.000000000000, 0.000000004605) (1350.000000000000, 0.000000005257) (1400.000000000000, 0.000000006263) (1450.000000000000, 0.000000007831) (1500.000000000000, 0.000000010372) (1550.000000000000, 0.000000014768) (1600.000000000000, 0.000000023186) (1650.000000000000, 0.000000042152) (1700.000000000000, 0.000000099272) (1750.000000000000, 0.000000439017) (1800.000000000000, 0.000113426000) (1850.000000000000, 0.000000385903) (1900.000000000000, 0.000000109906) (1950.000000000000, 0.000000054189) (2000.000000000000, 0.000000033867) (2050.000000000000, 0.000000024259) (2100.000000000000, 0.000000019029) (2150.000000000000, 0.000000015949) (2200.000000000000, 0.000000018028) (2250.000000000000, 0.000000013076) (2300.000000000000, 0.000000012503) (2350.000000000000, 0.000000012374) (2400.000000000000, 0.000000012629) (2450.000000000000, 0.000000013268) (2500.000000000000, 0.000000014337) (2550.000000000000, 0.000000015931) (2600.000000000000, 0.000000018217) (2650.000000000000, 0.000000021462) (2700.000000000000, 0.000000026105) (2750.000000000000, 0.000000032882) (2800.000000000000, 0.000000043013) (2850.000000000000, 0.000000059682) (2900.000000000000, 0.000000086507) (2950.000000000000, 0.000000135055) (3000.000000000000, 0.000000232318) (3050.000000000000, 0.000000460621) (3100.000000000000, 0.000001160450) (3150.000000000000, 0.000004958880) (3200.000000000000, 0.000513649000) (3250.000000000000, 0.000012797000) (3300.000000000000, 0.000003749960) (3350.000000000000, 0.000002278970) (3400.000000000000, 0.000001955770) (3450.000000000000, 0.000002205770) (3500.000000000000, 0.000003459050) (3550.000000000000, 0.000010399400) (3600.000000000000, 0.027344200000) (3650.000000000000, 0.000006066150) (3700.000000000000, 0.000001280270) (3750.000000000000, 0.000000487546) (3800.000000000000, 0.000000239625) (3850.000000000000, 0.000000136511) (3900.000000000000, 0.000000085837) (3950.000000000000, 0.000000058008) (4000.000000000000, 0.000000041454) (4050.000000000000, 0.000000030995) (4100.000000000000, 0.000000024071) (4150.000000000000, 0.000000019313) (4200.000000000000, 0.000000015947) (4250.000000000000, 0.000000013513) (4300.000000000000, 0.000000011724) (4350.000000000000, 0.000000010398) (4400.000000000000, 0.000000009415) (4450.000000000000, 0.000000008696) (4500.000000000000, 0.000000008190) (4550.000000000000, 0.000000007861) (4600.000000000000, 0.000000007691) (4650.000000000000, 0.000000007667) (4700.000000000000, 0.000000009471) (4750.000000000000, 0.000000012275) (4800.000000000000, 0.000000008711) (4850.000000000000, 0.000000009612) (4900.000000000000, 0.000000011073) (4950.000000000000, 0.000000013946) (5000.000000000000, 0.000000035136) (5050.000000000000, 0.000000010028) (5100.000000000000, 0.000000016792) (5150.000000000000, 0.000000024441) (5200.000000000000, 0.000000037256) (5250.000000000000, 0.000000063499) (5300.000000000000, 0.000000133268) (5350.000000000000, 0.000000447879) (5400.000000000000, 0.000017117600) (5450.000000000000, 0.000000988074) (5500.000000000000, 0.000000200398) (5550.000000000000, 0.000000085249) (5600.000000000000, 0.000000048111) (5650.000000000000, 0.000000031759) (5700.000000000000, 0.000000023328) (5750.000000000000, 0.000000018706) (5800.000000000000, 0.000000016472) (5850.000000000000, 0.000000016977) (5900.000000000000, 0.000000033251) (5950.000000000000, 0.000000000906) (6000.000000000000, 0.000000002810) (6050.000000000000, 0.000000003755) (6100.000000000000, 0.000000004093) (6150.000000000000, 0.000000004260) (6200.000000000000, 0.000000004392) (6250.000000000000, 0.000000004546) (6300.000000000000, 0.000000004754) (6350.000000000000, 0.000000005041) (6400.000000000000, 0.000000005444) (6450.000000000000, 0.000000006054) (6500.000000000000, 0.000000007300) (6550.000000000000, 0.000000015222) (6600.000000000000, 0.000015047700) (6650.000000000000, 0.000000866490) (6700.000000000000, 0.000000516803) (6750.000000000000, 0.000000017319) (6800.000000000000, 0.000000021210) (6850.000000000000, 0.000000028113) (6900.000000000000, 0.000000041092) (6950.000000000000, 0.000000064749) (7000.000000000000, 0.000000115720) (7050.000000000000, 0.000000255408) (7100.000000000000, 0.000000897436) (7150.000000000000, 0.000033594900) (7200.000000000000, 0.000002099730) (7250.000000000000, 0.000000379403) (7300.000000000000, 0.000000463117) (7350.000000000000, 0.000000196780) (7400.000000000000, 0.000000129246) (7450.000000000000, 0.000000098491) (7500.000000000000, 0.000000082155) (7550.000000000000, 0.000000073222) (7600.000000000000, 0.000000068868) (7650.000000000000, 0.000000067819) (7700.000000000000, 0.000000069738) (7750.000000000000, 0.000000074726) (7800.000000000000, 0.000000083467) (7850.000000000000, 0.000000097399) (7900.000000000000, 0.000000119254) (7950.000000000000, 0.000000154325) (8000.000000000000, 0.000000213595) (8050.000000000000, 0.000000322580) (8100.000000000000, 0.000000552645) (8150.000000000000, 0.000001137910) (8200.000000000000, 0.000003456120) (8250.000000000000, 0.000038056500) (8300.000000000000, 0.000031261100) (8350.000000000000, 0.000004457670) (8400.000000000000, 0.000001979950) (8450.000000000000, 0.000001294400) (8500.000000000000, 0.000001979530) (8550.000000000000, 0.000001236340) (8600.000000000000, 0.000001977650) (8650.000000000000, 0.000014011200) (8700.000000000000, 0.000001795700) (8750.000000000000, 0.000000054921) (8800.000000000000, 0.000000005668) (8850.000000000000, 0.000000053349) (8900.000000000000, 0.000000156293) (8950.000000000000, 0.000000433365) (9000.000000000000, 0.000001969780) (9050.000000000000, 0.001050620000) (9100.000000000000, 0.000001744470) (9150.000000000000, 0.000000455501) (9200.000000000000, 0.000000205293) (9250.000000000000, 0.000000116536) (9300.000000000000, 0.000000075434) (9350.000000000000, 0.000000053279) (9400.000000000000, 0.000000042561) (9450.000000000000, 0.000000032173) (9500.000000000000, 0.000000026839) (9550.000000000000, 0.000000023078) (9600.000000000000, 0.000000000444) (9650.000000000000, 0.000000014656) (9700.000000000000, 0.000000015119) (9750.000000000000, 0.000000015092) (9800.000000000000, 0.000000015521) (9850.000000000000, 0.000000019610) (9900.000000000000, 0.000000019286) (9950.000000000000, 0.000000024427) (10000.000000000000, 0.000000038065)
            };
            
        %TFEM
        \addplot[
            color=green!50!black,
            mark = *,
            mark size=1pt
            ]
            coordinates {
               (50.000000000000, 0.000000005172) (100.000000000000, 0.000000005508) (150.000000000000, 0.000000006141) (200.000000000000, 0.000000007210) (250.000000000000, 0.000000009010) (300.000000000000, 0.000000012202) (350.000000000000, 0.000000018496) (400.000000000000, 0.000000033541) (450.000000000000, 0.000000086275) (500.000000000000, 0.000000672658) (550.000000000000, 0.000000892708) (600.000000000000, 0.000000087128) (650.000000000000, 0.000000030508) (700.000000000000, 0.000000015532) (750.000000000000, 0.000000009546) (800.000000000000, 0.000000006605) (850.000000000000, 0.000000004985) (900.000000000000, 0.000000004128) (950.000000000000, 0.000000004649) (1000.000000000000, 0.000000002948) (1050.000000000000, 0.000000002696) (1100.000000000000, 0.000000002575) (1150.000000000000, 0.000000002555) (1200.000000000000, 0.000000002628) (1250.000000000000, 0.000000002805) (1300.000000000000, 0.000000003109) (1350.000000000000, 0.000000003586) (1400.000000000000, 0.000000004321) (1450.000000000000, 0.000000005470) (1500.000000000000, 0.000000007340) (1550.000000000000, 0.000000010595) (1600.000000000000, 0.000000016871) (1650.000000000000, 0.000000031116) (1700.000000000000, 0.000000074269) (1750.000000000000, 0.000000331516) (1800.000000000000, 0.000111332000) (1850.000000000000, 0.000000309668) (1900.000000000000, 0.000000089318) (1950.000000000000, 0.000000044889) (2000.000000000000, 0.000000028875) (2050.000000000000, 0.000000022870) (2100.000000000000, 0.000000033538) (2150.000000000000, 0.000000014526) (2200.000000000000, 0.000000012765) (2250.000000000000, 0.000000011951) (2300.000000000000, 0.000000011668) (2350.000000000000, 0.000000011804) (2400.000000000000, 0.000000012340) (2450.000000000000, 0.000000013311) (2500.000000000000, 0.000000014811) (2550.000000000000, 0.000000017013) (2600.000000000000, 0.000000020207) (2650.000000000000, 0.000000024888) (2700.000000000000, 0.000000031928) (2750.000000000000, 0.000000042939) (2800.000000000000, 0.000000061190) (2850.000000000000, 0.000000094837) (2900.000000000000, 0.000000160163) (2950.000000000000, 0.000000316208) (3000.000000000000, 0.000000816818) (3050.000000000000, 0.000004057920) (3100.000000000000, 0.000614642000) (3150.000000000000, 0.000004099600) (3200.000000000000, 0.000001365080) (3250.000000000000, 0.000000797733) (3300.000000000000, 0.000000610370) (3350.000000000000, 0.000000562301) (3400.000000000000, 0.000000611619) (3450.000000000000, 0.000000803982) (3500.000000000000, 0.000001400630) (3550.000000000000, 0.000004507310) (3600.000000000000, 0.323763000000) (3650.000000000000, 0.000003109520) (3700.000000000000, 0.000000674840) (3750.000000000000, 0.000000264763) (3800.000000000000, 0.000000133648) (3850.000000000000, 0.000000077939) (3900.000000000000, 0.000000050021) (3950.000000000000, 0.000000034418) (4000.000000000000, 0.000000024990) (4050.000000000000, 0.000000018951) (4100.000000000000, 0.000000014903) (4150.000000000000, 0.000000012093) (4200.000000000000, 0.000000010088) (4250.000000000000, 0.000000008626) (4300.000000000000, 0.000000007546) (4350.000000000000, 0.000000006743) (4400.000000000000, 0.000000006147) (4450.000000000000, 0.000000005712) (4500.000000000000, 0.000000005408) (4550.000000000000, 0.000000005214) (4600.000000000000, 0.000000005119) (4650.000000000000, 0.000000005113) (4700.000000000000, 0.000000012964) (4750.000000000000, 0.000000009920) (4800.000000000000, 0.000000005945) (4850.000000000000, 0.000000006590) (4900.000000000000, 0.000000007611) (4950.000000000000, 0.000000009601) (5000.000000000000, 0.000000023973) (5050.000000000000, 0.000000006999) (5100.000000000000, 0.000000011700) (5150.000000000000, 0.000000017046) (5200.000000000000, 0.000000026003) (5250.000000000000, 0.000000044299) (5300.000000000000, 0.000000092634) (5350.000000000000, 0.000000306658) (5400.000000000000, 0.000009637400) (5450.000000000000, 0.000000754541) (5500.000000000000, 0.000000148701) (5550.000000000000, 0.000000062900) (5600.000000000000, 0.000000035497) (5650.000000000000, 0.000000023489) (5700.000000000000, 0.000000017321) (5750.000000000000, 0.000000013964) (5800.000000000000, 0.000000012386) (5850.000000000000, 0.000000012902) (5900.000000000000, 0.000000025614) (5950.000000000000, 0.000000001142) (6000.000000000000, 0.000000002045) (6050.000000000000, 0.000000002794) (6100.000000000000, 0.000000003086) (6150.000000000000, 0.000000003246) (6200.000000000000, 0.000000003387) (6250.000000000000, 0.000000003583) (6300.000000000000, 0.000000004141) (6350.000000000000, 0.000000076768) (6400.000000000000, 0.000000006918) (6450.000000000000, 0.000000009725) (6500.000000000000, 0.000000009999) (6550.000000000000, 0.000000010867) (6600.000000000000, 0.000004156040) (6650.000000000000, 0.000000027304) (6700.000000000000, 0.000000011660) (6750.000000000000, 0.000000014784) (6800.000000000000, 0.000000019625) (6850.000000000000, 0.000000027412) (6900.000000000000, 0.000000040857) (6950.000000000000, 0.000000066180) (7000.000000000000, 0.000000121573) (7050.000000000000, 0.000000275745) (7100.000000000000, 0.000000985492) (7150.000000000000, 0.000029325700) (7200.000000000000, 0.000003144800) (7250.000000000000, 0.000000656288) (7300.000000000000, 0.000000522243) (7350.000000000000, 0.000000259157) (7400.000000000000, 0.000000187549) (7450.000000000000, 0.000000156590) (7500.000000000000, 0.000000144172) (7550.000000000000, 0.000000143879) (7600.000000000000, 0.000000154742) (7650.000000000000, 0.000000179747) (7700.000000000000, 0.000000227817) (7750.000000000000, 0.000000322315) (7800.000000000000, 0.000000533796) (7850.000000000000, 0.000001158590) (7900.000000000000, 0.000004827420) (7950.000000000000, 0.000634050000) (8000.000000000000, 0.000003406070) (8050.000000000000, 0.000000929420) (8100.000000000000, 0.000000440726) (8150.000000000000, 0.000000261173) (8200.000000000000, 0.000000181016) (8250.000000000000, 0.000000138803) (8300.000000000000, 0.000000115308) (8350.000000000000, 0.000000102663) (8400.000000000000, 0.000000097405) (8450.000000000000, 0.000000096935) (8500.000000000000, 0.000001935000) (8550.000000000000, 0.000000182587) (8600.000000000000, 0.000000328303) (8650.000000000000, 0.000002239440) (8700.000000000000, 0.000000493640) (8750.000000000000, 0.000000017593) (8800.000000000000, 0.000000002870) (8850.000000000000, 0.000000016139) (8900.000000000000, 0.000000047834) (8950.000000000000, 0.000000135999) (9000.000000000000, 0.000000609541) (9050.000000000000, 0.002789610000) (9100.000000000000, 0.000000737959) (9150.000000000000, 0.000000188744) (9200.000000000000, 0.000000086280) (9250.000000000000, 0.000000050023) (9300.000000000000, 0.000000033119) (9350.000000000000, 0.000000023950) (9400.000000000000, 0.000000020986) (9450.000000000000, 0.000000015408) (9500.000000000000, 0.000000013038) (9550.000000000000, 0.000000011800) (9600.000000000000, 0.000000001181) (9650.000000000000, 0.000000007192) (9700.000000000000, 0.000000007603) (9750.000000000000, 0.000000007721) (9800.000000000000, 0.000000008049) (9850.000000000000, 0.000000029146) (9900.000000000000, 0.000000010241) (9950.000000000000, 0.000000012988) (10000.000000000000, 0.000000020136)
            };

        %slivers
        \addplot[
            color=orange,
            mark = square*,
            mark size=1pt
            ]
            coordinates {
                (50.000000000000, 0.000000000000) (100.000000000000, 0.000000000000) (150.000000000000, 0.000000000000) (200.000000000000, 0.000000000000) (250.000000000000, 0.000000000000) (300.000000000000, 0.000000000000) (350.000000000000, 0.000000000000) (400.000000000000, 0.000000000000) (450.000000000000, 0.000000000000) (500.000000000000, 0.000000000000) (550.000000000000, 0.000000000000) (600.000000000000, 0.000000000000) (650.000000000000, 0.000000000000) (700.000000000000, 0.000000000000) (750.000000000000, 0.000000000000) (800.000000000000, 0.000000000000) (850.000000000000, 0.000000000000) (900.000000000000, 0.000000000001) (950.000000000000, 0.000000000020) (1000.000000000000, 0.000000000001) (1050.000000000000, 0.000000000000) (1100.000000000000, 0.000000000000) (1150.000000000000, 0.000000000000) (1200.000000000000, 0.000000000000) (1250.000000000000, 0.000000000000) (1300.000000000000, 0.000000000000) (1350.000000000000, 0.000000000000) (1400.000000000000, 0.000000000001) (1450.000000000000, 0.000000000001) (1500.000000000000, 0.000000000001) (1550.000000000000, 0.000000000001) (1600.000000000000, 0.000000000002) (1650.000000000000, 0.000000000005) (1700.000000000000, 0.000000000011) (1750.000000000000, 0.000000000052) (1800.000000000000, 0.000000017536) (1850.000000000000, 0.000000000049) (1900.000000000000, 0.000000000014) (1950.000000000000, 0.000000000007) (2000.000000000000, 0.000000000004) (2050.000000000000, 0.000000000003) (2100.000000000000, 0.000000000002) (2150.000000000000, 0.000000000002) (2200.000000000000, 0.000000000002) (2250.000000000000, 0.000000000002) (2300.000000000000, 0.000000000002) (2350.000000000000, 0.000000000001) (2400.000000000000, 0.000000000001) (2450.000000000000, 0.000000000001) (2500.000000000000, 0.000000000001) (2550.000000000000, 0.000000000002) (2600.000000000000, 0.000000000002) (2650.000000000000, 0.000000000002) (2700.000000000000, 0.000000000002) (2750.000000000000, 0.000000000002) (2800.000000000000, 0.000000000005) (2850.000000000000, 0.000000000008) (2900.000000000000, 0.000000000005) (2950.000000000000, 0.000000000006) (3000.000000000000, 0.000000000007) (3050.000000000000, 0.000000000009) (3100.000000000000, 0.000000000011) (3150.000000000000, 0.000000000015) (3200.000000000000, 0.000000000021) (3250.000000000000, 0.000000000031) (3300.000000000000, 0.000000000048) (3350.000000000000, 0.000000000080) (3400.000000000000, 0.000000000149) (3450.000000000000, 0.000000000320) (3500.000000000000, 0.000000000888) (3550.000000000000, 0.000000004503) (3600.000000000000, 0.000509770000) (3650.000000000000, 0.000000007791) (3700.000000000000, 0.000000002753) (3750.000000000000, 0.000000001828) (3800.000000000000, 0.000000001663) (3850.000000000000, 0.000000001943) (3900.000000000000, 0.000000003034) (3950.000000000000, 0.000000007949) (4000.000000000000, 0.000000208507) (4050.000000000000, 0.000000013540) (4100.000000000000, 0.000000002313) (4150.000000000000, 0.000000000862) (4200.000000000000, 0.000000000433) (4250.000000000000, 0.000000000257) (4300.000000000000, 0.000000000170) (4350.000000000000, 0.000000000121) (4400.000000000000, 0.000000000092) (4450.000000000000, 0.000000000073) (4500.000000000000, 0.000000000061) (4550.000000000000, 0.000000000052) (4600.000000000000, 0.000000000047) (4650.000000000000, 0.000000000044) (4700.000000000000, 0.000000000131) (4750.000000000000, 0.000000000048) (4800.000000000000, 0.000000000036) (4850.000000000000, 0.000000000037) (4900.000000000000, 0.000000000039) (4950.000000000000, 0.000000000042) (5000.000000000000, 0.000000000043) (5050.000000000000, 0.000000000060) (5100.000000000000, 0.000000000072) (5150.000000000000, 0.000000000093) (5200.000000000000, 0.000000000132) (5250.000000000000, 0.000000000211) (5300.000000000000, 0.000000000418) (5350.000000000000, 0.000000001317) (5400.000000000000, 0.000000039665) (5450.000000000000, 0.000000002914) (5500.000000000000, 0.000000000541) (5550.000000000000, 0.000000000251) (5600.000000000000, 0.000000000119) (5650.000000000000, 0.000000000073) (5700.000000000000, 0.000000000049) (5750.000000000000, 0.000000000035) (5800.000000000000, 0.000000000026) (5850.000000000000, 0.000000000019) (5900.000000000000, 0.000000000010) (5950.000000000000, 0.000000000028) (6000.000000000000, 0.000000000017) (6050.000000000000, 0.000000000014) (6100.000000000000, 0.000000000013) (6150.000000000000, 0.000000000012) (6200.000000000000, 0.000000000011) (6250.000000000000, 0.000000000010) (6300.000000000000, 0.000000000010) (6350.000000000000, 0.000000000010) (6400.000000000000, 0.000000000010) (6450.000000000000, 0.000000000010) (6500.000000000000, 0.000000000011) (6550.000000000000, 0.000000000012) (6600.000000000000, 0.000000000840) (6650.000000000000, 0.000000000014) (6700.000000000000, 0.000000000014) (6750.000000000000, 0.000000000016) (6800.000000000000, 0.000000000019) (6850.000000000000, 0.000000000023) (6900.000000000000, 0.000000000029) (6950.000000000000, 0.000000000040) (7000.000000000000, 0.000000000058) (7050.000000000000, 0.000000000095) (7100.000000000000, 0.000000000208) (7150.000000000000, 0.000000002074) (7200.000000000000, 0.000000000054) (7250.000000000000, 0.000000002963) (7300.000000000000, 0.000000002462) (7350.000000000000, 0.000000000360) (7400.000000000000, 0.000000000145) (7450.000000000000, 0.000000000080) (7500.000000000000, 0.000000000051) (7550.000000000000, 0.000000000036) (7600.000000000000, 0.000000000027) (7650.000000000000, 0.000000000022) (7700.000000000000, 0.000000000018) (7750.000000000000, 0.000000000015) (7800.000000000000, 0.000000000013) (7850.000000000000, 0.000000000012) (7900.000000000000, 0.000000000011) (7950.000000000000, 0.000000000010) (8000.000000000000, 0.000000000010) (8050.000000000000, 0.000000000009) (8100.000000000000, 0.000000000009) (8150.000000000000, 0.000000000009) (8200.000000000000, 0.000000000009) (8250.000000000000, 0.000000000009) (8300.000000000000, 0.000000000010) (8350.000000000000, 0.000000000010) (8400.000000000000, 0.000000000011) (8450.000000000000, 0.000000000012) (8500.000000000000, 0.000000014877) (8550.000000000000, 0.000000000018) (8600.000000000000, 0.000000000018) (8650.000000000000, 0.000000000015) (8700.000000000000, 0.000000000039) (8750.000000000000, 0.000000000044) (8800.000000000000, 0.000000000059) (8850.000000000000, 0.000000000089) (8900.000000000000, 0.000000000153) (8950.000000000000, 0.000000000339) (9000.000000000000, 0.000000001328) (9050.000000000000, 0.000005641740) (9100.000000000000, 0.000000001432) (9150.000000000000, 0.000000000359) (9200.000000000000, 0.000000000163) (9250.000000000000, 0.000000000094) (9300.000000000000, 0.000000000063) (9350.000000000000, 0.000000000047) (9400.000000000000, 0.000000000058) (9450.000000000000, 0.000000000031) (9500.000000000000, 0.000000000027) (9550.000000000000, 0.000000000030) (9600.000000000000, 0.000000000000) (9650.000000000000, 0.000000000011) (9700.000000000000, 0.000000000012) (9750.000000000000, 0.000000000013) (9800.000000000000, 0.000000000013) (9850.000000000000, 0.000000000270) (9900.000000000000, 0.000000000016) (9950.000000000000, 0.000000000019) (10000.000000000000, 0.000000000028)
            };
        
        \node [orange, align=center] at (rel axis cs: 
            0.26, 0.18) { Slivers \\ FEM };

        \node [blue, align=center] at (rel axis cs: 
            0.163, 0.77) { Regular\\ FEM };

        \node [green!50!black, align=center] at (rel axis cs: 
            0.575, 0.862) { Slivers\\ TFEM };

        \end{axis}
    \end{tikzpicture}
    
    \caption{Harmonic study of the 3D vibro-acoustic problem for the mean displacement norm of the top surface (top) and the curvature indicator (bottom) of the beam. The analysis is performed on a regular mesh (blue line, triangle), a mesh with slivers (orange line, square), and a mesh with slivers and TFEM (green line, circle).}
    \label{fig:vibro_graph}
\end{figure}
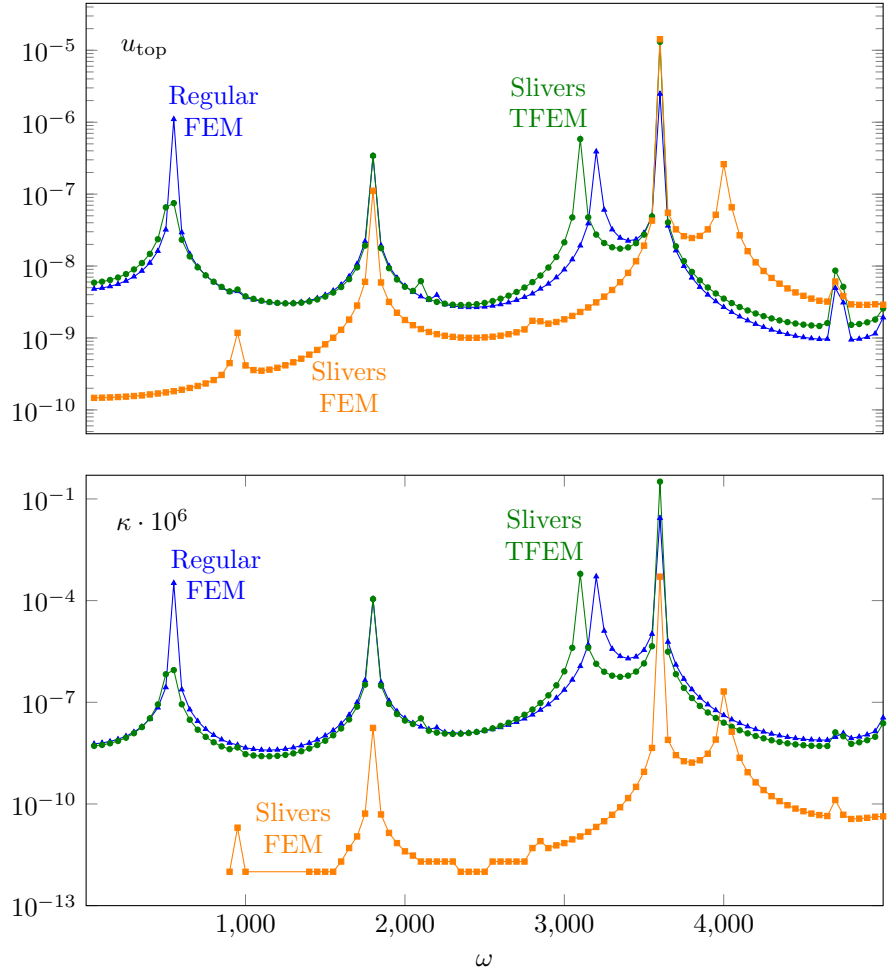

%% file: conclusion.tex
\section{Conclusion}

Sliver elements have long been regarded as one of the principal obstacles to the reliable use of three-dimensional Delaunay tetrahedralizations in finite element simulations. Their presence is traditionally associated with the violation of classical mesh-quality criteria, poor interpolation properties, locking phenomena, and potential numerical instabilities. The purpose of this work was to revisit these issues from the perspective of the finite element solution itself rather than from purely geometric considerations.

We first recalled that, unlike in two dimensions, where the min--max angle property prevents the appearance of degenerate elements, three-dimensional Delaunay tetrahedralizations may naturally contain slivers. A concise review of existing sliver-removal strategies showed that, despite decades of research, eliminating slivers remains a challenging and computationally expensive task, especially in large-scale industrial meshes.

The analysis of the stiffness matrices associated with degenerate tetrahedra revealed that the essential difficulty does not originate from interpolation considerations alone. Flat tetrahedra generate high-energy modes that behave as increasingly strong constraints on the discrete solution as the Jacobian tends to zero. For isolated degenerate elements, these constraints remain local and do not prevent convergence. In fact, as previously established by TFEM theory and further illustrated in this work, finite element convergence can still be recovered when the Jacobian determinant is simply bounded below by machine precision. This observation provides another illustration that classical geometric conditions such as the maximum-angle condition are sufficient but not necessary for finite element convergence.

The situation changes fundamentally when degenerate elements interact. We showed that locking phenomena arise when the constraints induced by individual degenerate elements couple together through clusters or sheets of slivers. In such configurations, the discrete solution becomes artificially restricted (basically losing a dimension for the representation of the solution) over entire regions of the domain, leading to severe loss of accuracy and, ultimately, loss of convergence. The analysis therefore identifies the true source of the difficulty: not the presence of isolated slivers, but the formation of connected structures of degenerate elements.

To address this issue, we employed the Tempered Finite Element Method (TFEM), which limits the growth of the singular part of the stiffness matrix by introducing a lower bound on the Jacobian determinant. This approach can be interpreted as a vanishing added-volume correction, in which the effective element volume is constrained to remain above the threshold ($J_{\min}$). The analysis of sliver bands shows that the appropriate scaling is
\[
J_{\min} = \frac{h^4}{H},
\]
where \(h\) denotes the local mesh size and \(H\) represents the characteristic length scale of the physical features that must be resolved. Importantly, the method does not require a precise estimate of \(H\). The numerical experiments demonstrate that the error is only weakly sensitive to this parameter in the vicinity of its optimum: variations over approximately one order of magnitude around the optimal value still produce errors comparable to those obtained on regular meshes. In practice, \(H\) should therefore be viewed as an order-of-magnitude estimate of the smallest relevant physical length scale rather than as a finely tuned parameter. With this choice, TFEM removes the artificial locking constraints while preserving the physically relevant solution modes, thereby recovering the optimal convergence properties predicted by the TFEM framework.

Finally, the method was assessed on a broad spectrum of representative physical problems, including incompressible Navier--Stokes flows, Cahn--Hilliard phase-field simulations, transient wave propagation, and vibro-acoustic fluid--structure interaction. Across all considered applications, standard finite elements exhibited the expected locking behaviour in the presence of sliver clusters, whereas TFEM consistently recovered accurate and physically meaningful solutions. These results demonstrate that TFEM provides a simple, robust, and broadly applicable alternative to geometric sliver-removal procedures, allowing computations to be performed directly on meshes that would otherwise be considered unsuitable for finite element analysis.

%% file: minmax.tex
\section{MinMax proprety of Delaunay triangulations}
\label{annex:1}
\begin{figure}
\begin{center}
\includegraphics[width=.45\linewidth]{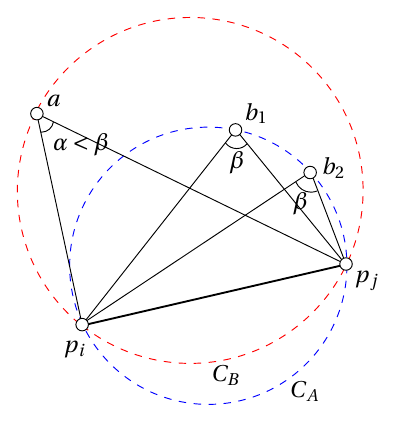}    
\end{center}
\caption{\label{fig:minmax}}    
\end{figure}

An edge in a two--dimensional triangulation is said to be \emph{locally Delaunay}
if it satisfies the empty--circumcircle condition with respect to the two triangles
that share it: the vertex opposite to the edge in each triangle must lie outside
the circumcircle of the other triangle. 

A triangulation in which \emph{every} edge
is locally Delaunay is then \emph{globally Delaunay}, since the local empty--circle
property is exactly the condition that characterizes the Delaunay triangulation of
a point set. Starting from any admissible triangulation, one can enforce this
property by performing a sequence of \emph{edge flips}: whenever an edge violates
the local Delaunay criterion, flipping it restores local Delaunayhood (see Figure \ref{fig:locDel1}). 
It is a classical result that repeated flips always converge to the unique Delaunay
triangulation of the point set.  The convergence of the flip procedure is immediate once the following property
has been proved: for any triangulation \(T\) that is not Delaunay, there exists
an edge \(e \in T\) whose flip strictly increases the local empty-circle 
quality (equivalently, strictly decreases the maximal angle). As each flip 
produces a strict improvement of a globally bounded functional, no infinite 
sequence of improving flips is possible. Hence the process must terminate on 
a triangulation that admits no improving flip, i.e., the Delaunay triangulation.

Let us thus prove that performing a flip that turns a
non--Delaunay edge into a locally Delaunay one strictly increases the minimum
angle in the affected quadrilateral, ensuring monotonic improvement of the
triangulation quality. 

Figure \ref{fig:minmax} illustrates the inscribed angle theorem. 
The points $b_1$ and $b_2$ lie on a circle $C_B$ passing through
$p_i$ and $p_j$, while the point $a$ lies on a larger circle $C_A$ that also
contains $p_i$ and $p_j$. Thales' theorem states that the inscribed angle
subtended by a fixed chord---here the segment $(p_i,p_j)$---depends only on the
position of the viewpoint along its circle: the farther the viewpoint lies from
the chord, the smaller the corresponding inscribed angle becomes. Since $a$
belongs to a circle strictly larger than the one containing $b_1$ and $b_2$, the
angle $\alpha = \angle a\,p_i\,p_j$ must be smaller than 
$\beta = \angle b\,p_i\,p_j$. The figure therefore expresses the classical fact
that inscribed angles decrease when their apex moves away from the chord, which
is precisely the inequality $\alpha < \beta$.

\begin{figure}
    \centering
    \input{locDel1}
    \caption{Edge flip. Edge $\bm{p}_k\bm{p}_i$ is not locally Delaunay because circles
    $\bm{p}_k\bm{p}_i\bm{p}_j$ and $\bm{p}_k\bm{p}_i\bm{p}_l$ are not empty. Flipping
    the edge to $\bm{p}_j\bm{p}_l$ solves the problem and creates two empty circles $\bm{p}_j\bm{p}_l\bm{p}_i$ and $\bm{p}_j\bm{p}_l\bm{p}_k$.}
    \label{fig:locDel1}
\end{figure}
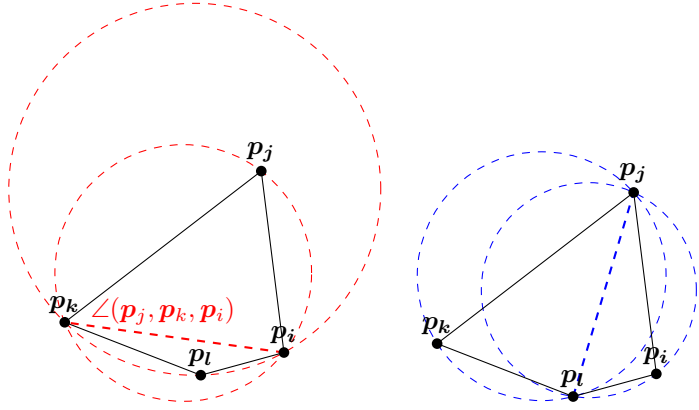

Thales' theorem applied to each side of the quadrilateral 
$(p_i,p_j,p_k,p_l)$ of Figure \ref{fig:locDel1} yields the four inequalities
\[
\text{Side } p_ip_j 
\;\;\longrightarrow\;\;
{\color{red}{\angle(p_i,p_k,p_j)}} < \angle(p_i,p_l,p_j),
\]
\[
\text{Side } p_jp_k 
\;\;\longrightarrow\;\;
{\color{red}{\angle(p_j,p_i,p_k)}} < \angle(p_j,p_l,p_k),
\]
\[
\text{Side } p_kp_l
\;\;\longrightarrow\;\;
{\color{red}{\angle(p_k,p_i,p_l)}} < \angle(p_k,p_j,p_l),
\]
\[
\text{Side } p_lp_i
\;\;\longrightarrow\;\;
{\color{red}{\angle(p_l,p_k,p_i)}} < \angle(p_l,p_j,p_i).
\]
Assume that $\color{red}\angle(p_l,p_k,p_i)$ is the smallest of the four 
angles written in red just above. The previous inequalities
show that $\angle(p_i,p_l,p_j)$, $\angle(p_j,p_l,p_k)$, $\angle(p_k,p_j,p_l)$
and $\angle(p_l,p_j,p_i)$ are all larger than $\color{red}\angle(p_l,p_k,p_i)$.

Two additional angles must also be compared to
$\color{red}\angle(p_j,p_i,p_l)$, namely
$\angle(p_j,p_k,p_l)$
and
$\angle(p_l,p_i,p_j)$.
We have obviously $\angle(p_j,p_k,p_l) = \angle(p_j,p_k,p_i)+\angle(p_i,p_k,p_l) > {\color{red} \angle(p_i,p_k,p_l)}$.
Finally $\angle(p_l,p_i,p_j) > \angle(p_l,p_i,p_k) = {\color{red}\angle(p_i,p_k,p_l)}$ which
demonstrate that making an edge locally Delaunay locally increases the
minimum angle of the two neighboring triangles.

%% file: locDel1.tex
\begin{tabular}{cc}
\begin{tikzpicture}[scale=1]
  % points indépendants

  \pgfmathsetmacro{\xi}{3.0}
  \pgfmathsetmacro{\yi}{0.2}
  \pgfmathsetmacro{\xj}{2.7}
  \pgfmathsetmacro{\yj}{2.6}
  \pgfmathsetmacro{\xk}{0.1}
  \pgfmathsetmacro{\yk}{0.6}
  \pgfmathsetmacro{\xl}{1.9}
  \pgfmathsetmacro{\yl}{-0.1}

  \coordinate (pi) at (\xi,\yi);
  \coordinate (pk) at (\xk,\yk);
  \coordinate (pj) at (\xj,\yj);
  \coordinate (pl) at (\xl,\yl);

  \node at ($(pk)+(1.3,0.15)$) {${\color{red}\angle(\bm{p}_j,\bm{p}_k,\bm{p}_i)}$};
  
  \newcommand{\computecircum}[6]{%
  % dénominateur
    \pgfmathsetmacro{\den}{2*(#1*(#4-#6)+#3*(#6-#2)+#5*(#2-#4))}%
  % centre
    \pgfmathsetmacro{\Ux}{((#1*#1+#2*#2)*(#4-#6)
                       +(#3*#3+#4*#4)*(#6-#2)
                       +(#5*#5+#6*#6)*(#2-#4))/\den}%
    \pgfmathsetmacro{\Uy}{((#1*#1+#2*#2)*(#5-#3)
                       +(#3*#3+#4*#4)*(#1-#5)
                       +(#5*#5+#6*#6)*(#3-#1))/\den}%
  % rayon
    \pgfmathsetmacro{\R}{sqrt((\Ux-#3)^2+(\Uy-#4)^2)}% 
  }

  \computecircum{\xi}{\yi}{\xj}{\yj}{\xk}{\yk}
  \draw[red,dashed] (\Ux,\Uy) circle (\R);
  \computecircum{\xi}{\yi}{\xl}{\yl}{\xk}{\yk}
  \draw[red,dashed] (\Ux,\Uy) circle (\R);
%  \computecircum{\xj}{\yj}{\xl}{\yl}{\xk}{\yk}
%  \draw[blue,dashed] (\Ux,\Uy) circle (\R);
%  \computecircum{\xi}{\yi}{\xl}{\yl}{\xj}{\yj}
%  \draw[blue,dashed] (\Ux,\Uy) circle (\R);

  % quadrilatère
  \draw (pi) -- (pj) -- (pk) -- (pl) -- cycle;
  % diagonale pi-pk
  \draw[red,thick, dashed] (pi) -- (pk);
%  \draw[blue,thick, dashed] (pl) -- (pj);

  % cercles circonscrits

  % points + labels au-dessus
  \foreach \P/\name in {pi/$\bm{p_i}$,pj/$\bm{p_j}$,pk/$\bm{p_k}$,pl/$\bm{p_l}$} {
    \fill (\P) circle (2pt);
    \node[above] at (\P) {\name};
  }
\end{tikzpicture}
&
\begin{tikzpicture}[scale=1]
  % points indépendants

  \pgfmathsetmacro{\xi}{3.0}
  \pgfmathsetmacro{\yi}{0.2}
  \pgfmathsetmacro{\xj}{2.7}
  \pgfmathsetmacro{\yj}{2.6}
  \pgfmathsetmacro{\xk}{0.1}
  \pgfmathsetmacro{\yk}{0.6}
  \pgfmathsetmacro{\xl}{1.9}
  \pgfmathsetmacro{\yl}{-0.1}

  \coordinate (pi) at (\xi,\yi);
  \coordinate (pk) at (\xk,\yk);
  \coordinate (pj) at (\xj,\yj);
  \coordinate (pl) at (\xl,\yl);

  \newcommand{\computecircum}[6]{%
  % dénominateur
    \pgfmathsetmacro{\den}{2*(#1*(#4-#6)+#3*(#6-#2)+#5*(#2-#4))}%
  % centre
    \pgfmathsetmacro{\Ux}{((#1*#1+#2*#2)*(#4-#6)
                       +(#3*#3+#4*#4)*(#6-#2)
                       +(#5*#5+#6*#6)*(#2-#4))/\den}%
    \pgfmathsetmacro{\Uy}{((#1*#1+#2*#2)*(#5-#3)
                       +(#3*#3+#4*#4)*(#1-#5)
                       +(#5*#5+#6*#6)*(#3-#1))/\den}%
  % rayon
    \pgfmathsetmacro{\R}{sqrt((\Ux-#3)^2+(\Uy-#4)^2)}% 
  }

%  \computecircum{\xi}{\yi}{\xj}{\yj}{\xk}{\yk}
%  \draw[red,dashed] (\Ux,\Uy) circle (\R);
%  \computecircum{\xi}{\yi}{\xl}{\yl}{\xk}{\yk}
%  \draw[red,dashed] (\Ux,\Uy) circle (\R);
  \computecircum{\xj}{\yj}{\xl}{\yl}{\xk}{\yk}
  \draw[blue,dashed] (\Ux,\Uy) circle (\R);
  \computecircum{\xi}{\yi}{\xl}{\yl}{\xj}{\yj}
  \draw[blue,dashed] (\Ux,\Uy) circle (\R);

  % quadrilatère
  \draw (pi) -- (pj) -- (pk) -- (pl) -- cycle;
  % diagonale pi-pk
%  \draw[red,thick, dashed] (pi) -- (pk);
  \draw[blue,thick, dashed] (pl) -- (pj);

  % points + labels au-dessus
  \foreach \P/\name in {pi/$\bm{p_i}$,pj/$\bm{p_j}$,pk/$\bm{p_k}$,pl/$\bm{p_l}$} {
    \fill (\P) circle (2pt);
    \node[above] at (\P) {\name};
  }
\end{tikzpicture}
\end{tabular}